\documentclass[12pt]{article}
\usepackage[all]{xy}
\usepackage{amssymb}
\usepackage{amscd}
\usepackage{amsmath} 
\usepackage{amsthm} 
\usepackage{amsfonts} 
\usepackage{graphicx}
\usepackage{latexsym} 
\usepackage{enumerate}  
\usepackage{indentfirst}

\newcommand{\N}{\mathbb{N}}
\newcommand{\Z}{\mathbb{Z}}
\newcommand{\Q}{\mathbb{Q}}

\newcommand{\C}{\mathbb{C}}

\newcommand{\ba}{\mbox{\boldmath $a$}}

\newcommand{\bc}{\mbox{\boldmath $c$}}
\newcommand{\be}{\mbox{\boldmath $e$}}

\newcommand{\bx}{\mbox{\boldmath $x$}}
\newcommand{\by}{\mbox{\boldmath $y$}}

\newcommand{\bv}{\mbox{\boldmath $v$}}
\newcommand{\bu}{\mbox{\boldmath $u$}}
\newcommand{\bw}{\mbox{\boldmath $w$}}

\newcommand{\bphi}{\mbox{\boldmath $\phi$}}
\newcommand{\bbphi}{\mbox{\boldmath $\Phi$}}
\newcommand{\re}{\mathrm{e}}

\newcommand{\bbf}{\mathbb{F}_p}

\newcommand{\btv}{\tilde{\mbox{\boldmath $v$}}}
\newcommand{\btw}{\tilde{\mbox{\boldmath $w$}}}
\newcommand{\bte}{\mbox{\boldmath $e$}}
\newcommand{\hB}{\hat{B}}
\newcommand{\hC}{\hat{C}}
\newcommand{\hlambda}{\hat{\Lambda}}
\newcommand{\hgamma}{\hat{\Gamma}}

\newcommand{\halpha}{\hat{\alpha}}
\newcommand{\hbeta}{\hat{\beta}}
\newcommand{\hhgamma}{ {\hat{\gamma}}}
\newcommand{\rmh}{\mathrm{h}_\eta}

\newcommand{\bhv}{\tilde{\mbox{\boldmath $v$}}}

\newcommand{\cR}{{\cal R}}

\newcommand{\cro}{\mathcal{R}_{\mathrm{odd}}}
\newcommand{\cre}{\mathcal{R}_{\mathrm{even}}}
\newcommand{\lgo}{L_{g\,\mathrm{odd}}}

\newcommand{\rmd}{\mathrm{d}}
\newcommand{\fd}{\mathcal{D}}

\def\t{\noindent}

\theoremstyle{definition}

\newtheorem{thm}{\bf Theorem}[section]
\newtheorem{cor}[thm]{\bf Corollary}
\newtheorem{lem}[thm]{\bf Lemma}
\newtheorem{slem}[thm]{\bf SubLemma}
\newtheorem{prop}[thm]{\bf Proposition}
\newtheorem{defi}[thm]{\bf Definition}
\newtheorem{rem}[thm]{\bf Remark}
\newtheorem{exam}[thm]{\bf Example}

\newtheorem{conv}[thm]{\bf Convention}
\def\t{\noindent}

\newtheorem*{assertion*}{\bf Assertion}
\newtheorem*{slem*}{\bf SubLemma}
\newtheorem*{lem*}{\bf Lemma}
\newtheorem*{nota*}{\bf Notation}
\newtheorem*{conv*}{\bf Convention}

\begin{document}

\title{On some finite dimensional complex representations of mapping class groups and Fox derivation }
\author{Yutaka KANDA}
\date{2019 May 18}
\maketitle

\begin{abstract} We study the finite dimensional complex representations of the mapping class group $\mathcal{M}_{g,1}$ that are derived from some finite Galois coverings of the compact oriented surface with one boundary component $\Sigma_{g,1}$. The key ingredients are Fox derivation, Magnus modules and the Skolem-Noether theorem, which enable us to compute the $\mathcal{M}_{g,1}$-action on the module $L_g^{\eta}$ very explicitly, where $\eta$ is a primitive $p$the root of unity for an odd prime $p$.    
\end{abstract}
\section{\bf Introduction}\label{section 1 intro}
Let $F_m$ be the free group of $\mathrm{rank}\; m$. Suppose that a subgroup $\mathcal{M} \subset \mathrm{Aut}(F_m)$ is given. Let $N \subset F_m$ be a finite index normal subgroup preserved by the $\mathcal{M}$-action. Set the quotient group $G:=F_m/N$ and denote the quotient map by $\rho_N : F_m \twoheadrightarrow G$. Then $N^{\mathit{ab}} \otimes_{\Z} \C$ affords a finite dimensional complex representation of $G \rtimes \mathcal{M}$, where $N^{\mathit{ab}}$ is the abelianization of $N$. Obviously this construction has  great generality and gives rise to a plenty of finite dimensional representations (modules) of the subgroups of $\mathrm{Aut}(F_m)$. It seems important to study their basic properties such as irreducibility, indecomposability, composition series and so on. 
\\
In the setting above, what Magnus showed in \cite{magnus} is that there exists an embedding
\[
F_m/[N,N] \hookrightarrow L_N \rtimes G \,:\, a \,\mathrm{mod}\, [N,N] \mapsto (\fd_{N}(a),\, a \,\mathrm{mod}\, N  ) \quad ( a \in F_n)
\]
, where $\fd_{N} : F_{m} \to L_N$ is a crossed homomorphism and $L_N$ a free left $\Z[G]$-module of $\mathrm{rank}\, m$. We call $\fd_{N}$ {\it Fox derivation} and $L_N$ {\it Magnus module} in the present paper.
\\
From a topological view point, $N^{\mathit{ab}}$ is nothing but $\mathrm{H}_1(F_m ; \Z[G])$ the 1st homology group of the classifying space $BF_m$ with the coefficient $\Z[G]$, which we regard as a right $G$-module. We have a canonical chain complex to compute 
 $\mathrm{H}_*(F_m ; \Z[G])\; (* = 0,1)$. In fact, there exists an exact sequence as follows;
\[
 0  \to N^{\mathit{ab}} \to L_N \to \Z[G] \to \Z \to 0.
\] 
The advantage of the chain complex $0 \to L_N \to \Z[G] \to 0$ is that it affords a canonical $\mathcal{M}$-action, which is very suitable for our purpose.
\\
In our perspective, $L_N$ has a redundancy in the sense that, although we are merely interested in $\mathcal{M}$-modules, it is a $G \rtimes \mathcal{M}$-module.  Occasionally, we can avoid this subsidiary $G$-action by applying the trick described below;  Let $S = \mathrm{e}_{S} \C[G]$ be a simple component of the group algebra $\C[G]$, where $\mathrm{e}_{S} \in S$ is the central idempotent corresponding to $S$. Assume that $S$ is preserved under the $\mathcal{M}$-action. Then we see that $\mathrm{e}_{S} L_{N} \subset L_{N}$ is a $\mathcal{M}$-submodule. The Skolem-Noether theorem implies that there exists a $\C$-vector space $V_{S}$ that affords a projective $\mathcal{M}$-representation such that $V_{S}^* \otimes_{\C} V_{S} \cong S$ as $\tilde{\mathcal{M}}$-module, where $\tilde{\mathcal{M}}$ is a central extension of $\mathcal{M}$. Associated with it, we can construct a $\tilde{\mathcal{M}}$-module $L_{N,S}$ such that $\mathrm{e}_{S} L_{N} \cong V_{S}^{*} \otimes_{\C} L_{N,S}$, where the latter is endowed with diagonal $\tilde{\mathcal{M}}$-action. Notice that $\tilde{\mathcal{M}}$-action on $V_{S}^*$ factors through a central extension of the finite group $\mathrm{Aut}(G)$. It seems that we should study $L_{N,S}$ prior to $\mathrm{e}_{S} N^{\mathit{ab}} \subset \mathrm{e}_{S} L_{N}$ since $L_{N,S}$ must reflect more directly aspects concerning the complexity of the infinite group $\mathcal{M}$.  
\\

Let $\mathcal{M}_{g,1}$ be the mapping class of the compact oriented surface $\Sigma_{g,1} $ of genus $g$ with one boundary component. This consists of the equivalence classes of the  diffeomorphisms of $\Sigma_{g,1}$ that fix $\partial \Sigma_{g,1}$ pointwise, with the equivalence relation being determined by isotopy. Now pick a base point $\star$ on the boundary. Since $\Sigma_{g,1}$ is homotopically equivalent to the bouquet of $2g$ circles, the fundamental group $\pi(\Sigma_{g,1},\star)$ can be identified with the free group $F_{2g}$. Then we see that
\[ \mathcal{M}_{g,1} \cong \big\langle h \in \mathrm{Auto}(F_n) \mid h(s) = s \big\rangle \]
, where $s \in \pi(\Sigma_{g,1},\star)$ is the homotopy class of the boundary loop based at $\star$ which travels in the direction compatible with the chosen orientation of $\Sigma_{g,1}$. It is known that the braid group $\mathrm{Br}_{2g+1}$ is embedded into $\mathcal{M}_{g,1}$, which by restriction gives rise to a functor from the category of $\mathcal{M}_{g,1}$-modules to that of $\mathrm{Br}_{2g+1}$-modules. In the present paper, we will utilize this functor to study some finite dimensional complex representations of $\mathcal{M}_{g,1}$, the construction of which are given by the above framework applied to the $\mathcal{M}_{g,1}$-equivariant surjection $F_{2g} \twoheadrightarrow \mathrm{H}(2g,p)$. Here $\mathrm{H}(2g,p)$ is a central extension by the cyclic group of order $p$ of the elementary $p$-abelian group of order $p^{2g}$.  
\\

The construction of the present paper is as follows;
\\
In Section 2, we provide basic facts on Fox derivation and Magnus modules as preliminaries. In Section 3, we give the general description of the $\mathrm{Br}_{2g+1}$-action on Magnus modules. In Section 4, we enter a concrete calculation. We focus on the case where $G=\mathrm{H}(2g,p)$, introduce the $\mathrm{Br}_{2g+1}$-module $L_g^{\eta}$ and give a detailed description of the $\mathrm{Br}_{2g+1}$-action. In Section 5, we prove the main theorem, which is related to the size of the part the braid group action occupies in the endomorphism algebra of the module $L_g^{\eta}$.   
\\


\section{\bf Fox derivation and Magnus embedding}\label{section 2}

\subsection{Fox derivation and automorphisms of free groups}
We will provide preliminary results concerning Fox derivation and Magnus modules.
In this subsection, we often omit the proof since it is easy or basic. See \cite{fox} or \cite{sokolov} for the detail.\\
Let $F_n:=\langle \bx_1,\bx_2,\dots,\bx_n \rangle$ be the free group on the $n$ letters $\bx_1,\bx_2,\dots,\bx_n$. For the moment, we will fix a quotient group $G$ of $F_{n}$. Denote by $\rho_{N} : F_n \twoheadrightarrow G$ the quotient map, where $N$ stands for the kernel of this map.
 
\begin{defi} Suppose $M$ to be a left $G$-module. A map $\rmd :F_n \to M$ is a $G$-{\sl crossed homomorphism} if
\[
\rmd(ab) = \rmd(a) + a\rmd(b)\quad (a\,b \in F_n)
\]
, where we regard $M$ as a left $F_n$-module via $\rho_{N}$.
\end{defi}
\begin{lem}\label{crossed hom functorial}
If $f: M \to M^{\prime}$ is a left $G$-module homomorphism and if $\rmd :F_n \to M$ is a $G$-crossed homomorphism, then $f\circ \rho$ is a $G$-crossed homomorphism.
\end{lem}
\begin{lem}
For any $G$-crossed homomorphism $\rmd :F_n \to M$, it holds that
\[
\rmd(1) =1,\; d(a^{-1}) =-a^{-1} \rmd(a) \quad (a \in F_n). 
\]
\end{lem}
\begin{lem}\label{crossed hom coincide} If two $G$-crossed homomorphisms $\rmd, \rmd^{\prime} :F_n \to M$ coincide on the generator $\{\bx_1,\dots,\bx_n\} \subset F_n$, then $\rmd \equiv \rmd^{\prime}$. 
\end{lem}
\begin{proof} (Sketch) This follows from an inductive argument with respect to the word length $|\cdot| : F_n \to \N \cup \{0\}$.
\end{proof}
\begin{thm}\label{crossed hom existance} For any left $G$-module $M$ and any $n$ elements $m_1,\dots,m_n\in M$, there exists a unique $G$-crossed homomorphism $\rmd :F_n \to M$ such that $\rmd(\bx_i)=m_i \; (1\leq i \leq n)$.
\end{thm}
\begin{proof} (Sketch) Denote by $W_n$ the set of words on the $n$ letters $\bx_1,\dots \bx_n$ and $\pi : W_n \twoheadrightarrow F_n = W_n/\sim$ the quotient map. Construct $\tilde{\rmd} : W_n \to M$ inductively with respect to the word length so that it satisfies the condition
\[
\tilde{\rmd}(w_1 w_2) = \tilde{\rmd}(w_1) + \rho_N(\pi(w_1)) \tilde{\rmd}(w_2) \quad (w_1,\,w_2 \in W_n).
\]
Check that $\tilde{\rmd}$ factors through $F_n$.
\end{proof}
\begin{exam} When we regard $\Z[G]$ as a left $G$-module, we have a unique $G$-crossed homomorphism $\rmd_{\rho_{N}}$ determined by $\rmd_{\rho_{N}}(\bx_i):= \rho_{N}(\bx_i)-1 \; (1\leq i\leq n)$. Then it holds that $\rmd_{\rho_{N}}(a) = \rho_{N}(a)-1 \; (a \in F_n)$.
\end{exam}
For any $G$-crossed homomorphism $\rmd:F_n \to M$, set $\rmd_{\Z} : \Z[F_n] \to M$ to be its $\Z$-linear extension.  Regarding $M$ as a left $\Z[F_n]$-module, it is readily seen that
\[
\rmd_{\Z}(ab)=\rmd_{\Z}(a) \epsilon(\rho_{N}(b)) + a\rmd_{\Z}(b) \quad (a,b \in F_n)
\]
, where $\epsilon: \Z[G] \to \Z$ is the augmentation map, which is the unique $\Z$-algebra homomorphism determined by the condition $\epsilon(g) =1$ for any $g\in G$.
\begin{lem}\label{crossed hom group hom} For any $G$-crossed homomorphism $\rmd :F_n \to M$, the restriction $\rmd|_{N} : N \to M$ is a group homomorphism. Recall that $G = F_{n}/N$.
\end{lem}
\begin{prop}\label{crossed hom group hom2} For any $G$-crossed homomorphism $\rmd :F_n \to M$, it holds that
$[N,N] \subset \mathrm{Ker} (\rmd)$. Recall that $G = F_{n}/N$.
\end{prop}
%
\begin{defi}[Fox derivation and Magnus module \,\cite{sokolov} ]\label{fox deri} Consider a normal subgroup $N \triangleleft F_n$. Let $\rho_{N} : F_n \twoheadrightarrow G:=F_n\slash N$ be the quotient map. Define the free left $\Z[G]$-module $L_{N}$ as 
\[
L_{N} := \bigoplus_{1\leq i\leq n} \Z[G] \re_i
\]
, where $\{ \re_i \mid 1\leq i\leq n \}$ is the free basis. We refer to it as the {\it Magnus module} of $N$. Define the $G$-crossed homomorphism $\fd_{N}: F_n \to L_{N}$ to be the one uniquely determined by the condition
\[
\fd_{N}(\bx_i) = \re_i \quad (1\leq i\leq n) .
\] 
Theorem \ref{crossed hom existance} ensures the existence of $\fd_{N}$ and Lemma \ref{crossed hom coincide} does the uniqueness. We refer to $\fd_{N}$ as the {\it Fox derivation} with respect to $N$. Further, we will define the $\Z[G]$-module homomorphism $\partial_{N} : L_{N} \to \Z[G]$ by the condition
\[
\partial_{N}(\re_i) = \rho_{N}(\bx_i) -1 \quad (1\leq i\leq n) .
\]
\end{defi}
\begin{prop}\label{partial d equality} For any $a \in F_n$, it holds that
\[
\partial_{N}\circ \fd_{N}(a) = \rho_{N}(a) -1 .
\]
\end{prop}
\begin{proof} Lemma \ref{crossed hom functorial} implies that $\partial_{N}\circ \fd_{N}$ is a $G$-crossed homomorphism. On the other hand, it is readily seen that the map $a \mapsto \rho_{N}(a)-1\; (a \in F_n)$ is a $G$-crossed homomorphism. By definition, these two coincide with each other on the subset $\{\bx_1,\dots,\bx_n\} \subset F_n$. Therefore, they are identical to each other by Lemma \ref{crossed hom coincide}.
\end{proof}
\begin{cor} As a corollary of the previous $2$ propositions, we see that
\[
[N,N]\subset \mathrm{Ker}(\fd_{N}), \quad  \fd_{N}(N) \subset \mathrm{Ker}(\partial_{N}) .
\]
\end{cor}
Actually, these $2$ inclusions are both equalities.
\begin{thm}[Magnus's Theorem\,\cite{magnus}. See also \cite{sokolov}]\label{magnus thm}
\[
[N,N]= \mathrm{Ker}(\fd_{N}), \quad \fd_{N}(N)= \mathrm{Ker}(\partial_{N}) .
\]
\end{thm}
\begin{prop}\label{funda sequences} We have the following exact sequences;
\begin{align*}
1 \to & N^{\mathit{ab}} \rightarrow L_{N} \overset{\partial_{N}}{\longrightarrow} \Z[G] \overset{\epsilon}{\longrightarrow} \Z \to 0,
\qquad
 \Z[F_n] \overset{(\fd_{N})_{\Z}}{\longrightarrow} L_{N} \to 0.
\end{align*}
We often omit the subscript "$\Z$" and denote $(\fd_{N})_{\Z}$ simply by $\fd_{N}$ when there is no fear of confusion.
\end{prop}
\begin{proof} First we will treat the 1st sequence. \\
The exactness at the 2nd and the 3rd terms are ensured by Theorem \ref{magnus thm}.  As for the 4th term, notice that $\mathrm{Ker}(\epsilon) = \langle  g-1 \mid g \in G \rangle_{\Z}$. Thus the construction of $\partial_{N}$  immediately implies that $\mathrm{Image}(\partial_{N}) \subset \mathrm{Ker}(\epsilon)$. On the other hand, Proposition \ref{partial d equality} implies $\mathrm{Image}(\partial_{N} \circ \fd_{N}) \supset \mathrm{Ker}(\epsilon)$. The proof for the 5th term is obvious. Thus we are done.\\
Secondly we will treat the 2nd sequence.\\
The argument just before ensures that, for any $m \in L_{N}$, there exists some $a \in \Z[F_N]$ such that $\partial_{N}\circ \fd_{N}(a) = \partial_{N}(m) \in \mathrm{Ker}(\epsilon)$. Then $m-\fd_{N}(a) \in \mathrm{Ker}(\partial_{N})$. The 2nd assertion of Theorem \ref{magnus thm} implies that there exists some $n \in N$ such that $m-\fd_{N}(a) = \fd_{N}(n)$, which implies that $m=\fd_{N}(n+a)$. Thus we have shown the surjectivity of $\fd_{N}$. 
\end{proof}
%
%
%
\begin{defi}[Fundamental sequence, Magnus embedding]\label{funda sequence defi} The 1st sequence in Proposition \ref{funda sequences} is refered to as the {\it fundamental sequence} associated with $N \triangleleft F_n$. Its 2nd arrow $N^{\mathit{ab}} \hookrightarrow L_{N}$ are refered to as the {\it Magnus embedding} associated with $\fd_{N}$. 
\end{defi}
\begin{rem} Notice that $N^{\mathit{ab}}$ is free abelian since $N \subset F_n$ is a free group. The adjoint action of $F_n$ on $N$ induces the action of $G=F_n \slash N$ on $N^{\mathit{ab}}$. Thus $N^{\mathit{ab}}$ amounts to a $\Z[G]$-module, which is free over $\Z$. Actually the embedding $N^{\mathit{ab}} \hookrightarrow L_{N}$ is a $\Z[G]$-module homomorphism. See Remark \ref{magnus emb module hom} below for the detail.
\end{rem}

$N^{\mathit{ab}}$ has an important meaning in terms of group homology;
\begin{prop}
\[
N^{\mathit{ab}}=\mathrm{Ker}(\partial_{N}) \cong \mathrm{H}_1(F_n : \Z[G])
\]
\end{prop}
, where the coefficient $\Z[G]$ of the R.H.S. is regarded as a right $F_n$ module through $\rho_{N}$, which gives a local system on the classifying space $BF_n $.
In fact, the sequence
\[
0 \to L_{N} \overset{\partial_{N}}{\longrightarrow} \Z[G] \to 0
\]
is a chain complex to compute the homology group of $F_n$ with coefficient $\Z[G]$. The left $G$-action on the coefficient $\Z[G]$ induces the left $G$-action on the homology groups.
\\

The Magnus module $L_N$ has some sort of universal property in the sense stated below;
\begin{thm}[Universality of Magnus module]\label{fox univ}
Suppose $M$ to be a left $G$-module and $\rmd : F_n \to M$ a $G$-crossed homomorphism. Then there exist a unique $\Z[G]$-homomorphism $h : L_N \to M$ such that $\rmd=h\circ \fd_{N}$. 
\end{thm}
\begin{proof}
\textbullet\; Uniqueness.) The factorization $\rmd_{\Z} = h \circ (\fd_{N})_{\Z}$ together with  the surjectivity of $(\fd_{N})_{\Z}$ implies the uniqueness of $h$.
\\
\textbullet\; Existence.) Since $L_N$ is a free $\Z[G]$-module on the basis $\{ \re_1,\dots, \re_n\}$, the map $\re_i \mapsto \rmd(x_i) ,\; (1\leq i \leq n)$ extends uniquely to a $\Z[G]$-module homomorphism $h: L_N \to M$. Then $h \circ \fd_{N}$ is a $G$-crossed homomorphism by Lemma \ref{crossed hom functorial}. On the other hand, Lemma \ref{crossed hom coincide} implies that $\rmd = h \circ \fd_{N}$.
\end{proof}
Now, we will make use of Magnus embeddings for the investigation of $\mathrm{Aut}(F_n)$. 
\\
Suppose that a subgroup $\mathcal{M} \subset \mathrm{Aut}(F_n)$ is given. Consider a normal subgroup $N \triangleleft F_n$. Set $G :=F_n\slash N$ the quotient group.
\begin{defi}[Characteristic subgroup]
The normal subgroup $N \triangleleft F_n$ is $\mathcal{M}$-{\sl characteristic} if $\phi(N)=N$ for any $\phi \in \mathcal{M}$.
\end{defi}
\begin{nota*}
For any $\phi \in \mathrm{End}(F_n)$ such that $\phi(N) \subset N$, denote by $\overline{\phi} \in \mathrm{End}(G)$ the induced homomorphism.
\end{nota*}
%
\begin{prop}
Suppose that $\phi \in \mathrm{End}(F_n)$ satisfys the condition $\phi(N)\subset N$. Then there exists a unique $\hat{\phi} \in \mathrm{End}_{\Z}(L_N)$ that satisfies the relation $\fd_{N} \circ \phi = \hat{\phi} \circ \fd_{N}$. Further, it holds that
\[
\hat{\phi}(gm) = \overline{\phi}(g)\hat{\phi}(m) \quad (g\in G,\; m\in M) .
\]
\end{prop}
\begin{proof}
Define the $\Z[G]$-module $L_{N}^{\overline{\phi}}$ as follows;
\[
L_{N}^{\overline{\phi}} \overset{\iota}{\equiv} L_{N} \text{ as }\Z\text{-module}.
\]
For any $m \in M$, we denote by $m_{\overline{\phi}}$ the corresponding element of $L_{N}^{\overline{\phi}}$ under this identification. We define the $G$-action on $L_{N}^{\overline{\phi}}$ by
\[ 
g m_{\overline{\phi}} := \overline{\phi}(g)m \quad (g\in G, \; m\in M) .
\]
Then $\iota^{-1} \circ \fd_{N} \circ \phi : F_N \to L_{N}^{\overline{\phi}}$ is a $G$-crossed homomorphism. Theorem \ref{fox univ} implies that there exists a unique $\Z[G]$-module homomorphism $\tilde{\phi} : L_N\to L_{N}^{\overline{\phi}}$ such that $\iota^{-1} \circ \fd_{N} \circ \phi = \tilde{\phi} \circ \fd_{N}$. Thus $\hat{\phi} := \iota \circ \tilde{\phi}$ is the desired endomorphism.
\end{proof}
\begin{thm}[Induced $G \rtimes \mathcal{M}$-action on Magnus module]\label{fox action thm} Let $\mathcal{M} \subset \mathrm{Aut}(F_n)$ be a subgroup. Assume that $N \triangleleft F_N$ is $\mathcal{M}$-characteristic. Then there exists a uniquely determined action of $\mathcal{M}$ on $L_N$ such that the Fox derivation $\fd_{N} : F_n \to L_N$ is $\mathcal{M}$-equivariant. Further, the structure map $G \times L_N \to L_N$ that determines the left action of $G$ on $L_{N}$ is $\mathcal{M}$-equivariant. In other words, $L_{N}$ is a left $G \rtimes \mathcal{M}$-module by the rule
\[
(g, f)\cdot m := g \cdot f(m) \quad (\, (g,f) \in G \ltimes \mathcal{M}, m \in L_{N}) .
\] 
\end{thm}
%
%
\subsection{ A separation criterion }
Suppose that a quotient homomorphism $\rho_{N}: F_n \twoheadrightarrow G$ is given. Complexifying the fundamental sequence in Proposition \ref{funda sequences}, we obtain 
\[
\C[F_n] \overset{(\fd_{N})_{\C}}{\longrightarrow} (L_N)_{\C} \overset{(\partial_N)_{\C}}{\longrightarrow} \C[G] \overset{\epsilon_{\C}}{\longrightarrow} \C \to 0.
\]
\begin{conv} From now on, unless it may cause a confusion, we omit the subscript "$\C$" for the simplicity of notation.
\end{conv}
For any element $h \in F_n$, let $\mathrm{adj}(h)$ be the conjugation action on $F_n$ determined by $h$, that is, $\mathrm{adj}(h)(k) = hkh^{-1}\; (k \in F_n)$. We will consider the induced endomorphism $\hat{\mathrm{adj}(h)} \in \mathrm{End}_{\C}( L_N )$.
\begin{prop}[Adjoint action formula]\label{fox deri formula}
\[
\hat{\mathrm{adj}(h)}(m) = -\rho_{N}(h) \partial_N(m) \rho_{N}(h)^{-1}\fd_{N}(h) + \rho_{N}(h)m \quad (m \in L_N).
\]
\end{prop}  
\begin{proof}
Since $\fd_{N} : \C[F_n] \to L_N$ is surjective, there exists an $a \in \C[F_n]$ such that $m= \fd_{N}(a)$. Then
\begin{align*}
\hat{\mathrm{adj}(h)}(m) &= \fd_{N}( \mathrm{adj}(h)(a) ) = \fd_{N}( hah^{-1})
\\
& = \fd_{N}(h) \epsilon(ah^{-1}) + \rho_{N}(h) \fd_{N}(a) \epsilon(h^{-1}) - \rho_{N}(hah^{-1}) \fd_{N}(h) 
\\
& = \rho_{N}(h) \{ \epsilon(a) -\rho_{N}(a) \} \rho_{N}(h)^{-1} \fd_{N}(h) + \rho_{N}(h) \fd_{N}(a) 
\\
&= R.H.S.
\end{align*}
\end{proof}

\vspace{-1cm}

\begin{cor}\label{Adjoint action formula cor} We see that
\[
\hat{\mathrm{adj}(h)} |_{\mathrm{Ker}((\partial_N)_{\C})} = \mathrm{L}_{\rho_{N}(h)} |_{\mathrm{Ker}((\partial_N)_{\C})}
\]
, where $\mathrm{L}_{g} \in \mathrm{End}_{\C}(L_N)\;(g \in G)$ denotes the left action of $g$ on $L_N$.
\end{cor}
\begin{defi} Suppose that $h \in F_n$ satisfies the condition $\rho_{N}(h) \in \mathrm{Z}(G) \cdots (\sharp)$. Set
\[
B_h := \hat{\mathrm{adj}(h)} -1 \in \mathrm{End}_{\C}(L_N).
\]
\end{defi}
\begin{prop}\label{Bh formula} For any $h \in F_n$ that satisfies the condition ($\sharp$), $B_h$ belongs to $\mathrm{End}_{\C[G]}(L_N)$. Further, it holds that
\[
B_h \circ B_h = (\mathrm{L}_{\rho_{N}(h)} -1) \circ B_h .
\]
\end{prop}
\begin{lem}\label{fox deri formula cor} As a direct corollary of Proposition \ref{fox deri formula}, if $h \in F_n$ satisfies the condition ($\sharp$), then we see that
\[
\hat{\mathrm{adj}(h)}(m) = -(\partial_N)(m)\fd_{N}(h) + \rho_{N}(h)m \quad (m \in L_N).
\]
\end{lem}
\t
{\bf Proof of Proposition \ref{Bh formula}} \; Since $h^2$ satisfies the condition ($\sharp$), it follows from the previous lemma that
\begin{align*}
( \hat{\mathrm{adj}(h)} )^{2} (m) &=  \hat{\mathrm{adj}(h^{2})} = -\partial_{N} (m)\fd_{N}(h^2) + \rho_{N}(h^2)m
\\
&= -\partial_{N}(m) (1 + \rho_{N}(h) ) \fd_{N}(h) + \rho_{N}(h)^{2} m
\\
&= -(1 + \rho_{N}(h) ) \partial_{N} (m) \fd_{N}(h) + \rho_{N}(h)^{2} m
\end{align*}
Therefore, the 2nd assertion follows from the calculation below;
\begin{align*}
( \hat{\mathrm{adj}(h)} -1)^2 (m) &= ( 1-\rho_{N}(h)) \partial_{N}(m) \fd_{N}(h)  + (\rho_{N}(h) -1)^2 m
\\
&= (\rho_{N}(h)-1) \{ -\partial_N (m) \fd_{N}(h) + \rho_{N}(h)m - m \}
\\
&= (\rho_{N}(h)-1) \{ \hat{\mathrm{adj}(h)} -1 \} (m) .
\end{align*}
The 1st assertion follows since, for any $g\in G$, we see that
\begin{align*}
\hat{\mathrm{adj}(h)}(gm) &= - \partial_{N}(gm) + \rho_{N}(h)gm
\\
&= - g \partial_{N}(m) + g\rho_{N}(h)m
\\
&= g\, \hat{\mathrm{adj}(h)}(m) .
\end{align*}

\rightline{$\Box$}

\vspace{-0.5cm}

\begin{prop}\label{Bh image} Suppose that $h \in F_n$ satisfies the condition $(\sharp)$. Then
\[
B_h( L_N) \subset \mathrm{Ker}\partial_N .
\]
\end{prop}
\begin{proof}
For any $a \in \Z[F_n]$, we see that
\begin{align*}
\partial_N \circ B_h(\fd_{N}(a)) &= \partial_{N} \circ \fd_{N} ( (\mathrm{adj}_{h} -1)(a) ) \;=\; \partial_{N}(hah^{-1}) - \partial_{N}(a) 
\\
&= \partial_{N}(a) - \partial_{N}(a) = 0 .
\end{align*}
\end{proof}

\vspace{-0.2cm}

From now on, we will impose the assumption that $h \in F_n$ satisfies the condition $(\sharp)$ and that $\rho_{N}(h) \in G$ is of finite order. The latter ensures that the left multiplication by $\rho_{N}(h)$ determines a semi-simple endomorphism. Denoted by $L_{\lambda}\; (\lambda \in \C)$ the eigenspace of $\mathrm{L}_{\rho_{N}(h)} \in \mathrm{End}_{\C[G]}(L_N)$ with eigenvalue $\lambda$.  Notice that $L_{\lambda}$ is a $\C[G]$-submodule. Similarly, denote by $M_{\zeta}\; (\zeta \in \C) $ the eigenspace of the (left) multiplication by $\rho_{N}(h)$ on $\C[G]$ with eigenvalue $\zeta$. Then $M_{\zeta}$ is an ideal of $\C[G]$.
%
\begin{prop}[Separation criterion]\label{fox decomp} If $\lambda\neq 1$, then we have the following direct sum decomposition as $\C[G]$-module;
\[
L_{\lambda} = \mathrm{Ker} ( \partial_N|_{L_{\lambda}} )\oplus M_{\lambda} \fd_{N}(h) .
\]
The 2nd summand of the R.H.S. is a free $M_{\lambda}$-module with the free basis $\{ 1_{M_{\lambda}} \!\cdot\! \fd_{N}(h) \}$, where $1_{M_{\lambda}}$ is the identity element of $M_{\lambda}$. 
\end{prop}
\begin{proof}
Set $A:= B_h|_{ L_{\lambda} },\; \partial_{\lambda} := \partial_N|_{L_{\lambda}}$. Then $A^2-(\lambda-1)A=0$. Since $\lambda-1\neq 0$, we have a $\C[G]$-module decomposition
\[
L_{\lambda} = \mathrm{Ker}(A-\lambda+1) \oplus \mathrm{Ker}(A) .
\]
Then Proposition \ref{Bh image} implies that
\[
\mathrm{Ker}(A-\lambda+1) = \mathrm{Image}(A) \subset \mathrm{Ker}(\partial_{\lambda}) .
\]
On the other hand, Lemma \ref{fox deri formula cor} implies, for any $m \in L_{\lambda}$, that
\begin{align*}
(A-\lambda+1)(m) &= -\partial_{\lambda}(m)\fd_{N}(h) + ( (\mathrm{L}_h-1) -\lambda+1 )(m)
\\
&=-\partial_{\lambda}(m)\fd_{N}(h) .
\end{align*}
Further, the exact sequence
\[
(L_N)_{\C} \overset{(\partial_N)_{\C}}{\longrightarrow} \C[G] \overset{\epsilon_{\C}}{\longrightarrow} \C
\]
induces the exact sequence
\[
L_{\lambda} \overset{\partial_{\lambda} }{\longrightarrow} M_{\lambda}  \to 0
\]
, which shows that $\partial_{\lambda}(L_{\lambda})=M_{\lambda}$. These two together imply that
\begin{align*}
\mathrm{Ker}(A) &= \mathrm{Image}(A-\lambda+1)
\\ 
&= \partial_{\lambda}(L_{\lambda})\fd_{N}(h) = M_{\lambda}\fd_{N}(h) .
\end{align*}
For any $a \in M_{\lambda}$, we see that
\[
\partial_{\lambda}(a\fd_{N}(h)) = a(\rho_{N}(h)-1) = a(\lambda-1)
\]
, which shows that $M_{\lambda} \fd_{N}(h)$ is free with the free basis $\{ 1_{M_{\lambda}} \!\cdot\! \fd_{N}(h) \}$ and that $\partial_{\lambda}|_{M_{\lambda}\fd_{N}(h)}$ is injective. It follows that $\mathrm{Ker}(\partial_{\lambda}) \cap \mathrm{Ker}(A) = \{0\}$, which in turn implies that \\
$\mathrm{Ker}(A-\lambda+1) = \mathrm{Ker}(\partial_{\lambda})$. Thus we are done.
\end{proof}
\begin{rem}\label{decomp invariant} If $\phi \in \mathrm{Aut}(F_n)$ preserevs both $N$ and $h \in F_n$, then the direct sum decomposition in Proposition \ref{fox decomp} is preserved by $\hat{\phi}|_{L_{\lambda}} \in \mathrm{End}_{\C}(L_{\lambda})$.  
\end{rem}
In the end of this subsection, we add some application of Proposition \ref{fox deri formula}.
\begin{prop}\label{fox submodule} For any normal subgroup $K \triangleleft F_n$ such that $K \subset N$, the image $\fd_{N}(K) \subset L_{N}$ is a $\Z[G]$-submodule. Recall that $N \triangleleft F_n$ is such that $G=F_n/N$.
\end{prop}
\begin{proof} Lemma \ref{crossed hom group hom} implies that $\fd_{N}(K)$ is a $\Z$-submodule. Further, $\fd_{N}(K) \subset \fd_{N}(N) = \mathrm{Ker}(\partial_{N})$. Then Corollary \ref{Adjoint action formula cor} together with the assumption $K \triangleleft F_n$ implies that $\fd_{N}(K)$ is preserved by $\mathrm{L}_{\rho_{N}(a)}$ for any $a \in F_n$. Thus we are done.
%
%
%
\end{proof} 
\begin{rem}\label{magnus emb module hom} The reasoning above also shows that the Magnus embedding $N^{\mathit{ab}} \hookrightarrow L_{N}$ is a $\Z[G]$-module homomorphism.
\end{rem}

%
%
%
%
\subsection{Braid Group Action}\label{braid group action}

Henceforth, we fix positive integer $g$. 
\\
Let $D$ be a $2$-dimensional closed disc with a base point $* \in \partial D$. Orient $D$ so that the induced orientation to $\partial D$ is clockwise. Let $l$ be the simple closed curve based at $*$ that travels on $\partial D$ clockwise.
Let the subset $\mathcal{P} \subset D$ comprise $2g+1$ distinct interior points $ p_1,\dots,p_{2g+1} $. Let $s_i \;(1\leq i \leq 2g+1)$ be oriented simple closed curve in $D \backslash \mathcal{P}$ based at $*$ such that
\begin{itemize}
\item $s_i \cap \partial D = \{*\}$ for $1\leq i \leq 2g+1$,
\item $s_i \cap s_j =\{*\}$ for $(1 \leq i < j \leq 2g+1)$,
\item each $s_i$ encloses $p_i$ clockwise. The open $2$-disc in $D$ that bounds $s_i$ does not intersects $\mathcal{P} \backslash \{p_i\}$,
\item $s_1 \cdot s_2 \cdot \cdots \cdot s_{2g+1}$ is homotopic to $l$ as a based loop in $D\backslash \mathcal{P}$.
\end{itemize}
Set $\by_i := [s_i] \in \pi_1(\partial D\backslash \mathcal{P}; *) \; (1\leq i \leq 2g+1)$. Since $D\backslash \mathcal{P}$ is homotopically equivalent to the bouquet of $2g+1$ circles, the fundamental group $\pi_1(D\backslash \mathcal{P}; *)$ is freely generated by $\{\by_1,\dots, \by_{2g+1} \}$.
\begin{defi} $T_{2g+1} := \langle \by_1,\dots, \by_{2g+1} \rangle \cong \pi_1(D\backslash \mathcal{P}; *)$.
\end{defi}
The total space of the double branched covering over $D$ that branches at $\mathcal{P}$ is homeomorphic to $\Sigma_{g,1}$, the compact oriented surface of genus $g$ with one boundary component. Thus we denote by $\mathrm{pr} : \Sigma_{g,1} \twoheadrightarrow D$ the covering map and $\tau$ the generator of the covering transformation group. Choose a base point $\star \in \mathrm{pr}^{-1}(*) \subset \partial\Sigma_{g,1}$.  Let $\hat{D}$ be the orbifold determined by the action of $\langle \tau \rangle \cong \{ \pm 1 \}$ on $\Sigma_{g,1}$. The orbifold fundamental group $\pi_{1}^{\mathit{orb}}(\hat{D};*)$ is the quotient of $T_{2g+1}$ by the normal closure of $\{ \by_1^2,\dots, \by_{2g+1}^2 \} \subset T_{2g+1}$.
\begin{defi} $\mathrm{H}_{2g+1} := \langle \by_1,\dots, \by_{2g+1} \mid \by_1^2,\dots, \by_{2g+1}^2 \rangle \cong \pi_{1}^{\mathit{orb}}(\hat{D};*)$.
\end{defi}
We have the exact sequence
\[
1 \to \pi_{1}(\Sigma_{g,1}; \star) \overset{ \mathit{pr}_{*} }{\longrightarrow} \pi_{1}^{\mathit{orb}}(\hat{D};*) \overset{ \mu }{\longrightarrow} \langle \tau \rangle \to 1.
\]
, where $\mu$ is determined by $\mu(\by_i) = \tau \; (1\leq i \leq 2g+1)$. Set $\bx_i:= \by_i \by_{i+1} \; (1\leq i \leq 2g)$. Then it is readily seen (e.g. by the Artin-Schreier theory) that $\pi_{1}(\Sigma_{g,1}; \star) \cong \mathrm{Ker}(\mu) $ is a free group freely generated by $\{ \bx_1,\dots,\bx_{2g} \}$.  Henceforth, by abuse of notation, we denote this subgroup by $F_{2g}$. Denote by $\Delta$ the element $\by_{1} \by_{2} \cdots \by_{2g+1} \in H_{2g+1}$, which is represented by the boundary loop of $\hat{D}$ based at $*$ with the clockwise orientation. Since $\mu(\Delta^2)=1$, $\Delta^2 \in F_{2g}$, which is represented by the boundary loop of $\Sigma_{g,1}$ based at $\star$ with the clockwise orientation. 
\begin{lem} A short calculation shows that
\[ \Delta^2 = \bx_1 \bx_3 \dots \bx_{2g-1} \cdot (\bx_1 \bx_2 \dots \bx_{2g})^{-1} \cdot \bx_2 \bx_4 \dots \bx_{2g} .
\]
\end{lem}
\begin{lem}\label{intersection pairing}
The intersection form on $\mathrm{H}_1(\Sigma_{g,1};\Z) \cong \langle [\bx_1], \dots , [\bx_g] \rangle_{\Z}$ is described as follows;
\[
[\bx_i]\cdot [\bx_j] = 
\begin{cases}
j-i & \text{ if } |i-j| = 1, \\
0 & \text{ if otherwise}.
\end{cases}
\]
\end{lem}
\begin{proof} This can be shown e.g.\ by counting algebraic intersection number of the oriented loops on $\Sigma_{g,1}$ based at $\star$ that are the lifts of the loops $s_i s_{i+1} \; (1\leq i \leq 2g)$ in $D\setminus \mathcal{P}$.
\end{proof}
\begin{rem} Set $\ba_j \in F_{2g} \; (1\leq j \leq 2g)$ as follows;
\begin{align*}
\ba_1 &:= \bx_1, \quad \ba_{2i+1} := (\bx_2 \bx_4 \dots \bx_{2i})^{-1}\cdot \bx_1\bx_2\dots \bx_{2i+1} \quad (2\leq i \leq g),
\\
\ba_{2i} &:=\bx_{2i}^{-1} \quad  (1\leq i \leq g) .
\end{align*}
Then it is readily seen that 
\[ 
\Delta^2 = [\ba_1,\ba_2] [\ba_3,\ba_4] \dots [\ba_{2g-1},\ba_{2g}]
\]
, which implies that $\{ \ba_j \}_{1\leq j \leq 2g}$ amounts to the generator of the standard presentation for the fundamental group of the closed orientable surface of genus $g$. Note that we adopt here the notation that the group bracket $[a,b] =aba^{-1}b^{-1}$. We can deduce Lemma \ref{intersection pairing} also as a corollary of this remark.
\end{rem}
Now, we will introduce the braid group $\mathrm{Br}_{2g+1}$ of $2g+1$ strings.
\begin{defi}
\[
\mathrm{Br}_{2g+1} := \big\{ \phi \in \mathrm{Aut}(T_{2g+1}) \;\bigm| \; \phi(\Delta) = \Delta \big\}
\]
\end{defi}
It is well-known that $\mathrm{Br}_{2g+1}$ has the following presentation;
\[
\mathrm{Br}_{2g+1} = \big\langle \sigma_i \; (1\leq i\leq 2g) \;\bigm| \; \sigma_{i}\sigma_{i+1}\sigma_{i} = \sigma_{i+1}\sigma_{i}\sigma_{i+1} \; (1\leq i\leq 2g-1) , \; \sigma_{i}\sigma_{j} = \sigma_{j}\sigma_{i} \; ( |i-j|>1) \,\big\rangle .
\]
\begin{prop}\label{braid action on T} The action of $\mathrm{Br}_{2g+1}$ on $T_{2g+1}$ is described as follows;
\[
\sigma_i(\by_j) = 
\begin{cases}
 \by_{i+1} & \text{ if } j = i,
 \\
 \by_{i+1}\by_{i}\by_{i+1} & \text{ if } j =i+1, 
 \\
 \by_{j} & \text{ if otherwise}.
 \end{cases}
\]
\end{prop}
It is readily seen that $ \sigma_i \; (1\leq i\leq 2g)$ preserves the normal closure $I_{2g+1}:= \mathrm{N}_{T_{2g+1}}(\{ \by_1^2, \dots,\by_{2g+1}^2 \})$. Thus we have the induced $\mathrm{Br}_{2g+1}$-action on the quotient $H_{2g+1} = T_{2g+1}/I_{2g+1}$. Further, the normal subgroup $F_{2g} = \langle \bx_1, \dots, \bx_{2g} \rangle \,\triangleleft\, H_{2g+1}$ is preserved by the induced action. Restricting the $\mathrm{Br}_{2g+1}$-action to $F_{2g}$ gives rise to the group homomorphism
\[
\lambda : \mathrm{Br}_{2g+1} \longrightarrow \mathcal{M}_{g,1} \cong \{ \phi \in \mathrm{Auto}(F_n) \mid \phi(\Delta^2) = \Delta^2 \}
\]
\begin{rem}
It is known that $\lambda$ is injective, which is a consequence of the Birman-Hilden theorem. See \cite{margalit winarski} and the reference in there. 
\end{rem}
%
%
%
%
\subsection{An extension to $H_{2g+1}$ of the Fox derivation $\fd_{N}$ }

Taking over the setting of the previous subsection, we have the $\mathrm{Br}_{2g+1}$-equivariant exact sequence
\[
1 \to F_{2g} \overset{\iota}{\longrightarrow} H_{2g+1} \overset{\mu}{\longrightarrow} \langle \tau \mid \tau^2 \rangle \to 1 . 
\]
\quad Suppose that a $\mathcal{M}_{g,1}$-characteristic subgroup $N \triangleleft F_{2g}$ is given. 
\\
At first we will see that $\iota(N)$ is $\mathrm{Br}_{2g+1}$-characteristic. Since $\iota(N)$ is $\mathrm{Br}_{2g+1}$-invariant, it is sufficient to show that $\iota(N)$ is normal in $H_{2g+1}$. Recall that $H_{2g+1}$ is generated by $F_{2g} \cup \{\Delta\}$. But $\mathrm{Br}_{2g+1}$ includes the adjoint action determined by $\Delta$, which coincides with the generator of $Z(\mathrm{Br}_{2g+1}) \cong \Z$. Thus we are done.
\\
Set the quotient groups $G:= F_N/N,\; G^{+} := H_{2g+1}/\iota(N)$. Let $\rho_{N}: F_n \twoheadrightarrow G$ and $\rho^{+}_{N}: H_{2g+1} \twoheadrightarrow G^{+}$ be the quotient maps respectively. Taking the quotient by $N$ of the sequence above, we obtain the $\mathrm{Br}_{2g+1}$-equivariant exact sequence
\[
1 \to G \overset{\overline{\iota}}{\longrightarrow} G^{+} \overset{\overline{\mu}}{\longrightarrow} \langle \tau \mid \tau^2 \rangle \to 1.
\]
Set $\tilde{N} := \pi_{2g+1}^{-1}(N)$, where $\pi_{2g+1} : T_{2g+1} \twoheadrightarrow H_{2g+1}$ is the quotient map. Then $\tilde{N} \triangleleft T_{2g+1}$ is $\mathrm{Br}_{2g+1}$-characteristic. Consider
the fundamental sequence associated with $\tilde{N} \triangleleft T_{2g+1}$, which turns out to be $\mathrm{Br}_{2g+1}$-equivariant;
\[
T_{2g+1} \overset{\fd_{\tilde{N}}}{\longrightarrow} L_{\tilde{N}} \overset{\partial_{\tilde{N}}}{\longrightarrow} \Z[G^{+}] \overset{\epsilon_{G^{+}}}{\rightarrow} \Z \to 0
\]
Recall that $T_{2g+1}$ is a free group. Thus $L_{\tilde{N}}$ is the associated Magnus module and $\fd_{\tilde{N}}$ the corresponding Fox derivation. See Proposition \ref{funda sequences} for the detailed construction. Theorem \ref{fox action thm} ensures the existence of the $\mathrm{Br}_{2g+1}$-action on $L_{\tilde{N}}$ and the $\mathrm{Br}_{2g+1}$-equivariance of $\fd_{\tilde{N}}$. 
\begin{lem}\label{lem L plus as quotient}
$\fd_{\tilde{N}}(I_{2g+1}) \subset L_{\tilde{N}}$ is a $\Z[G^{+}]$-submodule invariant under the $\mathrm{Br}_{2g+1}$-action.
\end{lem}
\begin{proof}
The result follows from Proposition \ref{fox submodule} since $I_{2g+1} \subset \tilde{N} \!=\! \mathrm{Ker}(\pi_{2g+1})$ and since $I_{2g+1} \subset T_{2g+1}$ is $\mathrm{Br}_{2g+1}$-invariant.
\end{proof}
\begin{defi} 
Set $L_N^{+} :=  L_{\tilde{N}} / \fd_{\tilde{N}}(I_{2g+1})$. It amounts to a left $G^{+} \rtimes \mathrm{Br}_{2g+1}$-module due to the previous lemma. Define the $\Z[G^{+}]$-homomorphism $\partial_{N}^{+} : L_N^{+} \to \Z[G^{+}]$ to be the map induced from $\partial_{\tilde{N}}$ on the quotient $L_N^{+}$. Notice that 
$\partial_{\tilde{N}}$ factors through $L_N^{+}$ since $\partial_{\tilde{N}} ( \fd_{\tilde{N}}(I_{2g+1}) ) \subset  \partial_{\tilde{N}} ( \fd_{\tilde{N}}(\tilde{N}) )=\{0\}$. 
\end{defi}
\begin{prop}\label{funda diagram}
We have the $\mathrm{Br}_{2g+1}$-equivariant commutative diagram as follows;
\[
\xymatrix{
H_{g,1} \ar[r]^{\fd_{N}^{+}} & L_{N}^{+} \ar[r]^{\partial_{N}^{+}} & \Z[G^{+}] \ar[r]^{\epsilon_{G^{+}}} & \Z \ar[r]  & 0 \\
F_{2g} \ar[u]^{\iota} \ar[r]^{\fd_{N}} & L_{N} \ar[u]^{\hat{\iota}} \ar[r]^{\partial_{N}} & \Z[G] \ar[u]^{\overline{\iota}} \ar[r]^{\epsilon_{G}} & \Z \ar[u]^{\equiv} \ar[r]  & 0
}
\]
, where $\fd_{N}^{+}$ is a $G^{+}$-crossed homomorphism, $\hat{\iota}$ a $\Z[G]$-monomorphism. The 2nd row is the fundamental sequence associated with $N \triangleleft F_{2g}$, which is $\mathcal{M}_{g,1}$-equivariant. Further, it holds that
$\mathrm{Ker}(\fd_{N}^{+}) = \iota([N,N]),\; \fd_{N}^{+} \circ \iota(N) = \mathrm{Ker}(\partial_{N}^{+})$ and $\mathrm{Image}(\partial_{N}^{+}) = \mathrm{Ker}(\epsilon_{G^{+}}).$
\end{prop}

\begin{proof}
\textbullet\, The construction of $\fd_{N}^{+}$.) Denote by $\mathrm{pr} : L_{\tilde{N}} \twoheadrightarrow L_{N}^{+}$ the quotient map. Then $\mathrm{pr} \circ \fd_{\tilde{N}} : T_{2g+1} \to  L_{N}^{+}$ factors through $H_{2g+1}=T_{2g+1}/I_{2g+1}$ since $\mathrm{pr} \circ \fd_{\tilde{N}} (I_{2g+1}) =0$. It follows that there exists a unique $G^{+}$-crossed homomorphism $\fd_{N}^{+} : H_{2g+1} \to L_{N}^{+}$ such that $\fd_{N}^{+} \circ \pi_{2g+1} = \mathrm{pr} \circ \fd_{\tilde{N}}$.
\\
\textbullet\, The construction of $\hat{\iota}$.)  Apply the universality property of the Magnus module $L_{N}$ to the $G$-crossed homomorphism $\fd_{N}^{+} \circ \iota : F_{2g} \to L_{N}^{+}$. (See Theorem \ref{fox univ}.)
\\
\textbullet\, Commutativity of the diagram.) The commutativity of the leftmost square follows from the construction of $\fd_{N}^{+}$. As for the rectangle with vertical edges $\iota$ and $\overline{\iota}$, the upper and lower horizontal morphisms are given by
\begin{align*}
H_{2g+1} \ni a & \mapsto \rho_{N}^{+}(a)-1 \in\Z[G^{+}]
\\
F_{2g} \ni b & \mapsto \rho_{N}(b)-1 \in \Z[G]
\end{align*}
, which shows the results since $\rho_{N}^{+} \circ \iota = \overline{\iota} \circ \rho_{N}$.
As for the next to the leftmost square, the previous result together with the surjectivity of $(\fd_{N}^{+})_{\Z}$ and $(\fd_{N})_{\Z}$ shows the result.
\\
\textbullet\, Asserted 3 properties.) Theorem \ref{magnus thm} implies that $\mathrm{Ker}(\fd_{N}^{+}) = [\tilde{N}, \tilde{N}]$, which shows that $(\fd_{\tilde{N}})^{-1} (\fd_{\tilde{N}}(I_{2g+1}))= I_{2g+1} [\tilde{N}, \tilde{N}]$. 
It follows that $\mathrm{Ker} (\fd_{N}^{+}\circ \pi_{2g+1}) = \mathrm{Ker}(\mathrm{pr}\circ \fd_{\tilde{N}} ) = I_{2g+1} [\tilde{N}, \tilde{N}]$. Since $\pi_{2g+1}$ is surjective, it follows that $\mathrm{Ker} (\fd_{N}^{+}) = \pi_{2g+1}(I_{2g+1} [\tilde{N}, \tilde{N}]) = [\pi_{2g+1}(\tilde{N}), \pi_{2g+1}(\tilde{N})] = [\iota(N), \iota(N)] = \iota([N,N])$. 
The 2nd and the 3rd properties follow from the construction of the 1st row, which was constructed by taking some quotient of the corresponding sequence associated with $\tilde{N} \triangleleft T_{2g+1}$. 
\\
\textbullet\, The injectivitiy of $\hat{i}$.) This follows applying the 5-lemma to the commutative diagram below;
\[
\xymatrix{
0 \ar[r] & \iota(N)^{ab} \ar[r] & L_{N}^{+} \ar[r]^{\partial_{N}^{+}} & \Z[G^{+}] \ar[r]^{\epsilon_{G^{+}}} & \Z \ar[r]  & 0 \\
0 \ar[r] & N^{ab} \ar[u]^{\iota^{ab}}_{\cong} \ar[r] & L_{N} \ar[u]^{\hat{\iota}} \ar[r]^{\partial_{N}} & \Z[G] \ar[u]^{\overline{\iota}} \ar[r]^{\epsilon_{G}} & \Z \ar[u]^{\equiv} \ar[r]  & 0
}
\]
, where each row is exact and $\overline{\iota}$ injective.
\end{proof}
\begin{cor}\label{cor rank} $L_{N}^{+}$ is a free $\Z[G]$-module with $\mathrm{rank}_{\Z[G]} L_{N}^{+} = \mathrm{rank}_{\Z[G]}L_{N} +1 = 2g+1$.
\end{cor}
\begin{proof}
Take any $\tilde{\tau} \in \overline{\mu}^{-1}(\tau)$, where $\overline{\mu} : G^{+} \twoheadrightarrow \langle \tau \mid \tau^2 \rangle$. Then we have a left $\Z[G]$-module decomposition
\[
\Z[G^{+}] = \mathrm{Image}(\iota) \oplus \Z[G](\tilde{\tau}-1)
\]
since $\Z[G^{+}]$ is a free left $\Z[G]$-module with the free basis $\{1_{\Z[G^{+}]},\, \tilde{\tau} \}$. Since $\Z[G](\tilde{\tau}-1) \subset \mathrm{Ker}(\epsilon^{+}) = \mathrm{Image}(\partial_{N}^{+})$, if we set $K := (\partial_{N}^{+})^{-1}(\Z[G](\tilde{\tau}-1))$, we have the surjection $\partial_{N}^{+}|_{K} : K \to \Z[G](\tilde{\tau}-1)$. Since $\Z[G](\tilde{\tau}-1)$ is projective as a $\Z[G]$-module, it follows that there exists a right inverse $\Z[G]$-homomorphism $\lambda : \Z[G](\tilde{\tau}-1) \to K$. On the other hand, the commutative diagram that appeared in the proof of Proposition \ref{funda diagram} shows that $(\partial_{N}^{+})^{-1}(\overline{\iota}(\Z[G])) = \hat{\iota}(L_{N})$. Thus we have the following direct sum decomposition as $\Z[G]$-module;
\[
L_{N}^{+} = \hat{\iota}(L_{N}) \oplus \mathrm{Image}(\lambda)
\]
, where the 1st summand in the R.H.S. is free of $\mathrm{rank}\, 2g$ since $\hat{\iota}$ is injective and the 2nd summand is free of $\mathrm{rank}\, 1$ since $\lambda$ is injective. Thus, we are done.
\end{proof}
The argument in the proof of Corollary \ref{cor rank} shows the following rather technical result.
\begin{cor}\label{cor tech basis}   Let $\{ a_1,\dots, a_{2g} \}$ be a free $\Z[G]$-basis of $L_{N}$ and $b\in L_{N}^{+}$ satisfy the condition $\partial_{N}(b) +1 \in G^{+} \backslash G$. Then $\{ \hat{\iota}(a_1),\dots,\hat{\iota}(a_{2g}), b \}$ is a free $\Z[G]$-basis of $L_{N}^{+}$.
\end{cor}
\begin{defi} Set
\[
\triangle : = \rho_{N}^{+}(\Delta) \in G^{+} .
\]
\end{defi}
\begin{rem} $\triangle^2 \in G$. $\triangle^2$ is central in $G$ if and only if it is central in $G^{+}$.
\end{rem}
In the end of the present subsection, we will give an application of the separation criterion. (See Proposition \ref{fox decomp}.)
\\
Suppose that $\triangle^2$ is central in $G$ (hence in $G^{+}$) and that $\triangle^{2p}=1$ where $p\neq \pm1$ (not necessarily prime). Then $\triangle^2$ belongs to the center of $G^{+} \rtimes \mathrm{Br}_{2g+1} $ since the $\mathrm{Br}_{2g+1}$-action on $G^{+}$ preserves $\triangle$. Then the eigenspace of the operator  $\mathrm{L}_{\triangle^2}$ determined by the left multiplication by $\triangle^2$ are preserved under the $G^{+} \rtimes \mathrm{Br}_{2g+1}$-action. Take any $p$th root of unity $\lambda \neq 1$. Complexifying the diagram in Proposition \ref{funda diagram}, taking its eigenspace decomposition with respect to $\mathrm{L}_{\triangle^2}$ and bring out the $\lambda$-part, we have the following $G^{+} \rtimes \mathrm{Br}_{2g+1}$-equivariant commutative diagram with exact rows;
\[
\xymatrix{
 (L_{N}^{+})_{\lambda} \ar[r]^{(\partial_{N}^{+})_{\lambda}} & \Z[G^{+}]_{\lambda} \ar[r]  & 0 \\
 (L_{N})_{\lambda} \ar[u]^{\hat{\iota}_{\lambda}}  \ar[r]^{(\partial_{N})_{\lambda}} & \Z[G]_{\lambda} \ar[u]^{\overline{\iota}_{\lambda}} \ar[r]  & 0
}
\]
, where the subscript "$\lambda$" refers to the eigenspace of eigenvalue $\lambda$ w.\ r.\ t. the $\mathrm{L}_{\triangle^2}$-action. Notice that $\hat{\iota}$ and $\overline{\iota}$ are injective.
\begin{prop} Assume that $\triangle^2$ is central in $G$ and that $\triangle^{2p}=1$. Then, for any $p$th root of unity $\lambda\neq 1$, we have the direct sum decomposition as $G \rtimes \mathrm{Br}_{2g+1}$-module
\begin{align*}
(L_{N}^{+})_{\lambda} &= \mathrm{Ker}(\partial_{N})_{\lambda} \oplus \C[G^{+}]_{\lambda} \fd_{N}^{}(\Delta^2)
\end{align*}
, where the 2nd summand of the R.H.S. is a free $\C[G^{+}]_{\lambda}$-module on the basis $\{ 1_{\Z[G^{+}]_{\lambda} } \fd_{N}^{+}(\Delta^2) \}$. Here  $1_{\Z[G^{+}]_{\lambda} } \in \Z[G^{+}]_{\lambda}$ denotes the multiplicative identity. Both summands are free as $\C[G]_{\lambda}$-module such that the former is of $\mathrm{rank} \,g$ and the latter of $\mathrm{rank} \,2$.
\end{prop}
\begin{proof}
Notice that $\mathrm{Ker}(\partial_{\tilde{N}}^{+})=\mathrm{Ker}(\partial_{\tilde{N}})$ due to Proposition \ref{funda diagram} and that $\fd_{N}^{+}(\Delta^2)=\fd_{N}(\Delta^2)$ since $\Delta^2 \in F_n$. Applying the proof of Proposition \ref{fox decomp} to $\fd_{N}^{+} \ : H_{2g+1} \to L_{N}^{+}$, we obtain the desired decomposition. Each summand is preserved under $\mathrm{Br}_{2g+1}$-action due to Remark \ref{decomp invariant}. Thus we are done.
\end{proof}
\begin{rem} Restricting the argument above to $(L_{N})_{\lambda} \subset (L_{N}^{+})_{\lambda}$, we obtain immediately the following direct sum decomposition as $G \rtimes \mathrm{Br}_{2g+1}$-module (actually as $G \rtimes \mathcal{M}_{g,1}$-module); 
\[
\\
(L_{N})_{\lambda} = \mathrm{Ker}(\partial_{N})_{\lambda} \oplus \C[G]_{\lambda} \fd_{N}^{}(\Delta^2) \quad (\lambda \neq 1).
\]
In other words, $(N^{\mathit{ab}})_{\lambda}\; (\lambda \neq 1)$ is a direct summand of $L_{N}$ as $G \rtimes \mathcal{M}_{g,1}$-module if $\triangle^2 \in G$ is of finite order and included in $Z(G)$.
\end{rem}
%
%
%
\section{\bf Fundamental formulae describing the braid group action on the Magnus modules }\label{section 3}

\subsection{ A good basis of $L_{N}$ }

In the previous section, for any $\mathcal{M}_{g,1}$-characteristic subgroup $N \triangleleft F_{2g}$, we introduced the Magnus modules $L_{N}$ and $L_{N}^{+}$ such that the former is a $\Z[ F_{2g}/N \rtimes \mathcal{M}_{g,1} ]$-module and the latter 
a $\Z[ H_{2g+1}/N \rtimes \mathrm{Br}_{2g+1} ]$-module. Further, we gave the embedding $\hat{\iota} : L_{N} \hookrightarrow L_{N}^{+}$ as  $\Z[ F_{2g}/N \rtimes \mathrm{Br}_{2g+1} ]$-modules and defined the Fox derivation $\fd_{N}^{+} : H_{2g+1} \to L_{N}^{+}$ so that it extends the usual Fox derivation $\fd_{N} : F_{2g} \to L_{N}$.  
\\

Now we will introduce some good bases of these Magnus modules, $L_{N}$ and $L_{N}^{+}$, for the convenience of calculation. For the moment, we will fix a $\mathcal{M}_{g,1}$-characteristic subgroup $N \triangleleft F_{2g}$. Denote by $G$ and $G^{+}$ the quotient groups $F_n/N$ and $H_{2g+1}/N$, respectively, and denote the quotient maps by $\rho_{N} : F_n \twoheadrightarrow F_n/N$ and $\rho_{N}^{+} : H_{2g+1} \twoheadrightarrow H_{2g+1}/N$, respectively.
\begin{defi} Define the group elements
\[
x_j:=\rho_{N}(\bx_j) \in G \;\; (1\leq j\leq 2g) , \quad  y_i : = \rho_{N}^{+}(\by_i) \in G^{+} \;\; (1\leq i \leq 2g+1).
\]
\end{defi}
\begin{defi} Define the module elements
\[
 f_j:=\fd_{N}(\bx_j) \in L_{N} \;\; (1\leq j\leq 2g) , \quad t_i := \fd_{N}^{+}(\by_i) \in L_{N}^{+} \;\; (1\leq i \leq 2g+1).
\]
Notice that
\[
f_j =\fd_{N}^{+}(\by_i \by_{i+1}) = t_i + y_i t_{i+1} \quad (1\leq j\leq 2g).
\]
\end{defi}
%
%
\begin{lem} The sets $\{f_1,\dots,f_{2g}\}$ and $\{ t_1,\dots,t_{2g+1} \}$ are free $\Z[G]$-bases of $L_{N}$ and $L_{N}^{+}$, respectively.
\end{lem}
\begin{proof} These are the direct consequences of the construction of $L_{N}$ and $L_{N}^{+}$.
\end{proof}
\begin{defi}
Denote by $\Psi_i \in \mathrm{Aut}(G^{+}) \; (1\leq i \leq 2g)$ the automorphism induced by the action of $\sigma_i \in \mathrm{Br}_{2g+1}$ on $H_{2g+1}$. Denote by $\hat{\Psi}_i \in \mathrm{End}_{\Z}(L_{N}^{+})  \; (1\leq i \leq 2g)$ the endomorphism induced by the action of $\sigma_i$ on $L_{N}^{+}$. (See Theorem \ref{fox action thm}.) Notice that each $\hat{\Psi}_i$ preserves $L_{N} \subset L_{N}^{+}$.
\end{defi}
\begin{lem} As a direct consequence of Proposition \ref{braid action on T}, the action of $\Psi_i \; (1\leq i \leq 2g)$ on $G^{+}$ is written as follows;
\[
\Psi_i(y_j) = 
\begin{cases}
 y_{i+1} & \text{ if } j=i,
 \\
 y_{i+1}y_{i}y_{i+1} & \text{ if } j=i+1, 
 \\
 y_{j} & \text{ if } j \neq i,i+1.
 \end{cases}
\]
\end{lem}
\begin{prop} A short calculation shows that $\hat{\Psi}_i  \; (1\leq i \leq 2g)$ is described as follows;
\[ 
\hat{\Psi}_i(t_j) = 
\begin{cases}
 t_{i+1} & \text{ if } j=i,
 \\
 y_{i+1}t_{i} + (1 + y_{i} y_{i+1} ) t_{i+1} & \text{ if } j=i+1, 
 \\
 t_{j} & \text{ if } j \neq i,i+1.
 \end{cases}
\]
\end{prop}
We will introduce a new basis of $L_{N}^{+}$.
\begin{defi} Define a set of elements of $L_{N}^{+}$ as follows;
\begin{align*} 
\bu_{2i} &:= y_1 y_2 \dots y_{2i-1} \fd_{N}^{+}(\by_{2i} \by_{2i+1} \dots \by_{2g+1}) \quad (1\leq i \leq g),
\\
\bu_{2i-1} &:= -y_{2g+1} y_{2g} \dots y_{2i+1} \fd_{N}^{+}(\by_{2i} \by_{2i-1}\dots \by_{1}) \quad (1\leq i \leq g),
\\
\bw_{0} &:= \fd_{N}^{+} (\by_{1} \by_{2} \dots \by_{2g+1}) \equiv \fd^{+}_{N}(\Delta).
\end{align*}
\end{defi}
\begin{rem} It is readily seen that
$\bu_{i},\, \bw_{0} \notin \hat{\iota}(L_{N})$ for $1\leq i \leq 2g$.
\end{rem}
Recall that $\triangle \equiv y_1\dots y_{2g+1} = \rho_{N}^{+}(\Delta)$, where $\Delta=\by_{1} \by_{2} \dots \by_{2g+1}$.
\begin{thm}[Good basis of $L_{N}^{+}$]\label{free new basis} Assume that the group $G$ is finite. Then the set
\[
\{ \bu_{i} \mid 1 \leq i \leq 2g \} \cup \{ \bw_{0} \}
\] 
is a free $\Z[G]$-basis of $L_{N}^{+}$
\end{thm}
\begin{proof} It follows from the $2$ lemmas just below
that $\{y_1\bu_1,\dots y_1\bu_{2g}, \triangle^{-1} \bw_{0} \}$ is a free $\Z[G]$-basis of $L_{N}^{+}$. Since $y_1 \triangle \in G$, we see that $\{ \triangle^{-1}\bu_1,\dots \triangle^{-1}\bu_{2g}, \triangle^{-1} \bw_{0} \}$ is another free $\Z[G]$-basis of $L_{N}^{+}$, which implies that $\{\bu_1,\dots \bu_{2g}, \bw_{0} \}$ is a free $\triangle (\Z[G]) \triangle^{-1}$-basis of $L_{N}^{+}$. But $\triangle (\Z[G]) \triangle^{-1} = \Z[G]$ since $G\triangleleft G^{+}$. Thus we are done.
\end{proof}
\begin{lem}\label{free new basis 1} The set $\{f_1,\dots,f_{2g}, \triangle^{-1}\bw_0 \}$ is a free $\Z[G]$-basis of $L_{N}^{+}$.
\end{lem}
\begin{proof}
Since $\partial_{N}^{+}(\bw_{0}) = \partial_{N}^{+} \circ \fd_{N}^{+}(\Delta) = \rho_{N}^{+}(\Delta)-1$, we see that $\partial_{N}^{+}( -\triangle^{-1} \bw_{0}) +1 = -\triangle^{-1} (\rho_{N}^{+}(\Delta)-1) +1 = \triangle^{-1} \in G^{+} \backslash G$. Therefore Corollary \ref{cor tech basis} shows the assertion.
\end{proof}
\begin{lem}\label{free new basis 2} The set $\{ y_1\bu_1, y_1 \bu_2 \dots, y_1 \bu_{2g} \}$ is a free $\Z[G]$-basis of $L_{N}$.
\end{lem}
\begin{proof} Since $G$ is finite, it is sufficient to show that $\{ y_1\bu_1, y_1 \bu_2 \dots, y_1 \bu_{2g} \}$ generates the $\Z[G]$-module $L_{N}$. By definition, it holds that
\[
y_1\bu_{2i} =x_2 x_4 \dots x_{2i-2}\, \fd_{N}^{+}(\bx_{2i} \dots \bx_{2g})
\]
, which implies that
\[
 f_{2i} = \fd_{N}^{+}(\bx_{2i}) \in \langle f_{2i+2}, \dots, f_{2g},\, y_1\bu_{2i} \rangle_{\Z[G]} .
\]
In particular, $f_{2g} \in \langle y_1 \bu_{2g} \rangle_{\Z[G]}$. Thus a downward inductive argument w.r.t.\ $i$ shows that
\[
f_{2i} \in \langle y_1\bu_{2i}, \dots, y_1\bu_{2g} \rangle_{\Z[G]} \quad (1\leq i \leq g) .
\]
Similarly, it holds that
\[
y_1 \bu_{2i-1} = (x_1 x_2 \dots x_{2g})(x_{2+1} x_{2i+3} \dots x_{2g-1})^{-1}\, \fd_{N}^{+} \left( (\bx_1 \bx_3 \dots \bx_{2i-1})^{-1} \right)
\]
, which implies that
\[
f_{2i-1} \in \langle f_1, \dots f_{2i-3},\, y_1\bu_{2i-1} \rangle_{\Z[G]}.
\]
In particular, $f_1 \in \langle y_1 \bu_{1} \rangle_{\Z[G]}$. Thus an upward inductive argument w.r.t.\ $i$ shows that
\[
f_{2i-1} \in \langle y_1\bu_{1}, \dots, y_1\bu_{2i-1} \rangle_{\Z[G]} \quad (1\leq i \leq g).
\]
Now that we have shown that $\{ f_i \mid 1\leq i \leq 2g \} \subset \langle y_1\bu_{i} \mid 1 \leq i \leq 2g \rangle_{\Z[G]}$, we are done.
\end{proof}
%
Now we will give a set of formulae describing the $\mathrm{Br}_{2g+1}$ action on $L_{N}^{+}$ with respect to the basis introduced in Theorem \ref{free new basis}. This is a point of departure for the explicit calculation in the present paper. 
\begin{thm}[Formulae for the braid group action]\label{funda formulae} For $1\leq k \leq 2g$, it holds that
\begin{align*}
\hat{\Psi}_{2i-1} (\bu_{j}) -\bu_{j} &=
\begin{cases} 
 - \left\{ x_1 x_3 \dots x_{2i-3} \cdot x_{2i} x_{2i+2}\dots x_{2g} \right\} (\bu_{2i-1} - \bu_{2i-3}) & \text{ if } j=2i,
\\
0 & \text{ if } j\neq 2i.
\end{cases}
\\
\hat{\Psi}_{2i} (\bu_{j}) - \bu_{j} &= 
\begin{cases}
 \left\{ x_1 x_3 \dots x_{2i-1} \cdot x_{2i} x_{2i+2}\dots x_{2g} \right\}^{-1} (\bu_{2i} - \bu_{2i+2}) & \text{ if } j=2i-1.
 \\
 0 & \text{ if } j\neq 2i-1.
\end{cases}
\\
\hat{\Psi}_{k} (\bw_0) - \bw_0 &= 0 .
\end{align*}
\end{thm}
%
\begin{proof} At first, we will prove the 1st formula. Since $\fd^{+}_{N}$ commutes with the $\mathrm{Br}_{2g+1}$-action, we see that
\begin{align*}
\hat{\Psi}_{j}(\bu_{2i}) &= \Psi_{j}(y_1y_2\dots y_{2i-1}) \hat{\Psi}_{j} \left( \fd^{+}_{N} (\by_{2i} \by_{2i+1} \dots \by_{2g+1}) \right)
\\
&= \Psi_{j}(y_1y_2\dots y_{2i-1}) \fd^{+}_{N} \left( \sigma_{j}(\by_{2i} \by_{2i+1} \dots \by_{2g+1}) \right)
.
\end{align*}
If $j\neq 2i-1$, then $\hat{\Psi}_{j}(\bu_{2i})=\bu_{2i}$ since $\Psi_{j}$ fixes $y_1y_2\dots y_{2i-1}$ and since $\sigma_j$ fixes $\by_{2i} \by_{2i+1} \dots \by_{2g+1}$. If $j= 2i-1$, then it follows that
\begin{align*}
\hat{\Psi}_{2i-1}(\bu_{2i}) =& \Psi_{2i-1}(y_1y_2\dots y_{2i-1}) \, \fd^{+}_{N} \left( \sigma_{2i-1} (\by_{2i} \by_{2i+1} \dots \by_{2g+1}) \right).
\\
=& y_1y_2\dots y_{2i-2} y_{2i} \, \fd (\by_{2i} \by_{2i-1} \by_{2i} \by_{2i+1}\dots \by_{2g+1})
\\
=& y_1y_2\dots y_{2i-2} y_{2i} \left( t_{2i} + y_{2i} t_{2i-1} \right)
\\
 & + y_1y_2\dots y_{2i-2} y_{2i} (y_{2i} y_{2i-1} ) \, \fd^{+}_{N} (\by_{2i} \by_{2i+1}\dots \by_{2g+1} )
\\
=& y_1y_2\dots y_{2i-2} ( t_{2i-1} + y_{2i} t_{2i} ) + \bu_{2i}.
\end{align*}
On the other hand, we see that
$
\bu_{2i-1}-\bu_{2i-3} = -y_{2g+1} y_{2g} \dots y_{2i+1} \, \fd^{+}_{N}(\by_{2i} \by_{2i-1}) 
$, 
which implies that
\begin{align*}
& y_1y_2\dots y_{2i-2} \cdot y_{2i} y_{2i+1} \dots y_{2g+1} (\bu_{2i-1}-\bu_{2i-3})
\\
&=- y_1y_2\dots y_{2i-2} \cdot y_{2i} \, \fd^{+}_{N}( \by_{2i} \by_{2i-1} )
\\
&= - y_1y_2\dots y_{2i-2} ( y_{2i} t_{2i} +  y_{2i}^2 t_{2i-1} )
\\
&= - y_1y_2\dots y_{2i-2} ( t_{2i-1} + y_{2i} t_{2i}) .
\end{align*}
Thus we have
\[
\hat{\Psi}_{2i-1}(\bu_{2i}) = - y_1y_2\dots y_{2i-2} \cdot y_{2i} y_{2i+1} \dots y_{2g+1} (\bu_{2i-1}-\bu_{2i-3}) + \bu_{2i}.
\]
\\
Now we will prove the 2nd formula. The same reasoning as in the case of 1st formula shows that
\[
\hat{\Psi}_{j}(\bu_{2i-1})  = -\Psi_{j}(y_{2g+1} y_{2g} \dots y_{2i+1}) \, \fd^{+}_{N} \left( \sigma_{j}(\by_{2i} \by_{2i-1}\dots \by_{1}) \right).
\]
If $j\neq 2i$, then $\hat{\Psi}_{j}(\bu_{2i-1}) = \bu_{2i-1}$ since $\Psi_{j}$ fixes $y_{2g+1} y_{2g} \dots y_{2i+1}$ and since $\sigma_{j}$ fixes $\by_{2i} \by_{2i-1}\dots \by_{1}$. If $j=2i$, it follows that
\begin{align*}
\hat{\Psi}_{2i}(\bu_{2i-1})  =& -y_{2g+1} y_{2g} \dots y_{2i+1} y_{2i} y_{2i+1} \, \fd^{+}_{N}  (\by_{2i+1} \by_{2i-1}\dots \by_{1})
\\
=&  -y_{2g+1} y_{2g} \dots y_{2i+1} y_{2i} y_{2i+1} \, \fd^{+}_{N} (\by_{2i+1} \by_{2i} \by_{2i} \by_{2i-1}\dots \by_{1})
\\
=& -y_{2g+1} y_{2g} \dots y_{2i+1} y_{2i} y_{2i+1} (t_{2i+1} + y_{2i+1} t_{2i} )
\\
& -y_{2g+1} y_{2g} \dots y_{2i+1} \, \fd^{+}_{N} (\by_{2i} \by_{2i-1}\dots \by_{1})
\\
&= -y_{2g+1} y_{2g} \dots y_{2i+1} y_{2i}  ( t_{2i} + y_{2i+1} t_{2i+1} ) + \bu_{2i-1} .
\end{align*}
On the other hand, we see that 
$
\bu_{2i}-\bu_{2i+2} = y_1y_2\dots y_{2i-1} \,\fd^{+}_{N}(\by_{2i} \by_{2i+1})
$,\\ 
which implies that
\begin{align*} 
y_{2g+1} y_{2g} \dots y_{2i+1}  y_{2i-1} \dots y_{1} (\bu_{2i}-\bu_{2i+2}) 
& = y_{2g+1} y_{2g} \dots y_{2i+1} \,\fd^{+}_{N}(\by_{2i} \by_{2i+1})
\\
&= y_{2g+1} y_{2g} \dots y_{2i+1} (t_{2i} + y_{2i} t_{2i+1} )
\\
&= y_{2g+1} y_{2g} \dots y_{2i+1} y_{2i} ( y_{2i} t_{2i} +  t_{2i+1} )
\\
&= -y_{2g+1} y_{2g} \dots y_{2i+1} y_{2i} (  t_{2i} +  y_{2i+1} t_{2i+1} )
\end{align*}
, where the last equality follows from the calculation 
\[
0 = \fd^{+}_{N}(1) = \fd^{+}_{N}(y_k^2) = (1 + y_{k})t_{k} \quad (1\leq k \leq 2g).
\]
 Thus we have
\[ 
\hat{\Psi}_{2i}(\bu_{2i-1}) = y_{2g+1} y_{2g} \dots y_{2i+1} \cdot y_{2i-1} \dots y_{1} (\bu_{2i}-\bu_{2i+2}) + \bu_{2i-1} .
\]
\\
The proof of the 3rd formula is fairly easy. In fact,
\begin{align*}
\hat{\Psi}_{k}(\bw_{0}) &= \fd^{+}_{N} \left( \sigma_{k}(\by_1 \by_2\dots \by_{2g+1}) \right)
\\
&= \fd^{+}_{N} ( \by_1 \by_2\dots \by_{2g+1})
\\
&= \bw_{0} .
\end{align*} 
\end{proof}

\vspace{-1cm}

\subsection{ Multiplicative Jordan decomposition of $\hat{\Psi}_i$ }
%
Consider the multiplicative Jordan decomposition of the invertible endomorphism $\hat{\Psi}_i \;(1\leq i \leq 2g)$ as follows;  
\[
\hat{\Psi}_i = \hat{\Psi}_{i,\mathrm{ss}} \circ \hat{\Psi}_{i,\mathrm{uni}} = \hat{\Psi}_{i,\mathrm{uni}} \circ \hat{\Psi}_{i,\mathrm{ss}}
\]
, where $\hat{\Psi}_{i,\mathrm{ss}}$ is semisimple and $\hat{\Psi}_{i,\mathrm{uni}}$ unipotent, respectively. 
\begin{prop}\label{funda formulae jordan} Suppose that $x_i^p =1 \;(1\leq i \leq 2g)$ for some $p \in \N$. For any $h \in \C[G]$, it holds that
\footnotesize
\begin{align*}
& \hat{\Psi}_{2k-1,\mathrm{ss}} (h\bu_{i}) 
= \Psi_{2k-1}(h) \Big\{ \bu_{i} \,-\, \delta_{i, 2k} \left\{ x_1 x_3 \dots x_{2k-3} \left(1-\bphi_p(x_{2k-1})\right) x_{2k} x_{2k+2}\dots x_{2g} \right\} ( \bu_{2k-1} - \bu_{2k-3} ) \Big\} ,
\\
& \hat{\Psi}_{2k-1,\mathrm{uni}} (h\bu_{i}) 
= h \Big\{ \bu_{i} \,-\, \delta_{i,2k} \left\{ x_1 x_3 \dots x_{2k-3} \bphi_p(x_{2k-1}) x_{2k} x_{2k+2}\dots x_{2g} \right\} (\bu_{2k-1} - \bu_{2k-3}) \Big\} ,
\\
& \hat{\Psi}_{2k,\mathrm{ss}} (h\bu_{i}) = 
\Psi_{2k}(h) \Big\{ \bu_{i} \,+\, 
\delta_{i,2k-1}  \left\{ x_{2g}^{-1} x_{2g-2}^{-1} \dots  x_{2k}^{-1} \left(1-\bphi_p(x_{2k})\right) x_{2k-1}^{-1} x_{2k-3}^{-1} \dots  x_{1}^{-1} \right\} (\bu_{2k} - \bu_{2k+2}) \Big\} , 
\\
& \hat{\Psi}_{2k,\mathrm{uni}} (h\bu_{i}) = 
h \Big\{ \bu_{i} \,+\, 
\delta_{i,2k-1}  \left\{ x_{2g}^{-1} x_{2g-2}^{-1} \dots  x_{2k}^{-1} \bphi_p(x_{2k}) x_{2k-1}^{-1} x_{2k-3}^{-1} \dots  x_{1}^{-1} \right\} (\bu_{2k} - \bu_{2k+2}) \Big\} . 
\end{align*}
\normalsize
, where $\phi_p$ is the polynomial defined by 
\[
\phi_p(t):=\frac{1}{p}\sum_{i=1}^{p-1} t^{i}.
\] 
\end{prop}
\begin{proof} Applying the 1st formula in Theorem \ref{funda formulae} $m$ times, we see that
\footnotesize
\begin{align*}
& (\hat{\Psi}_{2k-1,\mathrm{ss}})^{m} (h\bu_{2k}) - (\Psi_{2k-1})^m (h) \,\bu_{2k}
\\
&= -(\Psi_{2k-1})^{m} (h) \left\{ x_1 x_3 \dots x_{2k-3} \left(1+ x_{2k-1}^{-1} + \cdots + x_{2k-1}^{-(m-1)} \right) x_{2k} x_{2k+2}\dots x_{2g} \right\} ( \bu_{2k-1} - \bu_{2k-3} ) .
\end{align*}
\normalsize
Taking into account the assumption that $x_{2k-1}^p =1$, a short calculation shows that $((\hat{\Psi}_{2k-1})^p-\mathrm{id} )^2 =0$, which implies the minimal polynomial of $\hat{\Psi}_{2k-1}$ divides $(t^p-1)^2$. Thus Lemma \ref{multi jordan} below shows the result for $i=2k-1$. The case where $i$ is even is left to the reader since the argument is very similar.
\end{proof}
%
\begin{lem}\label{multi jordan} If the minimal polynomial of a linear transformation $A$ divides $(t^p-1)^2$, the multiplicative Jordan decomposition of $A$ is given by $A_{ss}= A(1-\frac{1}{p} (A^p -1)),\; A_{uni} = 1 + \frac{1}{p} (A^p -1)$. 
\end{lem}
\begin{proof} By the assumption, $A^p = 1 + N$ such that $N^2=0$. Write the multiplicative Jordan decomposition of $A$ as $A=S (1+L)$,  where $S$ is semisimple and $L$ nilpotent such that $SL=LS$. Then $A^p = S^p+ pS^p L (1 + LM)$ where $M$ is some nilpotent element that commutes with both $S$ and $L$.  The uniqueness of additive Jordan decomposition for $A^p$ implies that $1=S^p,\; N=pS^p L (1+LM)$. Therefore, $N = pL (1 + LM)$. But, $N^2=0$ implies $L^2=0$ , which in turn implies $N= pL$. Therefore, we have $L=\frac{1}{p} (A^p -1)$ and $S=A(1+L)^{-1} = A(1-L)$. Thus, we are done.
\end{proof}
Notice that, in the above proof, we have deduced the following fact which we will fully make use of in the subsequent sections.
\begin{cor}\label{42} Under the same assumption as Proposition \ref{funda formulae jordan}, $(\hat{\Psi}_{i,\mathrm{ss}} )^p = \mathrm{id}_{L_{N}}$.
\end{cor}

%
%
%

\section{\bf Concrete Calculation}\label{section 4}

\subsection{\bf $\mathrm{H}(p,2g)$}

Let $p$ be an odd prime. Consider the free group $F_{2g} =\langle \bx_1,\bx_2,\dots,\bx_{2g} \rangle$. The abelianization $F_{2g}^{\mathit{ab}}$ is a free abelian group of $\mathrm{rank}\, 2g$ generated by $\{ [\bx_1], [\bx_2],\dots, [\bx_{2g}] \}$. The intersection form on 
$F_{2g}^{\mathit{ab}}$ was defined in Subsection \ref{braid group action}, which we will denote by $\omega(\cdot,\cdot) : F_{2g}^{\mathit{ab}} \times F_{2g}^{\mathit{ab}} \to \Z$. Since the induced action of $\mathcal{M}_{g,1}$ on $F_{2g}^{\mathit{ab}}$ preserves $\omega$, we have the group homomorphism
\[
\mu : \mathcal{M}_{g,1} \to \mathrm{Aut}(F_{2g}^{\mathit{ab}}, \omega):= \{ \phi \in \mathrm{Aut}(F_{2g}^{\mathit{ab}}) \mid \phi(\omega)=\omega \} \cong \mathrm{Sp}(2g,\Z).
\]
Set $A_{2g} := F_{2g}^{\mathit{ab}} \otimes_{\Z} \bbf$. Denote by $\overline{\omega}$ the induced skew-symmetric $\bbf$-bilinear form on $A_{2g}$, which is non-degenerate since $\omega$ does so. 
\begin{nota*} For any subset $S$ of a group $G$, we denote by $S^p$ the normal closure in $G$ of the subset $\{ s^p \mid s \in S\}$.
\end{nota*}
\begin{nota*}
For any $h \in F_{2g}$, we denote
\begin{align*}
\overline{h} &:= h \;\;\mathrm{mod}\;\; [F_{2g}, F_{2g}] ,
\qquad
\overline{\overline{h}} := h \;\;\mathrm{mod}\;\; (F_{2g}^{\mathit{ab}})^p .
\end{align*}
\end{nota*}
\begin{nota*} Denote by $F_{2g}^{(k)}$ the $k$-th member of the lower central series of $F_{2g}$, that is , $F_{2g}^{(1)} = F_{2g}$, $F_{2g}^{(2)} = [F_{2g}, F_{2g}]$ , $F_{2g}^{(3)} = [F_{2g}^{(2)}, F_{2g}]$ and so on.
\end{nota*}
\begin{lem} We have the canonical isomorphisms
\begin{align*}
A_{2g} &\;\cong\; F_{2g}^{(1)} \bigm\slash F_{2g}^{(2)} (F_{2g}^{(1)})^p,
\qquad
\overset{2}{\wedge} A_{2g} \;\cong\; F_{2g}^{(2)} \bigm\slash F_{2g}^{(3)} (F_{2g}^{(2)})^p.
\end{align*}
\end{lem}
\begin{proof} The 1st statement is rather trivial. As for the 2nd, we have the canonical isomorphism
\[
F_{2g}^{(2)} \bigm\slash F_{2g}^{(3)} \;\cong\; \overset{2}{\wedge} F_{2g}^{\mathit{ab}} \; : \; [a,b] \;\mathrm{mod}\; F_{2g}^{(3)} \;\longleftrightarrow\; \overline{a} \wedge \overline{b}.
\]
Applying the functor $\otimes_{\Z} \bbf$ to this isomorphism, we are done.
\end{proof}
\begin{defi} Define the contraction map
\[  
\mathrm{Cont}_{\overline{\omega}} : F_{2g}^{(2)} \bigm\slash F_{2g}^{(3)} (F_{2g}^{(2)})^p \longrightarrow \bbf \; : \; [a,b] \;\mathrm{mod}\; F_{2g}^{(3)} (F_{2g}^{(2)})^p \longmapsto \omega(a,b) \;\mathrm{mod}\; p\Z.
\]
\end{defi}
Since $\omega$ is non-degenerate, $\mathrm{Cont}_{\overline{\omega}}$ is surjective. Since $\omega$ is invariant with respect to the $\mathcal{M}_{g,1}$-action, we have the $\mathcal{M}_{g,1}$-equivariant exact sequence as follows;
\begin{equation}
1 \to \mathrm{Ker}(\mathrm{Cont}_{\overline{\omega}}) \; \overset{\iota}{\longrightarrow} \;  F_{2g}^{(1)} \bigm\slash F_{2g}^{(3)} (F_{2g}^{(2)})^p \;\overset{\mathrm{Cont}_{\overline{\omega}} }{\longrightarrow} \; \bbf \to 0 .
\end{equation}
\begin{prop}\label{equiv central ext} We have the $\mathcal{M}_{g,1}$-equivariant central extension
\[
1 \to  F_{2g}^{(2)} \bigm\slash F_{2g}^{(3)} (F_{2g}^{(2)})^p \; \overset{\iota}{\longrightarrow} \;  F_{2g}^{(1)} \bigm\slash F_{2g}^{(3)} (F_{2g}^{(1)})^p \; \longrightarrow \; F_{2g}^{(1)} \bigm\slash F_{2g}^{(2)} (F_{2g}^{(1)})^p \to 1 .
\] 
\end{prop}
\begin{proof} The statement is rather trivial except for the injectivity of $\iota$. \\
Set $H:= F_{2g}^{(1)} \!\bigm\slash\! F_{2g}^{(3)}$. It is sufficient to show the injectivity of the natural map
\[
\iota^{\prime} :  H^{(2)} \!\bigm\slash\! (H^{(2)})^p \to H^{(1)} \!\bigm\slash\! (H^{(1)})^p 
\]
, which is equivalent to show $H^{(2)} \cap (H^{(1)})^{p} \subset (H^{(2)})^p$.
\begin{slem*} 
$a^p b^p \equiv (ab)^p \;\mathrm{mod}\; (H^{(2)})^p \;$ for any $a,b \in H$.
\end{slem*}
\begin{proof}
\[
a^p b^p = (ab)^p [a,b]^{\frac{1}{2} p(p-1)} \equiv (ab)^p \;\mathrm{mod}\; [H,H]^p
\]
, where at the congruence on the right we have used the assumption that $p$ is odd, which implies that $2 \,\mathrm{mod}\, (p)$ is invertible in $\bbf$. Thus we are done.
\end{proof}
Returning to the proof of the present proposition, let $a\in H^{(2)} \cap (H^{(1)})^{p}$. By the sublemma above, we may assume that $a=b^p$ for some $b \in H$. Then $(\overline{b})^p = 0$ in $H^{\mathit{ab}}$ since $b^p \in H^{(2)}$. But $H^{\mathit{ab}}$ is a free abelian group, which implies that $\overline{b} =0$, that is, $b \in H^{(2)}$. It follows that $a=b^p \in (H^{(2)})^p$. Thus we are done. 
\end{proof}
\begin{defi} Set the group
\[
\mathrm{H}(p,2g) := \left( F_{2g}^{(1)} \bigm\slash F_{2g}^{(3)} (F_{2g}^{(1)})^p \right) 
\bigm\slash \iota\left(\mathrm{Ker}(\mathrm{Cont}_{\overline{\omega}})\right).
\] 
\end{defi}
\begin{cor}\label{central ext finite} We have the $\mathcal{M}_{g,1}$-equivariant central extension
\[
1 \to \mathrm{C}_{p} \rightarrow \mathrm{H}(p,2g) \overset{\mathrm{p}}{\rightarrow} A_{2g} \to 1
\]
, where $\mathrm{C}_{p}$ is the cyclic group of order $p$.  
\end{cor}
\begin{proof} Consider the exact sequence in Proposition \ref{equiv central ext} . Take the quotients of the middle and the next to the left terms by the subgroup $\mathrm{Ker}(\mathrm{Cont}_{\overline{\omega}})$, which is contained in the center, so that we get the desired exact sequence.
\end{proof}
By the construction, $\mathrm{H}(p,2g)$ is a quotient group of $F_{2g}=\langle \bx_1,\bx_2,\dots,\bx_{2g} \rangle$. We will add an extra generator $\bc$ together with the relation $\bc=[\bx_1,\bx_2] \equiv \bx_1\bx_2\bx_1^{-1}\bx_2^{-1}$.
\begin{prop}\label{presentation} We have the presentation
\begin{align*}
&\mathrm{H}(p,2g) 
\\
&= \left\langle \bx_1, \bx_2,\dots \bx_{2g}, \bc \;\middle|\; \bx_i^{p} \,(1\leq i\leq 2g),\; [\bx_i, \bx_j]\bc^{-\omega_{i,j} } \; (1\leq i,j\leq 2g),\; \bc \text{ is central.} \right\rangle
\end{align*}
, where 
\[
\omega_{i,j} = 
\begin{cases} j-i & \text{ if } |i-j|=1,
\\
0 & \text{ if otherwise}.
\end{cases}
\]
\end{prop}
\begin{proof} The exact sequence in Corollary \ref{central ext finite} implies that the order of $\mathrm{H}(p,2g)$ is equal to $p^{2g+1}$. On the other hand, it is readily seen that $\mathrm{H}(p,2g)$ is some quotient of the group represented on the R.H.S., which we denote by $G$ temporarily. It is sufficient to show that the order of $G$ is no more than $p^{2g+1}$. But using the relation appropriately, every element of $G$ can be represented by some monomial as 
\[
\bx_1^{n_1} \bx_2^{n_2} \dots \bx_{2g}^{n_{2g}} \bc^{n_{2g+1}} \quad  (0\leq n_i \leq p-1 ;\;(1\leq i\leq 2g+1) ).
\]
But the number of such monomials is $p^{2g+1}$. Thus we are done.
\end{proof}
\begin{cor} The map $\mathrm{p} : \mathrm{H}(p,2g) \overset{\mathrm{p}}{\rightarrow} A_{2g}$ in Corollary \ref{central ext finite} is the abelianization map of $\mathrm{H}(p,2g)$. The number of the conjugacy classes of $\mathrm{H}(p,2g)$ is equal to $p^{2g} -1 + p$.
\end{cor}
\begin{proof} The 1st assertion is readily seen. As for the 2nd, for any $\ba \in A_{2g}\backslash \{1\}$, $\mathrm{p}^{-1}(\ba)$ constitutes a conjugacy class of 
$\mathrm{H}(p,2g)$. On the other hand $\mathrm{p}^{-1}(1)$ constitutes the center $Z(\mathrm{H}(p,2g)) \cong C_p$, each element of which constitutes a conjugacy class. Therefore, we have $(p^{2g} -1) + p$ conjugacy classes.
\end{proof}
%
\begin{rem} We can see that the $\mathrm{Br}_{2g+1}$-action on $\mathrm{H}(p,2g)$ factors through $\mathrm{Sp}(2g;\bbf)$. But this statement is false if we replace $\mathrm{Br}_{2g+1}$ by $\mathcal{M}_{g,1}$
\end{rem}
%

\subsection{\bf Finite Fourier Analysis and the group algebra $\Z[\mathrm{H}(p,2g)]$}


For the moment, we will fix an odd prime $p$. Let $G_{2g}$ be $\mathrm{H}(p,2g)$. By the construction given in the previous subsection, we have the natural projection $\rho_{N} : F_{2g} = \langle \bx_1,\dots \bx_{2g} \rangle \twoheadrightarrow G_{2g}$, where $N=\mathrm{Ker}(\rho_{N})$. Recall that $N\,\triangleleft\, F_{2g}$ is $\mathcal{M}_{g,1}$-characteristic as we have seen before. Set $x_i := \rho_{N}(\bx_i) \; (1\leq  i\leq 2g)$, which generate $G_{2g}$.
\\
Since $G_{2g}$ is finite, its group $\C$-algebra $\mathrm{M}_{2g} := \C[G_{2g}]$ is semisimple and decomposes uniquely as the direct product of simple algebras. As is well-known, the number of semisimple components is equal to the number of the conjugacy classes of $G_{2g}$, which is equal to $p^{2g} -1 + p$ as we saw in the previous subsection. Let $\bc$ be the generator of the center $Z(G_{2g})\cong \mathrm{C}_{p}$ that appeared in the presentation in Proposition \ref{presentation}. Then we see that $\mathrm{M}_{2g}$ decomposes into the direct sum of non-zero ideals as
\[
\mathrm{M}_{2g} = \C[A_{2g}] \oplus \bigoplus_{\eta: \eta^p=1,\, \eta\neq 1} \mathrm{M}_{2g}^\eta
\]
, where we set
\[ 
\mathrm{M}_{2g}^\eta := \{ x \in \mathrm{M}_{2g} \mid \bc x=\eta x \}. 
\]
We have $p^{2g}$ $1$-dimensional group representations which factor through the abelianization $A_{2g}$. The corresponding $p^{2g}$ simple components span the ideal $\C[A_{2g}]$. Thus $p-1$ simple components remain. But the complement to $\C[A_{2g}]$ in $\mathrm{M}_{2g}$ decomposes as the direct sum of $p-1$ non-zero components. Therefore each component $\mathrm{M}_{2g}^\eta$ must be simple. The $\mathrm{Galois}(\Q(\zeta_p)/\Q)$-action on $\mathrm{M}^{2g}$ induces the one on the set of simple components, where $\Q(\zeta_p)$ is the cyclotomic field of degree $p$. Since $\{ \mathrm{M}_{2g}^\eta \mid \eta^p=1,\, p\neq 1\}$ forms an orbit, these $p-1$ components are isomorphic to each other as $\Q$-algebras. But they are isomorphic to the matrix algebras over $\C$ since $\C$ is algebraically closed. Thus they are isomorphic to each other as $\C$-algebras.
\\

For the moment we will fix a primitive p-th root of unity $\eta$ and focus on the simple component $\mathrm{M}_{2g}^\eta$ corresponding to $\eta$. Notice that the $\mathcal{M}_{g,1}$-action on $\mathrm{M}_{2g}$ preserves $\mathrm{M}_{2g}^\eta$ since it fixes $\bc$. We will introduce a basis of $\mathrm{M}_{2g}^\eta$, which behaves as the set of matrix units, that is, the square matrices each of which has zero entries except for one that is $1$. In other words, we will construct an explicit isomorphism between $\mathrm{M}_{2g}^\eta$ and a matrix algebra over $\C$. 
\\
Let $\pi_\eta : \mathrm{M}_{2g} \twoheadrightarrow \mathrm{M}_{2g}^\eta$ be the projection. We adopt the following convention;
\begin{conv} \textbullet\, Henceforth as an abuse of notation, we denote $\pi_\eta(x_i) \in \mathrm{M}_{2g}^\eta$, the image by $\pi_\eta$ of $x_i$, by the same symbol $x_i$ unless otherwise specified. 
\\
\textbullet\, An alphabet with hat, say $\hat{i}$, means the multi-index $(i_1,i_2,\dots,i_{g})$ of possibly $g$ entries, with each $i_l \in \bbf$. We often regard $\hat{i}$ as an element of the $\bbf$-vector space $\bbf^g$.
\\
\textbullet\, $\re_k$ for $1\leq k \leq g$ denotes the $k$th coordinate vector of $\bbf^g$. 
\end{conv}
With this understood, we will define a set of elements in $\mathrm{M}_{2g}^\eta$ as follows;
\begin{defi} For any $\hat{i}, \hat{j} \in (\bbf)^g$, set
\[ 
\bx_{\mathit{odd}}^ {\hat{i}} := 
x_1^{i_1} x_{3}^{i_2} \dots x_{2g-1}^{i_g},\quad \bx_{\mathit{even}}^ {\hat{i}} := 
x_2^{i_1} x_{4}^{i_2} \dots x_{2g}^{i_g} ,
\]
\[
 \bphi(\hat{i}; \bx_{\mathit{odd}}) :=
\frac{1}{p^g}\sum_{\hat{a} \in (\Z/(p))^g} \eta^{-\hat{i}\cdot\hat{a}} \,\bx_{\mathit{odd}}^ {\hat{a}} , 
\quad
\bphi(\hat{i}; \bx_{\mathit{even}}) :=
\frac{1}{p^g}\sum_{\hat{a} \in (\Z/(p))^g} \eta^{-\hat{i}\cdot\hat{a}} \,\bx_{\mathit{even}}^ {\hat{a}} .
\]
\end{defi}
\begin{prop} It is readily seen that
\begin{align*}
\sum_{\hat{a} \in (\Z/(p))^g} \bphi(\hat{i}; \bx_{\mathit{odd}})  = 1_{\mathrm{M}_{2g}^\eta} &= \sum_{\hat{a} \in (\Z/(p))^g} \bphi(\hat{i}; \bx_{\mathit{even}}) ,
\\
\bphi(\hat{i}; \bx_{\mathit{odd}}) \bphi(\hat{j}; \bx_{\mathit{odd}}) &= \delta_{\hat{i},\hat{j}}\bphi(\hat{i}; \bx_{\mathit{odd}}) ,
\\
\bphi(\hat{m}; \bx_{\mathit{even}}) \bphi(\hat{n}; \bx_{\mathit{even}}) &= \delta_{\hat{m},\hat{n}}\bphi(\hat{m}; \bx_{\mathit{even}}) .
\end{align*}
\end{prop}
\begin{defi} For any $\hat{i}, \hat{j} \in (\bbf)^g$, set 
\[ 
\mathrm{E}^{\eta}_{\hat{i}, \hat{j}}  := \bphi(\hat{i}; \bx_{\mathit{odd}}) \bphi (\hat{0}; \bx_{\mathit{even}}) \bphi (\hat{j}; \bx_{\mathit{odd}}) \; \in \; \mathrm{M}_{2g}^\eta .
\]
\end{defi}
%
%
\begin{defi} Define the square matrix $T=(T_{i,j})_{1\leq i,j\leq g}$ of degree $g$ as
\[
T_{i,j}  :=
\begin{cases}
1 & \text{ if } i-j=1,
\\ 0 & \text{ if otherwise.}
\end{cases}
\]
Set $\Omega := 1 + T + \cdots + T^{g-1}$. Notice that $T$ is nilpotent, that $1-T = \Omega^{-1}$ and that $\Omega= (\Omega_{i,j})_{1\leq i,j\leq g}$ is a lower triangular matrix all of whose entries on the diagonal and in the lower left part are equal to $1$;
\[ 
 \Omega_{i,j} = 
 \begin{cases}
1 & \text{ if } i \geq j,
\\ 0 & \text{ if otherwise.}
\end{cases}
\]
\end{defi}
%
%
\begin{prop}\label{matrix unit rewriting} It holds that
\[
\mathrm{E}^{\eta}_{\hat{i}, \hat{j}} =\frac{1}{p^g} \bx_{\mathit{even}}^{ \Omega (\hat{i} - \hat{j}) } \,\bphi(\hat{j}; \bx_{\mathit{odd}}).
\]
\end{prop}
Note that in the expression "$\Omega (\hat{i} - \hat{j})$", we regard $\hat{i}$ and $\hat{j}$ as column vectors.
\begin{proof}
\begin{align*}
&  \bphi(\hat{i}; \bx_{\mathit{odd}}) \bphi (\hat{0}; \bx_{\mathit{even}}) \bphi (\hat{j}; \bx_{\mathit{odd}}) 
\\
&=  \frac{1}{p^{2g}} \sum_{\hat{a} \in (\Z/(p))^g} \, \eta^{-\hat{i}\cdot\hat{a}} \,\bx_{\mathit{odd}}^ {\hat{a}} \,\sum_{\hat{b} \in (\Z/(p))^g} \, \bx_{\mathit{even}}^{\hat{b}}  \, \bphi (\hat{j}; \bx_{\mathit{odd}})
\\
&=  \frac{1}{p^{2g}} \sum_{\hat{a} \in (\Z/(p))^g} \,\sum_{\hat{b} \in (\Z/(p))^g} \, \eta^{ -\hat{i} \cdot \hat{a}  +  \hat{a} \cdot \hat{b} -  \hat{a} \cdot T\hat{b}  } \, \bx_{\mathit{even}}^{\hat{b}} \, \bx_{\mathit{odd}}^ {\hat{a}}  \, \bphi (\hat{j}; \bx_{\mathit{odd}})
\\
&=  \frac{1}{p^{2g}} \sum_{\hat{a} \in (\Z/(p))^g} \,\sum_{\hat{b} \in (\Z/(p))^g} \, \eta^{ \hat{a} \cdot (\hat{j}-\hat{i})  +  \hat{a} \cdot \hat{b} -  \hat{a} \cdot T\hat{b}  } \, \bx_{\mathit{even}}^{\hat{b}} \, \bphi (\hat{j}; \bx_{\mathit{odd}})
\\
&=  \frac{1}{p^{g}} \sum_{\hat{b} \in (\Z/(p))^g} \, \delta_{\hat{j}-\hat{i} + \hat{b} -T\hat{b} , \, 0 }  \, \bx_{\mathit{even}}^{\hat{b}} \, \bphi (\hat{j}; \bx_{\mathit{odd}})
\\
&=  \frac{1}{p^{g}}\, \bx_{\mathit{even}}^{ \Omega (\hat{j}-\hat{i}) } \, \bphi (\hat{j}; \bx_{\mathit{odd}})
\end{align*}
, where in the last equality we have used the fact that
\[
(1-T)\hat{b} = \hat{i}-\hat{j}  \;\Longleftrightarrow\;  \hat{b} =\Omega(\hat{i}-\hat{j} ) .
\]
\end{proof}
%
As a corollary of the previous $2$ propositions, a short calculation shows the following result, whose proof is safely left to the reader;
\begin{prop}\label{isom matrix} For any $\hat{i}, \hat{j}, \hat{k}, \hat{l} \in (\Z/(p))^{g}$, it holds that
$\mathrm{E}^{\eta}_{\hat{i}, \hat{j}} \mathrm{E}^{\eta}_{\hat{k}, \hat{l}} = \delta_{\hat{j},\hat{k}} \mathrm{E}^{\eta}_{\hat{i}, \hat{l}}\,$, where the $ \delta_{\hat{j},\hat{k}}$ is the Kronecker delta. Further, it holds that $1_{\mathrm{M}_{2g}^\eta} = \sum_{\hat{i} \in (\Z/(p))^g} \mathrm{E}^{\eta}_{\hat{i}, \hat{i}}$.
\end{prop}
\begin{cor} $\{ \mathrm{E}^{\eta}_{\hat{i}, \hat{j}} \mid \hat{i}, \hat{j} \in (\Z/(p))^{g} \}$ is a $\C$-basis of $\mathrm{M}_{2g}^\eta$.
\end{cor}
\begin{proof} The previous proposition gives rise to the surjective $\C$-algebra homomorphism $\mu: \mathrm{M}_{2g}^\eta \twoheadrightarrow \mathrm{Mat}(p^{2g};\C)$. But $\mathrm{dim}_{\C} \,\mathrm{M}_{2g}^\eta = (p^{2g+1} - p^{g} ) / (p-1) =p^{2g}$, which shows that $\mu$ is isomorphism. Thus we are done.
\end{proof}

Recall that the Skolem-Noether theorem states that any automorphism of a finite dimensional central simple algebra over $\C$ is an inner one. We will determine the inner automorphisms of $\mathrm{M}_{2g}^\eta$ corresponding to the action of $\Psi_l\; (1\leq l\leq 2g)$ on $\mathrm{M}_{2g}^\eta$.
\begin{nota*}[Quadratic Gauss Sum] We denote the quadratic Gauss sum w.\ r.\ t.\ the prime $p$ as 
\[
\mathrm{G}(a) := \sum_{t \in \Z/(p) } \,\eta^{at^2} \quad (a  \in \Z/(p)) .
\]
\end{nota*}
As is well-known, $\mathrm{G}(a) \cdot \mathrm{G}(-a) = |\mathrm{G}(a)|^2 =p$ for $a \not\equiv 0 \;\mathrm{mod}\; (p)$.
\\
Before proceeding to the next step, we put the following convention for the notational simplicity.
\begin{conv}[Ignoring $\re_{g+1}$ convention]\label{conv ignore} Throughout the present paper, whenever the symbol $\re_{g+1}$ appears, we will ignore it and presume $\re_{g+1}=0$ unless otherwise stated.
\end{conv}
\begin{prop}\label{braid action on matrix} For $1\leq k \leq g$ and for any $\hat{i}, \hat{j} \in (\Z/(p))^{g}$, it holds that
\begin{align*}
\Psi_{2k-1} ( \mathrm{E}^{\eta}_{\hat{i}, \hat{j}} ) &=
 \eta^{ -\binom{i_k+1}{2} + \binom{j_k+1}{2} } \,\mathrm{E}^{\eta}_{\hat{i}, \hat{j}} , \\
\Psi_{2k} ( \mathrm{E}^{\eta}_{\hat{i}, \hat{j}} ) &=
\frac{1}{p} \sum_{t,s \in \Z/(p)} \eta^{ -\frac{1}{2}(t+ \frac{1}{2})^2 + \frac{1}{2}(s + \frac{1}{2})^2 } \,\mathrm{E}^{\eta}_{ \hat{i} + t(\re_k -\re_{k+1}) , \,\hat{j} + s(\re_k -\re_{k+1}) } .
\end{align*}
\end{prop}
%
\begin{proof} 
\begin{align*}
\Psi_{2k-1}( \bphi ( \hat{0}; \bx_{\mathit{even}}) ) & = \frac{1}{p^g} \sum_{\hat{a} \in ( \Z/(p) )^g} \Psi_{2k-1}(\bx_{\mathit{even}}^ {\hat{a}} )
\\
& = \frac{1}{p^g}\sum_{\hat{a} \in ( \Z/(p) )^g} ( \bx_{2k-2}\bx_{2k-1} )^{a_{k-1}}  ( \bx_{2k-1}^{-1}\bx_{2k} )^{a_k}  \,\bx_2^{a_1} \dots \bx_{2k-4}^{a_{k-2}} \bx_{2k+2}^{a_{k+1}} \dots \bx_{2g}^{a_{g}}
\\ 
& = \frac{1}{p^g}\sum_{\hat{a} \in ( \Z/(p) )^g} \eta^{-a_{k-1}a_{k} + \binom{a_{k-1}}{2} + \binom{a_{k}}{2} } \,\bx_{2k-1}^{(a_{k-1}-a_{k})} \bx_{\mathit{even}}^ {\hat{a}}
\end{align*}
Therefore, it follows that
\begin{align*}
\Psi_{2k-1}( \mathrm{E}^{\eta}_{\hat{i}, \hat{j}}) &= \bphi(\hat{i}; \bx_{\mathit{odd}}) \Psi_{2k-1}(\bphi ( \hat{0}; \bx_{\mathit{even}}))  \bphi(\hat{j}; \bx_{\mathit{odd}}) 
\\
&=\frac{1}{p^g}\sum_{\hat{a} \in ( \Z/(p) )^g} \eta^{-a_{k-1}a_{k} + \binom{a_{k-1}}{2} + \binom{a_{k}}{2} +i_k(a_{k-1}-a_{k})}  \bphi(\hat{i}; \bx_{\mathit{odd}}) \bx_{\mathit{even}}^ {\hat{a}} \bphi(\hat{j}; \bx_{\mathit{odd}}) .
\end{align*}
Notice that any term in the sum above contributes trivially unless the multi-index $\hat{a}$ satisfies the condition that $\hat{a} = \Omega(\hat{i} - \hat{j})$, that is,
\[
a_l = (i_1-j_1)+\dots + (i_l-j_l) \quad (1\leq l \leq g).
\]
See the proof of Proposition \ref{matrix unit rewriting}.
If $\hat{a}$ meets this condition, the exponent to $\eta$ is
\begin{align*}
&-a_{k-1}a_{k} + \binom{a_{k-1}}{2} + \binom{a_{k}}{2} +i_k(a_{k-1}-a_{k})
\\
 =& \;\frac{1}{2}(a_k-a_{k-1})(a_k-a_{k-1} -2i_k-1)
\\
 =& \;\frac{1}{2}(i_k-j_k)((i_k-j_k) -2i_k-1)
\\
 =& \;-\binom{i_k+1}{2} + \binom{j_k+1}{2}
\end{align*}
Thus the 1st equation has been proven.\\
Next, we will prove the 2nd equation. Similar calculation shows the following;
\begin{align*}
\Psi_{2k}( \bphi ( \hat{i}; \bx_{\mathit{odd}}) ) & =
\frac{1}{p^g}\sum_{\hat{a} \in ( \Z/(p) )^g} \eta^{-\hat{i}\cdot \hat{a} +a_{k}a_{k+1} - \binom{a_k}{2} - \binom{{a_{k+\!1}} \,+\,1 }{2} } \, \bx_{\mathit{odd}}^ {\hat{a}} \bx_{2k}^{(a_{k}-a_{k+1})}
\\
&= \frac{1}{p^g}\sum_{\hat{a} \in ( \Z/(p) )^g} \eta^{-\hat{i}\cdot \hat{a} -a_{k1}a_{k+1} + \binom{a_k \,+\,1}{2} + \binom{a_{k+\!1}}{2} } \,\bx_{2k}^{(a_{k}-a_{k+1})} \bx_{\mathit{odd}}^ {\hat{a}}
\end{align*}
\begin{align*}
\Psi_{2k}( \mathrm{E}^{\eta}_{\hat{i}, \hat{j}}) & = \Psi_{2k}(\bphi(\hat{i}; \bx_{\mathit{odd}}) ) \,\bphi ( \hat{0}; \bx_{\mathit{even}})\,  \Psi_{2k}(\bphi (\hat{j}; \bx_{\mathit{odd} } ) )
\\
&= \frac{1}{p^{2g}}\sum_{\hat{a}, \hat{b} \in ( \Z/(p) )^g} \,\eta^{ -\hat{i}\cdot \hat{a} +a_{k}a_{k+1} - \binom{a_k}{2} - \binom{{a_{k+\!1}} \,+\,1 }{2} } \cdot \eta^{ -\hat{j}\cdot \hat{b} -b_{k}b_{k+1} + \binom{b_k \,+\,1}{2} + \binom{b_{k+\!1}}{2} }
\\
& \cdot \bx_{\mathit{odd}}^ {\hat{a}} \bphi ( \hat{0}; \bx_{\mathit{even}}) \bx_{\mathit{odd}}^ {\hat{b}}
\end{align*}
Thus we have obtained the following factorization;
\begin{align*}
&\bphi (\hat{m}; \bx_{\mathit{odd}}) \Psi_{2k}( \mathrm{E}^{\eta}_{\hat{i}, \hat{j}}) \bphi (\hat{n}; \bx_{\mathit{odd}} )
\\
&= \{
\frac{1}{p^{g}}\sum_{\hat{a}\in ( \Z/(p) )^g} \,\eta^{ (\hat{m}-\hat{i}) \cdot \hat{a} +a_{k}a_{k+1} - \binom{a_k}{2} - \binom{  a_{k+\!1}  \,+\,1 }{2} }
\}
\\
& \cdot
\{ 
\frac{1}{p^{g}}\sum_{\hat{b} \in ( \Z/(p) )^g} \eta^{(\hat{n}-\hat{j})\cdot \hat{b} -b_{k}b_{k+1} + \binom{b_k \,+\,1}{2} + \binom{b_{k+\!1}}{2} } 
\} 
\cdot\,\mathrm{E}^{\eta}_{\hat{m}, \hat{n}}
\end{align*}
We will calculate each factor on the R.H.S. separately using the sublemma below;
%
\begin{slem*} Define the quadratic form $Q$ on $(\Z/(p))^{\oplus 4}$ as follows;
\[
\mathrm{Q}(\alpha,\beta; \lambda, \mu) :=\frac{\alpha(\alpha-1)}{2} +\frac{\beta(\beta+1)}{2} - \alpha\beta -\lambda\alpha - \mu\beta .
\]
Then for any $(\lambda, \mu) \in (\Z/(p))^{\oplus 2}$ it holds that 
\[
\sum_{\alpha,\beta \in \Z/(p) } \eta^{\pm\mathrm{Q}(\alpha,\beta; \,\lambda.\mu ) }
= p\mathrm{G}(\pm\frac{1}{2}) \delta_{\lambda+\mu, 0} \,\eta^{\mp \frac{1}{2}(\lambda + \frac{1}{2})^2} \quad (\text{double-sign corresponds}).
\]
\end{slem*}
\begin{proof}
\begin{align*}
 \text{L.H.S.}
&= \sum_{\alpha,\beta \in \Z/(p)} \eta^{\pm \left(\frac{1}{2}(\alpha-\beta-\lambda-\frac{1}{2})^2 - (\lambda+\mu)\beta - \frac{1}{2}(\lambda + \frac{1}{2})^2 \right) }
\\
&= \mathrm{G}(\pm\frac{1}{2})  \sum_{\beta \in \Z/(p)} \eta^{ \mp \left( (\lambda+\mu)\beta + \frac{1}{2}(\lambda + \frac{1}{2})^2 \right) }
= \text{R.H.S}
\end{align*}
\end{proof}
%
\t
Returning to the proof of the present proposition,
\begin{align*}
\text{ 1st factor }
&=\prod_{l\neq k, k+1} \delta_{m_l, i_l}
\cdot p^{-2} \sum_{\alpha,\beta \in \Z/(p) } \eta^{-\frac{\alpha(\alpha - 1)}{2} -\frac{\beta(\beta + 1)}{2} +\alpha\beta +(m_{k} - i_{k})\alpha +(m_{k+1} - i_{k+1})\beta }
\\
&= \prod_{l\neq k, k+1} \delta_{m_l, i_l} \cdot p^{-2} \sum_{\alpha,\beta \in \Z/(p)} \eta^{-\mathrm{Q}(\alpha,\beta;\, m_{k} - i_{k}, m_{k+1} - i_{k+1}) } ,
\\
&= \frac{1}{p} \mathrm{G}(\frac{-1}{2}) (\prod_{l\neq k, k+1} \delta_{m_l, i_l}) \cdot  \delta_{m_{k} - i_{k} + m_{k+1} - i_{k+1} ,0} \,\eta^{\frac{1}{2}(m_{k} - i_{k} + \frac{1}{2})^2 }
\end{align*}
\\
\t
Similarly, 
\begin{align*}
\text{2nd factor}
&= \prod_{l \neq k, k+1} \delta_{n_l, j_l} \cdot p^{-2} \sum_{\alpha,\beta\in \Z/(p)} \eta^{\mathrm{Q}(\alpha,\beta;\, -(n_{k+1} - j_{k+1}), -(n_{k} - j_{k}) ) }
\\
&= \frac{1}{p} \mathrm{G}(\frac{1}{2}) (\prod_{l\neq k, k+1} \delta_{n_l, j_l})\cdot \delta_{n_{k + 1} - j_{k + 1} + n_{k} - j_{k}, 0 } \, \eta^{-\frac{1}{2}(n_{k + 1}-j_{k + 1}-\frac{1}{2})^2 }
\\
&= \frac{1}{p} \mathrm{G}(\frac{1}{2}) (\prod_{l\neq k, k+1} \delta_{n_l, j_l})\cdot \delta_{ n_{k} - j_{k} + n_{k + 1} - j_{k + 1}, 0 } \, \eta^{-\frac{1}{2}(n_{k}-j_{k} + \frac{1}{2})^2 }
\end{align*}
\t
The product of the 1st and the 2nd factor is equal to the coefficient of
$\mathrm{E}_{\hat{m}, \hat{n}}$ in $\Psi_{2k}( \mathrm{E}_{\hat{i}, \hat{j}})
$. Thus rewriting the consequence, the result follows.
\end{proof}
%
%
%
\subsection{ The tensor product decomposition of $M_{2g}^{\eta}$ }

With Proposition \ref{braid action on matrix} in mind, in the present subsection, 
we will introduce the new $\C$-vector spaces, $V^{\eta}_{g}$ and its dual $V_{g}^{\eta*}$, endowed with $B_{2g}$-action.
\begin{defi}
Let $V_g^{\eta}$ be the $\C$-vector space spanned by the basis $\{ \be_{\hat{j}} \mid \hat{j} \in (\Z/(p))^g \}$. The right $G_{2g}$-module structure is determined by the rule 
\[ 
\be_{\hat{j}} x_{2k-1} := \eta^{j_k} \be_{\hat{j}}, \quad \be_{\hat{j}} x_{2k} = \be_{\hat{j} - \re_k + \re_{k+1} } \quad (1\leq k \leq g) .
\]
(Due to Convention \ref{conv ignore}, if $k=g$, we understand that $\be_{\hat{j}} x_{2g} = \be_{\hat{j} - \re_g }$.)  Define the left action of the symbol $\Psi_l \; (1\leq l \leq 2g)$ on $V_g$ by the following rule; 
\[ 
\Psi_{2k-1} (\be_{\hat{j}}) = \eta^{\binom{j_k+1}{2}} \be_{\hat{j}}, \quad  \Psi_{2k} (\be_{\hat{j}}) = \frac{\mathrm{G}(\frac{1}{2})}{p} \sum_{s \in \Z/(p)} \eta^{- \frac{1}{2}(s + \frac{1}{2})^2} \be_{\hat{j} + s(\re_{k} - \re_{k + 1}) }
\]
for $1\leq k \leq g$ and for any $\hat{j} \in (\Z/(p))^g$.
\end{defi}
It can be readily seen that the multiplication rule above surely reduces to a right $G_{2g}$-module structure on $V_{g}^{\eta}$. Recall that $\{ \bx_m \mid 1\leq m\leq 2g \}$ generates $G_{2g}$. This right $G_{2g}$-module structure and the left $\Psi_l$-action on $V_{g}^{\eta}$ are compatible in the sense that 
\[ 
\Psi_l(\be_{\hat{j}} x_m) = \Psi_l(\be_{\hat{j}}) \Psi_t( x_m) \quad ( \hat{j} \in (\Z/(p))^{g},\; l,m \in \{ 1,2,\dots,2g\}). 
\]
%
%
\begin{defi}\label{v action}
Let $V_g^{\eta*}$ be the $\C$-vector space spanned by the dual basis  
$\{ {\be}^{\prime}_{\hat{i}} \mid \hat{i} \in (\Z/(p))^g \}$. The induced left $G_{2g}$-module structure are given by the rule
\[ 
x_{2k-1} {\be}^\prime_{\hat{i}} = \eta^{i_k} {\be}^\prime_{\hat{i}} \,, \quad x_{2k} {\be}^\prime_{\hat{i}} = {\be}^\prime_{ \hat{i} + \re_k - \re_{k + 1} } \quad (1\leq k \leq g).
\]
The induced left action of the symbol $\Psi_l$ on $V^{\eta*}_g$ is determined as follows;
\[ 
\Psi_{2k-1} (\be^\prime_{\hat{i}}) = \eta^{-\binom{i_k+1}{2}} {\be}^\prime_{\hat{i}} \,, \quad  \Psi_{2k} ({\be}^\prime_{\hat{i}}) = \frac{\mathrm{G}(-\frac{1}{2})}{p} \sum_{t \in \Z/(p)} \eta^{ \frac{1}{2}(t + \frac{1}{2})^2}   {\be}^\prime_{\hat{i} + t(\re_k - \re_{k+1}) }
\]
for $1\leq k \leq g$ and for any $\hat{i} \in (\Z/(p))^g$.
\end{defi}
By definition, we have the canonical $\C$-bilinear pairing $\langle \cdot,\cdot \rangle : V^{\eta}_{g} \times V^{\eta*}_{g} \to \C$ such that 
\[
\langle \be_{\hat{i}}, {\be}^\prime_{\hat{j}} \rangle = \delta_{\hat{i},\hat{j}} \quad (\,\hat{i},\hat{j} \in (\Z/(p))^g\,) .
\]
Further, for any $(v,w) \in V_{g} \times V^{\eta*}_{g}$, it holds that
\begin{align*}
& \langle v, hw\rangle = \langle vh, w \rangle \;\; (h \in G) , 
\quad \langle \Psi_l(v), \Psi_l(w) \rangle = \langle v,w \rangle \;\; (1\leq l \leq 2g) .
\end{align*}
Consider the tensor product $V^{\eta*}_g \otimes_{\C} V^{\eta}_g$, which has the induced $\C[G_{2g}]$-bimodule structure. We have the canonical $\C$-linear isomorphism  
\[
\iota : V^{\eta*}_g \otimes_{\C} V^{\eta}_g \longrightarrow \mathrm{M}_{2g}^{\eta} \; :\; \be_{\hat{i}}^\prime \otimes \be_{\hat{j}} \longmapsto \mathrm{E}^{\eta}_{\hat{i}, \hat{j}}.
\]
, which turns out to be a $\C[G_{2g}]$-bimodule map. Further it holds that
\[
\iota\left(\Psi_{l}(u) \otimes \Psi_{l}(v)\right) = \Psi_{l}( \iota(u\otimes v)) \quad (u\otimes v \in V^{\eta*}_g\otimes_{\C} V^{\eta}_g,\; 1\leq l \leq 2g).
\]
In other words, the map $\iota$ intertwines the $\Psi_{l}$-actions on both sides, the diagonal one on $V^{\eta*}_g\otimes_{\C} V^{\eta}_g$ and the usual one on $\mathrm{M}_{2g}^{\eta}$. Thus we have
\begin{prop}\label{1}
$\iota : V^{\eta*}_g\otimes_{\C} V^{\eta}_g \to \mathrm{M}_{2g}^{\eta}$ is an isomorphism both as $\C[G_{2g}]$-bimodules and as $B_{2g}$-modules.
\end{prop}
%
Three remarks here are in order.
\\
\textbullet\, First, a short calculation shows that the action of $\Psi_{l}\; (1\leq l \leq 2g)$ on $\mathrm{M}^{2g}_\eta$ coincides with the adjoint action of the invertible element 
\[
\frac{\mathrm{G}(-\frac{1}{2})}{p} \sum_{t \in \Z/(p)} \eta^{ \frac{1}{2}(t + \frac{1}{2})^2} \pi_{\eta}(x_{l})^{t} \in \mathrm{M}_{2g}^{\eta}.
\]
\\
%
\textbullet\, Secondly, since the map determined by the adjoint action  
\[
\mathrm{adj} : \mathrm{Unit}(\mathrm{M}_{2g}^{\eta}) \to \mathrm{Aut}_{\C\!-\!\mathrm{alg}}(\mathrm{M}_{2g}^{\eta})
\]
has non trivial kernel $\mathrm{Unit}(\mathrm{M}_{2g}^{\eta}) \cap \mathrm{Cent}(\mathrm{M}_{2g}^{\eta}) \cong \C^{*}$, it is not obvious at a first glance whether the action given in Definition \ref{v action} amounts to a (true) $\mathrm{Br}_{2g+1}$-action or merely a projective representation. But the former is the case as shown in the next proposition.
\\
\textbullet\, Thirdly, both $V^{\eta}_g$ and $V^{\eta*}_g$ afford projective representaions over $\C$ of $\mathcal{M}_{g,1}$, which extend the $\mathrm{Br}_{2g+1}$-actions above, since $\mathcal{M}_{g,1}$ acts on $\mathrm{M}_{2g}^{\eta}$ as $\C$-algebra automorphisms.
\begin{prop}
The action of $\Psi_l \quad (1\leq l \leq 2g)$ on $V^{\eta}_g$ satisfies the defining relation of $\mathrm{Br}_{2g}$.
\end{prop}
\begin{proof} A short calculation shows that, for any $\hat{j} \in (\Z/(p))^g$ and for $1\leq k \leq g$,
\begin{align*}
\Psi_{2k-1} \Psi_{2k} \Psi_{2k-1}(\be_{\hat{j}}) &= 
\frac{\mathrm{G}(\frac{1}{2})}{p} \sum_{s \in \Z/(p)} \eta^{ -\frac{1}{2}(s + \frac{1}{2})^2 + \binom{j_{k} +1 + s}{2} + \binom{j_{k} +1}{2} } \be_{\hat{j} + s(\re_{k} -\re_{k+1}) }
\\
&= \frac{\mathrm{G}(\frac{1}{2})}{p} \sum_{s \in \Z/(p)} \eta^{ j_{k}(s  + j_{k} + 1) - \frac{1}{8} } \,\be_{\hat{j} + s(\re_{k} -\re_{k+1}) } .
\end{align*}
On the other hand, we see that
\begin{align*}
\Psi_{2k} \Psi_{2k-1} \Psi_{2k}(\be_{\hat{j}}) &= 
\frac{\mathrm{G}(\frac{1}{2})^2}{p^2} \sum_{s,t \in \Z/(p)} \eta^{ -\frac{1}{2}(s + \frac{1}{2})^2  - \frac{1}{2}(t + \frac{1}{2})^2 + \binom{j_{k} +1 + s}{2}  } \be_{\hat{j} + (s+t)(\re_{k} -\re_{k+1}) } 
\\
&= \frac{\mathrm{G}(\frac{1}{2})^2}{p^2} \sum_{s,t \in \Z/(p)} \eta^{ -\frac{1}{2} \{  -2j_{k}s + t^2 + t - j_{k}^2 - j_{k} + \frac{1}{2} \} } \,\be_{\hat{j} + (s+t)(\re_{k} -\re_{k+1}) }
\\
&= \frac{\mathrm{G}(\frac{1}{2})^2}{p^2} \sum_{s,t \in \Z/(p)} \eta^{ -\frac{1}{2} \{  -2j_{k}(s-t) + t^2 + t - j_{k}^2 - j_{k} + \frac{1}{2} \} } \,\be_{\hat{j} + s(\re_{k} -\re_{k+1}) }
\\
&=\frac{\mathrm{G}(\frac{1}{2})^2}{p^2} \sum_{s,t \in \Z/(p)} \eta^{ -\frac{1}{2} \{ (t + j_{k} + \frac{1}{2})^2   - 2j_{k}s  -2j_{k}^2  - 2j_{k} + \frac{1}{4} \} } \,\be_{\hat{j} + s(\re_{k} -\re_{k+1}) }
\\
&= \frac{\mathrm{G}(\frac{1}{2})}{p} \sum_{s \in \Z/(p)} \eta^{ j_{k}s  + j_{k}^2 + j_{k} - \frac{1}{8} } \,\be_{\hat{j} + s(\re_{k} -\re_{k+1}) }.
\end{align*}
Thus we have shown that $\Psi_{2k-1} \Psi_{2k} \Psi_{2k-1} = \Psi_{2k} \Psi_{2k-1} \Psi_{2k}$ holds on $V^{\eta}_{g}$. \\
Similarly, a short calculation shows that, for any $\hat{j} \in (\Z/(p))^g$ and for $1\leq k \leq g-1$,
\begin{align*}
\Psi_{2k+1} \Psi_{2k} \Psi_{2k+1}(\be_{\hat{j}}) &= 
\frac{\mathrm{G}(\frac{1}{2})}{p} \sum_{s \in \Z/(p)} \eta^{ -(j_{k+1} + 1)s  + j_{k+1}^2 + j_{k+1} - \frac{1}{8} } \be_{\hat{j} + s(\re_{k} -\re_{k+1}) }
\\
&=\Psi_{2k} \Psi_{2k+1} \Psi_{2k}(\be_{\hat{j}}) .
\end{align*}
Thus we have shown that $\Psi_{2k+1} \Psi_{2k} \Psi_{2k+1} = \Psi_{2k} \Psi_{2k+1} \Psi_{2k}$ holds on $V^{\eta}_{g}$. The remaining braid relations can be ensured more easily and left to the reader. 
\end{proof}
\begin{cor}\label{psi order p} The action of $\Psi_l \quad (1\leq l \leq 2g)$ on $V^{\eta}_g$ satisfies
\[
(\Psi_{l}|_{V^{\eta}_g})^p= \mathrm{id}_{V^{\eta}_g} .
\]
\end{cor}
\begin{proof}  If $l=2k-1\; (1\leq k \leq g)$, the assertion follows immediately from Definition \ref{v action}. If $l=2k\; (1\leq k \leq g)$, the braid relation implies that the action of $\Psi_{2k}$ is conjugate to that of $\Psi_{2k-1}$ in $\mathrm{End}_{\C}(V^{\eta}_{g})$. Thus we are done.
\end{proof}
%
%
We will provide a version of finite Fourier transformation needed in the subsequent part of the present paper.
\begin{defi}[finite Fourier transformation] For any $\hat{n} \in (\Z/(p))^g$, set
\[
 \be^*_{\hat{n}} := \alpha_g \! \sum_{\hat{j} \in (\Z/(p))^g} \eta^{\hat{n} \cdot \Omega \hat{j} } \,\be_{\hat{j}} \;\;\in\; V^{\eta}_{g}
\]
, where we set the constant $\alpha_g:= (\frac{ \mathrm{G}(\frac{1}{2}) }{p})^g$. 
\end{defi}
Then the inverse Fourier transformation takes the following form;
\begin{prop} For any $\hat{j} \in (\Z/(p))^g$, it holds that
\[
 \be_{\hat{j}} =  \bar{\alpha}_g \sum_{\hat{n} \in (\Z/(p))^g} \eta^{-\hat{n} \cdot \Omega \hat{j} } \,\be^*_{\hat{n}} .
\]
Notice that $\alpha_g \bar{\alpha}_g = \frac{1}{p^g}$ since $\bar{\alpha}_g = (\frac{ \mathrm{G}(-\frac{1}{2}) }{p})^g$.
\end{prop}
\begin{prop}\label{psi even act star} For $1\leq k \leq g$ and for any $\hat{n} \in (\Z/(p))^g$, it holds that
\[
\Psi_{2k} (\be^*_{\hat{n}}) = \eta^{\binom{n_k+1}{2}} \,\be^*_{\hat{n}} \,.
\]
\end{prop}
\begin{proof}
\begin{align*}
\Psi_{2k} (\be^*_{\hat{n}}) &= \alpha_{g} \sum_{\hat{j} \in (\Z/(p))^g} \,\eta^{\hat{n} \cdot \Omega \hat{j} } \,\Psi_{2k}(\be_{\hat{j}})
\\
&=  \frac{\mathrm{G}(\frac{1}{2})}{p} \, \alpha_{g} \sum_{s \in \Z/(p)} \, \sum_{\hat{j} \in (\Z/(p))^g} \,\eta^{\hat{n} \cdot \Omega \hat{j} \, -\frac{1}{2}(s + \frac{1}{2})^2 } \, \be_{\hat{j} + s(\re_{k} -\re_{k+1})}
\\
&=  \frac{\mathrm{G}(\frac{1}{2})}{p} \, \alpha_{g}  \sum_{s \in \Z/(p)} \, \sum_{\hat{j} \in (\Z/(p))^g} \,\eta^{\hat{n} \cdot \Omega \, \{ \hat{j} \, - s(\re_{k} -\re_{k+1}) \} -\frac{1}{2}(s + \frac{1}{2})^2 } \, \be_{\hat{j}}
\\
&=  \frac{\mathrm{G}(\frac{1}{2})}{p} \, \alpha_{g} \sum_{s \in \Z/(p)} \, \sum_{\hat{j} \in (\Z/(p))^g} \,\eta^{\hat{n} \cdot \Omega \hat{j} \, -n_{k}s -\frac{1}{2}(s + \frac{1}{2})^2 } \, \be_{\hat{j}}
\\
&=  \frac{\mathrm{G}(\frac{1}{2})}{p} \, \alpha_{g} \sum_{s \in \Z/(p)} \, \sum_{\hat{j} \in (\Z/(p))^g} \,\eta^{ \hat{n} \cdot \Omega \hat{j} \, -\frac{1}{2}(s + \frac{1}{2} + n_{k})^2 + \frac{1}{2} n_{k}(n_{k} + 1) } \, \be_{\hat{j}}
\\
&=  \alpha_{g} \sum_{\hat{j} \in (\Z/(p))^g} \,\eta^{ \hat{n} \cdot \Omega \hat{j} \, + \frac{1}{2} n_{k}(n_{k} + 1) } \, \be_{\hat{j}}
\\
&= \eta^{\frac{1}{2} n_{k}(n_{k} + 1) } \, \be^{*}_{\hat{n}}
\end{align*}
\end{proof}
%
%
%
%
\subsection{\bf The $\mathrm{Br}_{2g+1}$-module $L_g^{\eta}$} 
%
We have the canonical identifications as $\C$-vector space as follows;
\[
L_{N}^{+} \cong \C[G_{2g}] \otimes_{\C} \langle \bu_{1}, \bu_{2}, \dots, \bu_{2g}, \bw_0 \rangle_{\C}, \quad L_{N} \cong \C[G_{2g}] \otimes_{\C} \langle f_{1}, f_{2}, \dots, f_{2g} \rangle_{\C} .
\]
Keeping these in mind, we introduce the new $\mathrm{Br}_{2g+1}$-module $L_g^{\eta}$ associated with the latter.
\begin{defi} Define the $\C$-vector space $L_g^{\eta}$ as follows;
\[
L_g^{\eta} := V^{\eta}_{g} \otimes_{\C} \langle \bu_{i} \mid 1\leq i \leq 2g \rangle_{\C}
\]
Define the action of the symbol $\hat{\Psi}_{l} \; (1\leq l \leq 2g)$ on $L_g^{\eta}$ by the formula
\begin{align*}
\hat{\Psi}_{2k-1} (v \otimes \bu_{l})  &= \Psi_{2k-1}(v) \otimes \bu_{l}
\\
 &  - \delta_{l,2k} \,  \Psi_{2k-1}(v) \{ x_1 x_3 \dots x_{2k-3} \cdot x_{2k} x_{2k+2}\dots x_{2g} \} \otimes (\bu_{2k-1} - \bu_{2k-3}) ,
\\
\hat{\Psi}_{2k} (w \otimes \bu_{l})  &= \Psi_{2k}(w) \otimes \bu_{l}
\\
 &  +  \delta_{l,2k-1} \,  \Psi_{2k}(w) \{ x_1 x_3 \dots x_{2k-1} \cdot x_{2k} x_{2k+2}\dots x_{2g} \}^{-1} \otimes (\bu_{2k} - \bu_{2k+2}) .
\end{align*}
\end{defi}
\begin{defi} Set 
\[
\bphi(\bc;\eta) := \frac{1}{p}\sum_{n \in \Z/(p)} \eta^{-n}\, \bc^n \in \C[G_{2g}]
\]
, where $\bc$ is the generator of the center of $G_{2g} \equiv \mathrm{H}(p, 2g)$ that appeared in the presentation in Proposition \ref{presentation}. Recall that $\bc$ is invariant under the $\mathrm{Br}_{2g+1}$-action on $G_{2g}$.
\end{defi}
\begin{prop} We have a canonical $\mathrm{Br}_{2g+1}$-module isomorphism
\[
\bphi(\bc;\eta)L_{N} \cong \mathrm{M}_{2g}^{\eta} \otimes_{\C} \langle \bu_{l} \mid 1 \leq l \leq 2g \rangle_{\C} .
\]
\end{prop}
\begin{proof} Consider the self map of $L_{N}^{+}$ determined by the left multiplication by $\triangle \in G_{2g}^{+}$; 
\[
\mathrm{L}_{\triangle} : L_{N}^{+} \to L_{N}^{+} \; : \; w \mapsto \triangle w . 
\]
Then we have the induced $\C$-linear bijective map
\[
\mathrm{L}_{\triangle}|_{\iota(L_{N})} : \iota(L_{N}) \overset{\cong}{\longrightarrow} \C[G_{2g}] \otimes_{\C} \langle \bu_{l} \mid 1 \leq l \leq 2g \rangle_{\C} \subset L_{N}^{+}
\]
, which turns out to be a $\mathrm{Br}_{2g+1}$-module isomorphism since $\triangle$ is invariant under the $\mathrm{Br}_{2g+1}$-action. It follows that
\begin{align*} 
\bphi(\bc;\eta) L_{N} & \cong \bphi(\bc;\eta) \C[G_{2g}] \otimes_{\C} \langle \bu_{l} \mid 1 \leq l \leq 2g \rangle_{\C}
\\
& \cong \mathrm{M}_{2g}^{\eta} \otimes_{\C} \langle \bu_{l} \mid 1 \leq l \leq 2g \rangle_{\C} 
\end{align*}
, where the congruences mean "isomorphic as $\mathrm{Br}_{2g+1}$-modules".  At the 1st congruence, we have used the fact $\bc$ and $\triangle$ commute with each other.
\end{proof}
\begin{thm}\label{canonical isom} The action of $\hat{\Psi}_{l} \; (1\leq l \leq 2g)$ on $L_g^{\eta}$ amounts to a $\mathrm{Br}_{2g+1}$-action. Further, the tensor product $V_{g}^{\eta*} \otimes_{\C} L_g^{\eta}$ endowed with the diagonal $\mathrm{Br}_{2g+1}$-action is canonically isomorphic to $\bphi(\bc;\eta)L_{N}$ as a $\mathrm{Br}_{2g+1}$-module.   
\end{thm}
\begin{proof} The composition of the canonical $\C$-linear bijective maps
\begin{align*}
V_{g}^{\eta*} \otimes_{\C} L_g^{\eta} & \;\cong\; V_{g}^{\eta*} \otimes_{\C} V_{g}^{\eta}  \otimes_{\C} \langle \bu_{l} \mid 1 \leq l \leq 2g \rangle_{\C} 
\\
& \;\cong\; \mathrm{M}_{2g}^{\eta} \otimes_{\C} \langle \bu_{l} \mid 1 \leq l \leq 2g \rangle_{\C}
\\
& \;\cong\; \bphi(\bc;\eta) L_{N}
\end{align*}
intertwines the action of $\hat{\Psi}_{l}$ on the L.H.S. and that on the R.H.S., giving rise to a $\mathrm{Br}_{2g+1}$-module isomorphism. But since the  $\hat{\Psi}_{l}$-action on the 1st factor $V_{g}^{\eta*}$ of the L.H.S. (via $\Psi_{l}$-action) amounts to a $\mathrm{Br}_{2g+1}$-action, that on the 2nd factor $L_g^{\eta}$ of the L.H.S. reduces to a $\mathrm{Br}_{2g+1}$-action. Thus we are done.
\end{proof}
\begin{thm}\label{funda jordan new}  For $1\leq l\leq 2g$, the semisimple and the unipotent parts of the $\hat{\Psi}_{l}$-action on $L_g^{\eta}$ are respectively  described by the very same formulae as in Theorem \ref{funda formulae jordan} with the modified assumption that $h \in V_{g}^{\eta}$ replacing the one that $h \in \Z[G_{2g}]$ there.
\end{thm}
\begin{proof} Since the $\Psi_l$-action on $V_{g}^{\eta}$ satisfies the condition $(\Psi_l|_{V_{g}^{\eta}})^{p} = \mathrm{id}_{V_{g}^{\eta}}$ (See Corollary \ref{psi order p}), exactly the same reasoning that deduced Theorem \ref{funda formulae jordan} shows the result.
\end{proof}
\begin{rem} In the same spirit as the argument so far, define the $\C$-vectors space $L_g^{\eta\prime}$ by
\[
L_g^{\eta\prime} := V_{g}^{\eta*} \otimes_{\C} \langle f_i \mid 1\leq i \leq 2g \rangle_{\C}.
\]
We can endow $L_g^{\eta\prime}$ with a projective $\mathcal{M}_{g,1}$-module structure over $\C$ such that it is accompanied by a canonical $\mathcal{M}_{g,1}$-module isomorphism
\[
\xi : V_{g}^{\eta*} \otimes_{\C} L_g^{\eta\prime} \overset{\cong}{\longrightarrow} \bphi(\bc;\eta) L_{N} .
\]
Further, there exists a canonical $\mathrm{Br}_{2g+1}$-module isomorphism $\tau : L_g^{\eta\prime} \overset{\cong}{\longrightarrow} L_g^{\eta}$ such that the composition $(\mathrm{id}_{V_{g}^{\eta*}} \otimes \tau) \circ \xi^{-1}$ coincides with the isomorphism given in Proposition \ref{canonical isom}.
\end{rem}

%

\subsection{\bf A concrete description of the $B_{2g}$-action on $L_g^{\eta}$ : unipotent  part }
We will describe the unipotent part of the action of $\hat{\Psi}_i \; (1\leq i \leq 2g)$ on $L_g^{\eta}$.
\begin{thm}\label{psi hat unipotent action} For $1\leq k \leq g,\; 1\leq m \leq 2g$ and for any $\hat{i} \in (\Z/(p))^g$, it holds that
\begin{align*}
& \left(\hat{\Psi}_{2k-1,\mathrm{uni}} -1 \right) (\be_{\hat{i}}\bu_{m}) 
= -\delta_{m, 2k}\, \delta_{i_k,0} \,\eta^{ ( i_1+\dots +i_{k-1} ) } \,\be_{\hat{i}-\re_k } (\bu_{2k-1} - \bu_{2k-3}),
\\
& \left( \hat{\Psi}_{2k,\mathrm{uni}} - 1 \right) (\be_{\hat{i}}\bu_{m}) 
 = \delta_{m, 2k-1}\, \frac{1}{p}\, \eta^{-(i_1 + \dots + i_k)} \,\sum_{s\in \Z/(p)} \eta^{-s} \,\be_{\hat{i} +s\re_k +(1-s)\re_{k+1}}  (\bu_{2k} - \bu_{2k+2}). 
\end{align*}
\end{thm}
\begin{proof} At first we will show the 1st equation. 
\\
If $m\neq 2k$, it follows from Theorem \ref{funda formulae} that $\hat{\Psi}_{2k-1} (\be_{\hat{i}}\bu_{m}) = \Psi_{2k-1} (\be_{\hat{i}}) \bu_{m}$. But the action of $\Psi_{2k-1}$ on $V_{g}$ is semisimple. Thus the nilpotent part of its unipotent part contributes trivially, that is, the L.H.S. of the 1st equation is equal to zero. 
\\
If $m=2k$, it follows from Theorem \ref{funda jordan new} that 
\begin{align*}
&\left( \hat{\Psi}_{2k-1,\mathrm{uni}} -1 \right) (\be_{\hat{i}}\bu_{2k}) 
\\
&= -\be_{\hat{i}} \big\{ x_1 x_3 \dots x_{2k-3} \bphi_{p}(x_{2k-1}) x_{2k} x_{2k+2}\dots x_{2g} \big\} (\bu_{2k-1} - \bu_{2k-3}),
\\
&= -\eta^{(i_1 +\dots + i_{k-1})} \,\be_{\hat{i}} \big\{ \bphi_{p}(x_{2k-1}) x_{2k} x_{2k+2}\dots x_{2g} \big\} (\bu_{2k-1} - \bu_{2k-3})
\\
&= -\delta_{i_k,0} \,\eta^{(i_1 +\dots + i_{k-1})} \,\be_{\hat{i}} \big\{ x_{2k} x_{2k+2}\dots x_{2g} \big\} (\bu_{2k-1} - \bu_{2k-3})
\\
&= -\delta_{i_k,0} \,\eta^{(i_1 +\dots + i_{k-1})} \,\be_{\hat{i} -\re_k} (\bu_{2k-1} - \bu_{2k-3}).
\end{align*}
Then we will show the 2nd equation. If $m\neq 2k-1$, it follows from Theorem \ref{funda formulae} that $\hat{\Psi}_{2k} (\be_{\hat{i}}\bu_{m}) = \Psi_{2k}(\be_{\hat{i}})\bu_{m}$. Since $\Psi_{2k}$ is semisimple, the nilpotent part of the unipotent part is equal to zero. 
\\
If $m = 2k-1$, it follows from Theorem \ref{funda jordan new} that
\begin{align*}
& \left( \hat{\Psi}_{2k,\mathrm{uni}} -1 \right) (\be_{\hat{i}}\bu_{2k-1}) 
\\
& =\be_{\hat{i}} \big\{ x_{2g}^{-1} x_{2g-2}^{-1} \dots x_{2k+2}^{-1} x_{2k}^{-1} \bphi _{p}(x_{2k}) x_{2k-1}^{-1} x_{2k-3}^{-1} \dots x_{3}^{-1} x_{1}^{-1} \big\} (\bu_{2k} - \bu_{2k+2})
\\
& = \be_{\hat{i} + \re_{k+1}} \big\{ \bphi_{p}(x_{2k}) x_{2k-1}^{-1} x_{2k-3}^{-1} \dots x_{3}^{-1} x_{1}^{-1} \big\}  (\bu_{2k} - \bu_{2k+2})
\\
&= \frac{1}{p} \sum_{s \in \Z/(p)} \be_{ \hat{i} + \re_{k\!+\!1} + s( \re_k-\re_{k\!+\!1} ) } \big\{ x_{2k-1}^{-1} x_{2k-3}^{-1} \dots x_{3}^{-1} x_{1}^{-1} \big\} (\bu_{2k} - \bu_{2k+2})
\\
&= \frac{1}{p} \,\sum_{s \in \Z/(p)} \eta^{-(i_1+\dots +i_k)-s} \,\be_{ \hat{i} + \re_{k\!+\!1} + s( \re_k-\re_{k\!+\!1} ) }  (\bu_{2k} - \bu_{2k+2}) .
\end{align*}
\end{proof}
%
\subsection{\bf A concrete description of the $B_{2g}$-action on $L_g^{\eta}$ : semi-simple part }

We will describe the semisimple part of the action of $\hat{\Psi}_i \; (1\leq i \leq 2g)$ on $L_g^{\eta}$. 
\\
Recall that $p$ is an odd prime and that $\eta$ is a primitive $p$-th root of unity. Corollary \ref{psi order p} states that $(\hat{\Psi}_{i,\mathrm{ss}}|_{L_g^{\eta}})^p =1 \;(1\leq i \leq 2g)$. Henceforth, we omit the subscript "$|_{L_g^{\eta}}$" because it is apparent from the context. With these understood, we will define the finite Fourier transformation of $\hat{\Psi}_{i,\mathrm{ss}}$ as follows;
\begin{defi} For any $a \in \Z/(p)$ and for $1\leq i \leq 2g$, set
\[
\bbphi\left( \hat{\Psi}_{i,\mathrm{ss}}; a \right) =\frac{1}{p}\sum_{n \in \Z/(p)} \eta^{-an}\, (\hat{\Psi}_{i,\mathrm{ss}})^n .
\]
Notice that these are idempotents of $\mathrm{End}_{\C}(L_g^{\eta})$.
\end{defi}
%
%
\begin{thm}\label{psi hat odd action} For $1\leq k \leq g,\, 1\leq n \leq 2g$ and for any $a \in \Z/(p),\, \hat{j} \in (\Z/(p))^{g}$, it holds that
\begin{align*}
& \bbphi\left( \hat{\Psi}_{2k-1,\mathrm{ss}}; a \right) (\be_{\hat{j}} \bu_{n}) 
\\
&= \delta_{a, \binom{j_k+1}{2} } \left\{
\be_{\hat{j}} \bu_{n} \,-\,  \delta_{n,2k} \, \delta_{j_k\neq 0} \, (1- \eta^{- j_k} )^{-1} \, \eta^{ \left( j_1 + \dots + j_{k-1} \right) }  \be_{\hat{j} - \re_k} (\bu_{2i-1} - \bu_{2i-3}) \right\}.
\end{align*}
, where the symbol $\delta_{a\neq b}$ is determined by
\[
\delta_{a\neq b} :=
\begin{cases}
1 & \text{ if } a\neq b,
\\
0 & \text{ if otherwise} .
\end{cases}
\]
\end{thm}
\begin{proof} The proof in the case where $n\neq 2k$ is easy. In fact, 
\begin{align*}
\bbphi\left( \hat{\Psi}_{2k-1,\mathrm{ss}}; a \right) (\be_{\hat{j}} \bu_{n})
& = \frac{1}{p} \sum_{m \in \Z/(p)} \eta^{-am} \,\hat{\Psi}_{2k-1}^m (\be_{\hat{j}} \bu_{n})
\\
&= \frac{1}{p} \sum_{m \in \Z/(p)} \eta^{-am} \,\Psi_{2k-1}^m (\be_{\hat{j}}) \bu_{n}
\\
&= \frac{1}{p} \sum_{m \in \Z/(p)} \eta^{-am} \cdot \eta^{\binom{j_k+1}{2} m} \,\be_{\hat{j}} 
\\
&= \delta_{a, \binom{j_k+1}{2} } \be_{\hat{j}} \bu_{n} .
\end{align*}
Next we treat the remaining case where $n=2k$.  Applying the 1st formula in Theorem \ref{funda jordan new} succesively, it follows that, for each $m \in \Z/(p)$,
\footnotesize
\begin{align*}
& (\hat{\Psi}_{2k-1,\mathrm{ss}})^m (\be_{\hat{j}} \bu_{2k}) 
\\ 
& =(\Psi_{2k-1})^m (\be_{\hat{j}}) \left\{ \bu_{2k} 
 - \{ x_1 x_3 \dots x_{2i-3} \left( 1 + x_{2k-1}^{-1} + \dots + x_{2k-1}^{-m+1} -m \bphi ( x_{2k-1} )  \right)  x_{2i} x_{2i\!+\!2}\dots x_{2g} \}  (\bu_{2i\!-\!1} \!-\! \bu_{2i\!-\!3}) \right\} 
\\
& = \eta^{\binom{j_k+1}{2}m}  \left\{ \be_{\hat{j}}\bu_{2k} 
 - \be_{\hat{j}}\{ x_1 x_3 \dots x_{2i-3} \left( 1 + x_{2k-1}^{-1} + \dots + x_{2k-1}^{-m+1} -m \bphi ( x_{2k-1} )  \right)  x_{2i} x_{2i\!+\!2}\dots x_{2g} \}  (\bu_{2i\!-\!1} \!-\! \bu_{2i\!-\!3}) \right\} . 
\\
&= \eta^{\binom{j_k+1}{2}m} \left\{ \be_{\hat{j}} \bu_{2k} - \delta_{j_k\neq 0}\frac{1-\eta^{-m j_k} }{1- \eta^{- j_k} }\eta^{ \left( j_1 + \dots + j_{k-1} \right) } \be_{\hat{j} - \re_k} (\bu_{2i-1} - \bu_{2i-3}) \right\} .
\end{align*}
\normalsize
Multiplying both sides by $\eta^{-am}$ and summing up the results with $m$ running through $\Z/(p)$, we obtain the desired formula.
\end{proof}
%
\begin{defi} Define the map
\[
\tau_{p} : \Z/(p) \to \Z/(p) \; :\; i \mapsto \binom{i+1}{2} .
\]
Further, set $\mathrm{I}_{p} := \mathrm{Image}(\tau_{p}) \subset \Z/(p)$.
\end{defi}
\begin{rem} \textbullet\, $\sharp\,\mathrm{I}_{p} = \frac{p+1}{2}$.
\\
\textbullet\, $0, -\frac{1}{8} \in \mathrm{I}_{p},\; \tau_{p}^{-1}(0) = \{0,-1\},\;  \tau_{p}^{-1}(-\frac{1}{8}) = \{ -\frac{1}{2} \}$.
\\
\textbullet\, $\tau_{p}(a) = \tau_{p}(b) \Longleftrightarrow a+b=1$.
\end{rem}
\begin{defi} Define the map
\[
\hat{\tau_{p}} : (\Z/(p))^{g} \to \mathrm{I}_{p}^{g} \; :\; \hat{i} \mapsto (\tau_{p}(i_{1}),\dots, \tau_{p}(i_{g}) ) .
\]
\end{defi}
\begin{cor} The spectrum of the action $\hat{\Psi}_l \; (1\leq l \leq 2g)$ on $L_g^{\eta}$ coincides with $\{ \eta^a \mid a \in \mathrm{I}_{p} \} \subset \C$. Thus one may safely say that, for each $\hat{\Psi}_l$, a copy of $\mathrm{I}_{p}$ parametrizes its spectrum. 
\end{cor}
\begin{proof} Theorem \ref{psi hat odd action} shows the assertion for $\hat{\Psi}_{2k-1} \; (1\leq k \leq g)$. But the braid relation implies that $\hat{\Psi}_{2k-1}$ and $\hat{\Psi}_{2k}$ are conjugate to each other in $\mathrm{End}_{\C}(L_g^{\eta})$, which completes the proof.
\end{proof}
%
%
%
\begin{defi} For each $a,s \in \Z_p$, define the constant $\hB^a_s$ as follows;
\[
\hB^a_s := 
\begin{cases} 0 & \text{ if } a \not\in \mathrm{I}_{p}, \\
\eta^{ms} + \eta^{-ms} & \text{ if } a \in \mathrm{I}_{p}
\end{cases}
\]
, where $m \in \{0,1,\dots \frac{p-1}{2}\}$ is the element such that $\mathrm{I}_{p}(m-\frac{1}{2})=a$, that is, $m^2 = 2a+\frac{1}{4}$. 
\end{defi}
\begin{rem} We have the symmetry $\hB^a_{-s} = \hB^a_s$.
\end{rem}
%
%
%
\begin{defi} For each $a,s \in \Z_p$, define the constant $\hC^a_s$ as follows;
\[
\hC^a_s := \begin{cases}
0  & \text { if } a \not\in \mathrm{I}_{p},
\\
\{ \eta^{(s-\frac{1}{2})} + \eta^{-(s-\frac{1}{2})} \} ( \eta^{-\frac{1}{2}} - \eta^{\frac{1}{2}} )^{-1} 
& \text { if } a=0 ,
\\
\{ \eta^{\frac{1}{2}(s-\frac{1}{2})} + \eta^{-\frac{1}{2}(s-\frac{1}{2})} \} ( \eta^{-\frac{1}{4}} - \eta^{\frac{1}{4}} )^{-1}
& \text { if } a= -\frac{1}{8} ,
\\
 \{ \eta^{(m-\frac{1}{2})(s-\frac{1}{2})} + \eta^{-(m-\frac{1}{2}) (s-\frac{1}{2})} \} (\eta^{\frac{1}{2} (m-\frac{1}{2})} - \eta^{-\frac{1}{2}(m-\frac{1}{2}) } )^{-1} &
\\
\; -\{ \eta^{(m+\frac{1}{2})(s-\frac{1}{2})} + \eta^{-(m+\frac{1}{2}) (s-\frac{1}{2})} \} (\eta^{\frac{1}{2} (m+\frac{1}{2})} - \eta^{-\frac{1}{2}(m+\frac{1}{2}) } )^{-1} & \text { if otherwise}
\end{cases}
\]
, where $m \in \{0,1,\dots \frac{p-1}{2}\}$ is the one such that $m^2 = 2a+\frac{1}{4}$. 
\end{defi}
\begin{rem} We have the symmetry $\hC^a_{1-s} = \hC^a_s$.
\end{rem}
%
%
%
\begin{prop} For any $a \in \Z/(p)$ and for any $\hat{n} \in (\Z/(p))^{g}$, it holds that
\[ \bbphi\left( \Psi_{2k}; a \right) (\be^*_{\hat{n}}) = \delta_{a,\binom{n_k+1}{2} } \be^*_{\hat{n}}
\]
\end{prop}
\begin{proof} The result follows immediately from Proposition \ref{psi even act star}.
\end{proof}
\begin{prop}\label{phi psi action} For any $a \in \Z/(p)$ and for any $\hat{i} \in (\Z/(p))^{g}$, it holds that
\[
\bbphi\left( \Psi_{2k}; a \right) (\be_{\hat{i}}) = \frac{1}{p} \sum_{ s \in \Z/(p)} \eta^{-\frac{s}{2}}\hB^a_s \,\be_{\hat{i} + s(\re_k-\re_{k+1})} .
\]
\end{prop}
\begin{proof}
%
%
\begin{align*}
& \bbphi\left( \Psi_{2k}; a \right) (\be_{\hat{i}})
\\
& = \alpha_g \sum_{\hat{n} \in (\Z/(p))^g} \eta^{-\hat{n} \cdot \Omega \hat{i} } \,\bbphi\left( \Psi_{2k}; a \right) (\be^*_{\hat{n}} ) 
\\
&  = \alpha_g  \sum_{\hat{n}^\prime : \hat{n}^\prime \perp \re_k} \;\sum_{n: \binom{n+1}{2}=a }  \eta^{-(\hat{n}^\prime + n\re_k) \cdot \Omega \hat{i} } \,\be^*_{\hat{n}} 
\\
&  = \alpha_g \bar{\alpha}_g \sum_{\hat{n}^\prime : \hat{n}^\prime \perp \re_k} \;\sum_{n: \binom{n+1}{2}=a } \;\sum_{\hat{j} \in (\Z_p)^g } \eta^{-(\hat{n}^\prime + n\re_k) \cdot \Omega (\hat{i} - \hat{j}) } \,\be_{\hat{j}} 
\\
& = \frac{1}{p}  \sum_{n: \binom{n+1}{2}=a } \;\sum_{ \hat{j} :\, \Omega (\hat{i} - \hat{j}) \in \langle \re_k \rangle_{\bbf} } \eta^{- n\re_k \cdot \Omega (\hat{i} - \hat{j}) } \,\be_{\hat{j}} 
\\
& = \frac{1}{p}  \sum_{n: \binom{n+1}{2}=a } \;\sum_{ \hat{j} : \, (\hat{i} - \hat{j}) \in \langle (1-T)\re_k \rangle_{\bbf} } \eta^{-n\re_k \cdot \Omega (\hat{i} - \hat{j}) } \,\be_{\hat{j}} 
\\
& = \frac{1}{p}  \sum_{n: \binom{n+1}{2}=a } \;\sum_{ s \in \Z/(p)} \eta^{ns} \,\be_{\hat{i} + s(1-T)\re_k} 
\\
& = \frac{1}{p}  \sum_{ s \in \Z/(p)} \;\sum_{n: \binom{n+1}{2}=a } \eta^{ns} \,\be_{\hat{i} + s(\re_k-\re_{k+1})} 
\\
\end{align*}
, where we have used the result of the previous proposition at the 2nd equality and the fact that $\Omega^{-1} = 1-T$ at the 3rd to the last equality. Thus the result follows from Lemma \ref{tech lem1} below.
\end{proof}
\begin{lem}\label{tech lem1} For any $a \in \Z/(p)$, it holds that
\[
  \sum_{n: \binom{n+1}{2}=a } \eta^{ns} = \eta^{-\frac{s}{2}}\hB^a_s .
\]
\end{lem}
\begin{proof} Put $m =n+\frac{1}{2}$. Then,
\[
L.H.S.= 
 \sum_{m: m^2 = 2a + \frac{1}{4} } \eta^{(m-\frac{1}{2}) s} = \eta^{-\frac{s}{2}} \sum_{m: m^2 = 2a + \frac{1}{4} } \eta^{ms} = R.H.S..
\]
\end{proof}
%
%
%
%
\begin{thm}\label{psi hat even action} For $1\leq k \leq g,\,1\leq m \leq 2g$ and for any $a \in \Z/(p),\,\hat{i} \in (\Z/(p))^{g}$, it holds that
\begin{align*}
 \bbphi\left(\hat{\Psi}_{2k,\mathrm{ss}} ; a \right) (\be_{\hat{i}} \bu_{m}) 
&= \frac{1}{p}\sum_{ s \in \Z/(p)} \eta^{-\frac{s}{2}}\hB^a_s  \be_{\hat{i} + s(\re_k-\re_{k+1})} \bu_{m}   
\\
& + \; \delta_{l,2k-1} \, \frac{1}{p}\sum_{ s \in \Z/(p)} \,\eta^{-(i_1+ \dots + i_k)} \eta^{-s} \hC^a_s \,\be_{\hat{i} + s\re_k+(1-s)\re_{k+1} } (\bu_{2k} - \bu_{2k+2}) .
\end{align*}
\end{thm}
\begin{proof} In the case where $m \neq 2k-1$, we see that
\[
\bbphi\left(\hat{\Psi}_{2k,\mathrm{ss}} ; a \right) (\be_{\hat{i}} \bu_{m}) = \bbphi\left(\hat{\Psi}_{2k,\mathrm{ss}} ; a \right)\! (\be_{\hat{i}}) \,\bu_{m} .
\]
Thus the result follows directly from Proposition \ref{phi psi action}.
\\

Now we are going to handle the case where $m=2k-1$. Applying the 3rd formula in Theorem \ref{funda jordan new} succesively, it follows that, for each $l \in \Z/(p)$,
\begin{align*}
& (\hat{\Psi}_{2k,\mathrm{ss}})^l (\be^*_{\hat{n}} \bu_{2k-1})
\\
&= (\Psi_{2k})^l (\be^*_{\hat{n}}) 
\\
& \cdot\!\left\{ \bu_{2k-1}  +  x_{2g}^{-1}   \cdots x_{2k+2}^{-1} x_{2k}^{-1} \left( 1+ x_{2k}^{-1} +\cdots + x_{2k}^{-(l-1)} - l \bphi (x_{2k} ) \right) x_{2k-1}^{-1} x_{2k-3}^{-1} \cdots  x_{1}^{-1} (\bu_{2k} - \bu_{2k+2}) \right\}  
\\
&= \eta^{\binom{n_k+1}{2} l } \left\{ \,\be^*_{\hat{n}} \bu_{2k-1} + \delta_{n_k \neq 0} \frac{1-\eta^{-l n_k} }{1- \eta^{- n_k} } \eta^{-(n_k + \dots + n_g)} \,\be^*_{\hat{n}} (\bu_{2k} - \bu_{2k+2}) \right\}
\\
&= \eta^{\binom{n_k+1}{2} l } \be^*_{\hat{n}} \bu_{2k-1}
\\
&\quad + \delta_{n_k \neq 0}  \left( \eta^{\binom{n_k+1}{2} l } -  \eta^{\binom{n_k}{2} l }\right) (\eta^{n_k}-1 )^{-\!1} \,\eta^{-(n_{k+1} + \dots + n_g)} \,\be^*_{\hat{n} -\re_k} (\bu_{2k} - \bu_{2k+2}) .
\end{align*}
%
Multiplying both sides by $p^{-1} \eta^{-al}$ and summing up the results with $l$ running through $\Z/(p)$, we obtain 
\begin{align*}
& \bbphi\left(\hat{\Psi}_{2k,\mathrm{ss}} ; a \right) (\be^*_{\hat{n}} \bu_{2k-1})
\\
&= \bbphi\left(\Psi_{2k} ; a \right) (\be^*_{\hat{n}}) \bu_{2k-1}
\\
&\quad + \delta_{n_k \neq 0} \left( \delta_{a, \binom{n_k+1}{2}} -  \delta_{a, \binom{n_k}{2}} \right) (\eta^{n_k} -1)^{-\!1} \,\eta^{-(n_{k+1} + \dots + n_g)} \,\be^*_{\hat{n} -\re_k} (\bu_{2k} - \bu_{2k+2}) .
\end{align*}
%
Using the Fourier inversion formula and linearity, we see that
\begin{align*}
& \bbphi\left(\hat{\Psi}_{2k,\mathrm{ss}} ; a \right) (\be_{\hat{i}} \bu_{2k-1})
\\
&= \bbphi\left(\Psi_{2k} ; a \right) (\be_{\hat{i}}) \bu_{2k-1}
\\
&\quad + \bar{\alpha}_g  \sum_{\hat{n}: n_k\neq 0, \binom{n_k+1}{2}=a} \eta^{-\hat{n} \cdot \Omega \hat{i} } (\eta^{n_k} -1)^{-\!1} \,\eta^{-(n_{k+1} + \dots + n_g)} \,\be^*_{\hat{n} -\re_k} (\bu_{2k} - \bu_{2k+2})
\\
&\quad - \bar{\alpha}_g  \sum_{\hat{n}: n_k\neq 0, \binom{n_k}{2}=a} \eta^{-\hat{n} \cdot \Omega \hat{i} } (\eta^{n_k} -1)^{-\!1} \,\eta^{-(n_{k+1} + \dots + n_g)} \,\be^*_{\hat{n} -\re_k} (\bu_{2k} - \bu_{2k+2}) .
\end{align*}
%
\t
The 1st term has been calculated just before. We will manage the 2nd and the 3rd terms.
\begin{align*}
&\text{2nd term}
\\
&=\bar{\alpha}_g  \sum_{\hat{n}: n_k\neq 0, \binom{n_k+1}{2}=a} \eta^{-\hat{n} \cdot \Omega (\hat{i} + \re_{k+1}) } (\eta^{n_k} -1)^{-\!1} \,\be^*_{\hat{n} -\re_k} (\bu_{2k} - \bu_{2k+2})
\\
& =\alpha_g \bar{\alpha}_g\!  \sum_{\hat{n}: n_k\neq 0, \binom{n_k+1}{2}=a} \; \sum_{\hat{j} \in (\Z/(p))^g} \eta^{ \hat{n} \cdot \Omega (\hat{i} + \re_{k+1}) \,-\, (\hat{n} -\re_k )\cdot \Omega \hat{j}  } \,(\eta^{n_k} -1)^{-\!1} \,\be_{\hat{j}} (\bu_{2k} - \bu_{2k+2})
\\
& =\alpha_g \bar{\alpha}_g\!  \sum_{\hat{n}: n_k\neq 0, \binom{n_k+1}{2}=a} \; \sum_{\hat{j} \in (\Z/(p))^g} \eta^{ -\hat{n} \cdot \Omega (\hat{i} - \hat{j} + \re_{k+1}) \,-\, \re_k \cdot \Omega \hat{j} } \,(\eta^{n_k} -1)^{-\!1} \,\be_{\hat{j}} (\bu_{2k} - \bu_{2k+2})
\\
&=p^{-g} \sum_{\hat{n}^\prime: \hat{n}^\prime \perp \re_k} \;\sum_{n: n\neq 0, \binom{n+1}{2}=a} \;\sum_{\hat{j} \in (\Z/(p))^g} \eta^{ -(\hat{n}^\prime + n\re_k) \cdot \Omega (\hat{i} - \hat{j} + \re_{k+1}) \,-\, \re_k \cdot \Omega \hat{j} } \,(\eta^{n} -1)^{-\!1} \,\be_{\hat{j}} (\bu_{2k} \!-\! \bu_{2k+2}) .
\end{align*}
\t
When one executes the summation above with respect to the variable $\hat{n}^\prime$ first, each term could give non trivial contribution only if  $\Omega (\hat{i} - \hat{j} + \re_{k+1}) \in \bbf \re_k$, which is equivalent to the condition that $\hat{j} =  \hat{i} + \re_{k+1} + s(1-T)\re_k \; (= \hat{i} + s\re_k + (1-s)\re_{k+1} \,)$ for some $s \in \Z/(p)$. Thus, 
\begin{align*}
&\text{R.H.S.}
\\ 
&= p^{-1} \sum_{s \in \Z/(p)} \;\sum_{n: n\neq 0, \binom{n+1}{2}=a}  \,\eta^{ns-\re_k \cdot \Omega \{ \hat{i} + \re_{k+1} + s(\re_k -\re_{k+1})\} } \,(\eta^{n} -1)^{-\!1} \,\be_{\hat{i} +  s\re_k + (1-s)\re_{k+1} } (\bu_{2k} - \bu_{2k+2}) 
\\
&= p^{-1} \sum_{s \in \Z/(p)} \;\sum_{n: n\neq 0, \binom{n+1}{2}=a}  \,\eta^{(n-1)s-(i_1+ \dots + i_k) } \,(\eta^{n} -1)^{-\!1} \,\be_{\hat{i} + s\re_k+(1-s)\re_{k+1} } (\bu_{2k} - \bu_{2k+2}) .
\end{align*}
%
Almost the same calculation shows that
\begin{align*}
&\text{3rd term}
\\
&=-p^{-1} \sum_{s \in \Z/(p)} \;\sum_{n: n\neq 0, \binom{n}{2}=a}  \,\eta^{(n-1)s-(i_1+ \dots + i_k) } \,(\eta^{n} -1)^{-\!1} \,\be_{\hat{i} + s\re_k+(1-s)\re_{k+1} } (\bu_{2k} - \bu_{2k+2})
\\
&=-p^{-1} \sum_{s \in \Z/(p)} \;\sum_{n: n\neq -1, \binom{n+1}{2}=a}  \,\eta^{ns-(i_1+ \dots + i_k) } \,(\eta^{n+1} -1)^{-\!1} \,\be_{\hat{i} + s\re_k+(1-s)\re_{k+1} } (\bu_{2k} - \bu_{2k+2}) .
\end{align*}
\normalsize
\t
Thus the result follows from Lemma \ref{tech lem2} below, whose proof is a tedious calculation.
\end{proof}
%
\begin{lem}\label{tech lem2} For any $a, s \in \Z/(p)$, it holds that
\[
\sum_{n: \binom{n+1}{2}=a} \left\{ \delta_{n\neq 0} \,\eta^{(n-1)s} (\eta^{n} -1)^{-\!1} - \delta_{n\neq -1} \,\eta^{ns} (\eta^{n+1} -1)^{-\!1} \right\}
= \eta^{-s} \hC^a_s .
\]
\end{lem}

\section{ The subalgebra of $\mathrm{End}_{\C}(L^g_\eta)$ generated by the $B_{2g+1}$ action }\label{section 5}

\subsection{Setting} 

The $B_{2g+1}$-action on $L_g^\eta$ induces an algebra morphism $\hat{\Psi}_\C : \C[B_{2g+1}] \to \mathrm{End}_\C (L_g^\eta)$. The important problem is to describe the image of this morphism as explicitely as possible. For example, the Burnside theorem states that $L_g^\eta$ is an irreducible $B_{2g+1}$-module if and only if $\hat{\Psi}_\C$ is surjective. 

Since the endomorphisms $ \hat{\Psi}_{2k-1} \!=\! \hat{\Psi}_\C(\sigma_{2k-1})  \in \mathrm{End}_\C (L_g^\eta)$ for $1\leq k \leq g$ commute to each other, we can consider their simultaneous generalized eigenspaces. Set the projections to these generalized eigenspaces as 
\[
\mathrm{Proj}( \hat{\Psi}_{\mathit{odd},\mathrm{ss}}; \hat{a} )
:=\prod_{k=1}^g \bbphi\left(\hat{\Psi}_{2k-1,\mathrm{ss}}; a_k \right) \quad (\;\hat{a} \in \mathrm{I}_{p}^{g} \;)
\]
, which are idempotents of the algebra $\mathrm{End}_\C (L^g_\eta)$.
\begin{defi} For any $\hat{a},\; \hat{b} \in \mathrm{I}_{p}^{g}$, set
\begin{align*}
\mathrm{Hom}(\eta; \hat{a},\hat{b}) &:= \mathrm{Proj}( \hat{\Psi}_{\mathit{odd},\mathrm{ss}}; \hat{a} ) \circ \mathrm{End}_\C (L_g^\eta) \circ \mathrm{Proj}( \hat{\Psi}_{\mathit{odd},\mathrm{ss}}; \hat{b} ) , 
\\
\mathrm{End}(\eta; \hat{a}) &:= \mathrm{Hom}(\eta; \hat{a},\hat{a}) ,
\\
L_g^\eta(\hat{a}) &:= \mathrm{Image}(\mathrm{Proj}( \hat{\Psi}_{\mathit{odd},\mathrm{ss}}; \hat{a} )).
\end{align*}
\end{defi}
Thus by definition, we see that 
\[
v \in L_g^\eta(\hat{a}) \;\Longleftrightarrow\; \hat{\Psi}_{2k-1}(v) = \eta^{a_{k}}v \;\; (1\leq k \leq g) .
\]
Obviously we have the direct sum decompositions
\begin{align*}
L_g^\eta &= \bigoplus_{\hat{a} \in \mathrm{I}_{p}^{g}} L_g^\eta(\hat{a}), \quad \mathrm{End}_\C (L^g_\eta) = \bigoplus_{\hat{a},\hat{b} \in \mathrm{I}_{p}^{g} } \mathrm{Hom}(\eta; \hat{a},\hat{b}) \equiv \bigoplus_{\hat{a},\hat{b} \in \mathrm{I}_{p}^{g} } \mathrm{Hom}_{\C}(L_g^\eta(\hat{a}), L_g^\eta(\hat{b})).
\end{align*}
Thus it holds that
\[
\mathrm{Hom}(\eta; \hat{a},\hat{b}) \mathrm{Hom}(\eta; \hat{c},\hat{d}) =
\begin{cases} 
\mathrm{Hom}(\eta; \hat{a},\hat{d}) & \text{ if } \hat{b}=\hat{c},
\\
\{ 0 \} & \text{ if otherwise}.
\end{cases}
\]
In particular, $\mathrm{End}(\eta; \hat{a}) \subset \mathrm{End}_\C (L^g_\eta)$ is a subalgebra and $L_g^\eta(\hat{a})$ a faithful left $\mathrm{End}(\eta; \hat{a})$-module.\\
\; Denote by $\mathcal{B}^\eta$ the subalgebra $\hat{\Psi}_\C (\C[B_{2g+1}]) \subset \mathrm{End}_\C (L_g^\eta)$. Since $\mathcal{B}^\eta$ contains the set of the projections $\{ \mathrm{Proj}( \hat{\Psi}_{\mathit{odd},\mathrm{ss}}; \hat{a} ) \mid \hat{a} \in \hat{I}_p \}$, $\mathcal{B}^\eta$ decomposes into the direct sum of the homogeneous components as follows; 
\[ \mathcal{B}^\eta = \bigoplus_{\hat{a},\hat{b} \in \mathrm{I}_{p}^{g} } \mathcal{B}^\eta_{\hat{a},\hat{b}}
\]
, where $\mathcal{B}^\eta_{\hat{a},\hat{b}} = \mathcal{B}^\eta \cap \mathrm{Hom}(\eta; \hat{a},\hat{b})$. In particular, $\mathcal{B}^\eta_{\hat{a},\hat{a}}$ is a subalgebra of both $\mathcal{B}^\eta$ and $\mathrm{End}(\eta; \hat{a})$ for any $\hat{a}$.

Now we will introduce a useful new basis of $L_g^\eta(\hat{a})$.
\begin{defi}\label{new basis} For $1\leq k\leq g$ and for any $\hat{j} \in (\Z_p)^g$, set 
\begin{align*}
\btw_{2k-1}^{\hat{j}} &:= \be_{\hat{j}} (\bu_{2k-1} - \bu_{2k-3}),
\\
\btv_{2k}^{\hat{j}} &:= \eta^{ -( j_1 + \dots + j_{k-1} ) }\bbphi\left(\hat{\Psi}_{2k-1,\mathrm{ss}}; \binom{j_k+1}{2} \right) (\be_{\hat{j}} \bu_{2k} )
\\
& = \eta^{ -( j_1 + \dots + j_{k-1} ) } \be_{\hat{j}} \bu_{2k} -  \delta_{j_k\neq 0} (1- \eta^{- j_k} )^{-1}  \be_{\hat{j} - \re_k} (\bu_{2k-1} - \bu_{2k-3}).
\end{align*}
Note that in the case where $k=1$ we presume $\bu_{-1}\equiv 0$. Further, notice that we have used the result of Theorem \ref{psi hat odd action} at the 2nd equality in the 2nd equation. 
\end{defi}
By the very definition, it holds that for $1\leq k\leq g$ and for any $\hat{j} \in \hat{\tau}_{p}^{-1}(\hat{a})$,
\[
\btw_{2k-1}^{\hat{j}},\; \btv_{2k}^{\hat{j}} \in L_g^\eta(\hat{a}) .
\]
In fact, these vectors compose a basis of $L_g^\eta(\hat{a})$.
\begin{defi}[Distinguished Basis] The {\it distinguished basis} of $L_g^\eta(\hat{a}) \; (\hat{a} \in \mathrm{I}_{p}^{g})$ is
\[
 \left\{ \btw_{2k-1}^{\hat{j}} ,\; \btv_{2k}^{\hat{j}}\; \mid \; 1\leq k \leq g,\;  \hat{j} \in \hat{\tau}^{-1}(\hat{a}) \right\}.
\]  
\end{defi}
\begin{defi}[Distinguished Subspace] A subspace $V \subset L_{g}^\eta(\hat{a})$ is {\it distinguished} if $V$ is spanned by some subset $F$ of the distinguished basis of $L_{g}^\eta(\hat{a})$. If it is the case, we say that $V$ has the {\it distinguished basis} $F$.
\end{defi}
\begin{defi}[Odd and Even Subspaces] For any $\hat{a} \in \mathrm{I}_{p}^{g}$, define the {\it odd} and the {\it even subspaces} of $L_{g}^\eta(\hat{a})$ by
\begin{align*}
L_{g\,\mathrm{odd}}^\eta(\hat{a}) &:= \left\langle  \btw_{2k-1}^{\hat{j}} \; \middle| \; 1\leq k \leq g,\;  \hat{j} \in \hat{\tau}^{-1}(\hat{a}) \right\rangle_{\C} ,
\\
L_{g\,\mathrm{even}}^\eta(\hat{a}) &:= \left\langle  \btv_{2k}^{\hat{j}} \; \middle| \; 1\leq k \leq g,\;  \hat{j} \in \hat{\tau}^{-1}(\hat{a}) \right\rangle_{\C}
\end{align*}
, respectively. Thus we have
\[
L_{g}^\eta(\hat{a}) = L_{g\,\mathrm{odd}}^\eta(\hat{a}) \oplus L_{g\,\mathrm{even}}^\eta(\hat{a}) .
\]
\end{defi}
\begin{defi}[Distinguished Morphism]\label{dis morphism} A $\C$-linear endomorphism $\phi$ of $L_{g}^\eta(\hat{a})$ is a {\it distinguished morphism} if $\mathrm{Ker}(\phi)$ is distinguished and if, for any distinguished subspace $V \subset L_{g}^\eta(\hat{a})$, $\phi(V)$ is distinguished.  
\end{defi}


Recall that we are subject to Convention \ref{conv ignore} so that we presume the symbol $\re_{g+1}$ to be identically zero. Now we add the following convention.
\begin{conv}\label{conv ignore2} Whenever we meet the symbol $\btw_{2g+1}$ or $\btv_{2g+2}$, we ignore it, that is, we presume that $\btw_{2g+1}=0=\btv_{2g+2}$.
\end{conv}
The following formulae are essential ingredients in the subsequent part of the present paper. 
\begin{prop}\label{important odd} For any $k,l \in \{1,2,\dots, g\},\; a \in \Z/(p) ,\; \hat{j} \in (\Z/(p))^g$, it holds that
\begin{align*}
& \bbphi\left(\hat{\Psi}_{2k,\mathrm{ss}} ; a \right) (\be_{\hat{j}} \bu_{2l-1} )
\\
&= \frac{1}{p}\sum_{ s \in \Z/(p)} \eta^{-\frac{s}{2}}\hB^a_s \,\be_{\hat{j} + s(\re_k - \re_{k+1})} \bu_{2l-1}
\\
& + \delta_{l,k}\, \frac{1}{p}\sum_{ s \in \Z/(p)} \hC^a_{s+1} \,\delta_{s+j_k+1\neq 0} \left(\eta^{s + j_k + 1} - 1 \right)^{-1} \btw_{2k-1}^{ \hat{j} + s(\re_k - \re_{k+1}) }
\\
& + \delta_{l,k}\, \frac{1}{p}\sum_{ s \in \Z/(p)}  \hC^a_s \,\delta_{s - j_{k+1} -1\neq 0} \left(\eta^{s - j_{k+1} -1} - 1 \right)^{-1} \btw_{2k+1}^{\hat{j} + s(\re_k - \re_{k+1}) }
\\
& + \delta_{l,k}\, \frac{1}{p}\sum_{ s \in \Z/(p)} \hC^a_s \left\{ \eta^{-(j_k + s)}\, \btv_{2k}^{ \hat{j} + s\re_k + (1-s)\re_{k+1} } \,-\,  \btv_{2k+2}^{\hat{j} + s\re_k + (1-s)\re_{k+1} } \right\}
\end{align*}
\end{prop}
\begin{proof} The result follows plugging in the defining equations of $\btv$s for the formula in Theorem \ref{psi hat even action}. 
\end{proof}
\begin{defi} For $1\leq k \leq g$ and for any $s \in \Z/(p),\; \hat{j} \in (\Z/(p))^g$, set 
\[
\bv(\hat{j},\,k,\,s) :=  \eta^{-(j_k +s)} \btv_{2k}^{\hat{j} + s(\re_k-\re_{k+1})} \,-\, \btv_{2k+2}^{\hat{j} + s(\re_k-\re_{k+1})} . 
\]
\end{defi}
\begin{cor}\label{important odd cor} As a direct corollary of the previous proposition, we see that
\[
\bbphi\left(\hat{\Psi}_{2k,\mathrm{ss}} ; a \right) (\be_{\hat{j}} \bu_{2l-1} )
\underset{ \mathrm{mod}\; \lgo^{\eta} }{\equiv}  \delta_{l,k}\, \frac{1}{p}\sum_{ s \in \Z/(p)} \hC^a_s \,\bv(\hat{j} + \re_{k+1}, \,k,\,s) . 
\]
\end{cor}
\begin{prop}\label{important even} For $1\leq k,l \leq g$ and for any $a \in \Z/(p),\; \hat{j} \in (\Z/(p))^g$, the following equation modulo $\lgo^{\eta}$ holds;
\begin{align*}
\bbphi\left(\hat{\Psi}_{2k,\mathrm{ss}} ; a \right) ( \btv_{2l}^{\hat{j}} )
 \;\underset{ \mathrm{mod}\; \lgo^{\eta} }{\equiv} \;
& \frac{1}{p}\sum_{ s \in \Z/(p)} \eta^{-\varepsilon(l,k) \frac{s}{2}}\hB^a_s \, \btv^{\hat{j} + s(\re_k-\re_{k+1})}_{2l}
\\
&  -  \delta_{l,k}\,\delta_{j_k\neq 0} \,(1-\eta^{-j_k})^{-1} \, \frac{1}{p}\sum_{ s \in \Z/(p)} \hC^a_{s+1} \, \bv(\hat{j},\,k,\,s) 
\\
&  +  \delta_{l,k+1}\,\delta_{j_{k+1}\neq 0} \,(1-\eta^{-j_{k+1} } )^{-1} \, \frac{1}{p}\sum_{ s \in \Z/(p)} \hC^a_{s} \, \bv(\hat{j},\,k,\,s) 
\end{align*}
, where 
\[
\varepsilon(l,k)= 
\begin{cases}
-1 & \text{ if } l=k+1,
\\
1 & \text{ if otherwise}.
\end{cases}
\]
\end{prop}
\begin{proof} Suppose that $l \neq k+1$. Then we see that 
\begin{align*} 
& \bbphi\left(\hat{\Psi}_{2k,\mathrm{ss}} ; a \right) ( \btv_{2l}^{\hat{j}} ) 
\\
&= \eta^{ -( j_1 + \dots + j_{l-1} ) } \bbphi \left(\hat{\Psi}_{2k,\mathrm{ss}} ; a \right) (\be_{\hat{j}} \bu_{2l}) 
\\
& \quad -  \delta_{j_{l}\neq 0} (1- \eta^{- j_{l}} )^{-1} \bbphi\left(\hat{\Psi}_{2k,\mathrm{ss}} ; a \right) \left( \be_{\hat{j} - \re_{l}} (\bu_{2l-1} - \bu_{2l-3}) \right)
\\
& \underset{ \mathrm{mod}\; \lgo^{\eta} }{\equiv} 
 \eta^{ -( j_1 + \dots + j_{l-1}) } \; \frac{1}{p}\sum_{ s \in \Z/(p)} \eta^{- \frac{s}{2}}\hB^a_s \, \be_{\hat{j} + s(\re_k-\re_{k+1})} \bu_{2l}
\\
& \qquad\quad - \delta_{l,k} \delta_{j_{k}\neq 0} (1- \eta^{- j_{k}} )^{-1} \bbphi\left(\hat{\Psi}_{2k,\mathrm{ss}} ; a \right) ( \be_{\hat{j} - \re_{k}} \bu_{2k-1} ) 
\\
& \underset{ \mathrm{mod}\; \lgo^{\eta} }{\equiv} 
 \frac{1}{p}\sum_{ s \in \Z/(p)} \eta^{\frac{s}{2}}\hB^a_s \, \btv^{\hat{j} + s(\re_k-\re_{k+1})}_{2l}
\\
& \qquad\quad -  \delta_{l,k} \delta_{j_{k}\neq 0} (1- \eta^{-j_{k}} )^{-1} \, \frac{1}{p}\sum_{ s \in \Z/(p)} \hC^a_{s} \, \bv(\hat{j},\,k,\,s-1) .
\end{align*}
Next, suppose that $l=k+1$. Then we see that
\begin{align*}
& \bbphi\left(\hat{\Psi}_{2k,\mathrm{ss}} ; a \right) ( \btv_{2k+2}^{\hat{j}} ) 
\\
&= \eta^{ -( j_1 + \dots + j_{k} ) } \bbphi \left(\hat{\Psi}_{2k,\mathrm{ss}} ; a \right) (\be_{\hat{j}} \bu_{2k+2}) 
\\
& \quad -  \delta_{j_{k+1}\neq 0} (1- \eta^{- j_{k+1}} )^{-1} \bbphi\left(\hat{\Psi}_{2k,\mathrm{ss}} ; a \right) \left( \be_{\hat{j} - \re_{k+1}} (\bu_{2k+1} - \bu_{2k-1}) \right)
\\
& \underset{ \mathrm{mod}\; \lgo^{\eta} }{\equiv} 
  \; \frac{1}{p}\sum_{ s \in \Z/(p)}  \eta^{ -\{ j_1 + \dots + j_{k-1}  + (j_{k} +s) \} +s } \cdot \eta^{- \frac{s}{2}}\hB^a_s \, \be_{\hat{j} + s(\re_k-\re_{k+1})} \bu_{2k+2}
\\
& \qquad\qquad -  \delta_{j_{k+1}\neq 0} (1- \eta^{- j_{k+1}} )^{-1} \bbphi\left(\hat{\Psi}_{2k,\mathrm{ss}} ; a \right) \left( -\be_{\hat{j} - \re_{k+1}} \bu_{2k-1} \right)
\\
& \underset{ \mathrm{mod}\; \lgo^{\eta} }{\equiv} 
 \frac{1}{p}\sum_{ s \in \Z/(p)} \eta^{\frac{s}{2}}\hB^a_s \, \btv^{\hat{j} + s(\re_k-\re_{k+1})}_{2k+2}
\\
& \qquad\qquad +  \delta_{j_{k+1}\neq 0} (1- \eta^{-j_{k+1}} )^{-1} \, \frac{1}{p}\sum_{ s \in \Z/(p)} \hC^a_{s} \, \bv(\hat{j},\,k,\,s) .
\end{align*}
\end{proof}

\begin{defi} Set $\hat{\Psi}_{l,\mathrm{unil}} := \hat{\Psi}_{l,\mathrm{uni}} -1 \; (1\leq l \leq 2g)$, which is the nilpotent part of the unipotent part $\hat{\Psi}_{l,\mathrm{uni}}$ of $\hat{\Psi}_{l}$.
\end{defi}
\begin{prop}\label{important unil odd} For $1\leq k,l \leq g$ and for any $\hat{j} \in (\Z/(p))^g$, we have
\begin{align*}
& \hat{\Psi}_{2k-1,\mathrm{unil}} (\btw_{2l-1}^{\hat{j}} ) = 0, 
\quad
\hat{\Psi}_{2k-1,\mathrm{unil}}(\btv_{2l}^{\hat{j}}) = -\delta_{k,l}\, \delta_{j_k,0} \,\btw_{2k-1}^{\hat{j} -\re_k} . 
\end{align*}
\end{prop}
\begin{prop}\label{important unil even} For $1\leq k, l \leq g$ and for any $\hat{j} \in (\Z/(p))^g$, the following equation modulo $\lgo^{\eta}$ holds;
\[
\hat{\Psi}_{2k,\mathrm{unil}} (\be_{\hat{j}} \bu_{2l-1}) 
 \underset{ \mathrm{mod}\; \lgo^{\eta} }{\equiv} 
\delta_{k,l}\, \frac{1}{p} \, \sum_{s \in \Z/(p) } 
\left\{ \eta^{-(j_k+s)} \btv_{2k}^{ \hat{j} +s\re_{k}+ (1-s)\re_{k+1} } - \btv_{2k+2}^{ \hat{j} +s\re_{k}+ (1-s)\re_{k+1} } \right\} 
\]
\end{prop}
\begin{proof} Proposition \ref{important unil odd} and Proposition \ref{important unil even} are both direct consequences of Theorem \ref{psi hat unipotent action}.
\end{proof}


\subsection{Useful operators of $\mathcal{B}^\eta_{\hat{0},\hat{0}}$}
%
{\bf Caution.} In the present subsection, we halt both Convention \ref{conv ignore} and Convention \ref{conv ignore2} and will take care of the difference between the "$k<g$" case and the "$k=g$" case.
\begin{nota*} Denote $\mathrm{Proj}( \hat{\Psi}_{\mathit{odd},\mathrm{ss}}; \hat{0} )$ by the symbol $\mathrm{Proj}(\mathit{odd}; \hat{0})$ for the simplicity of notation.
\end{nota*}
We introduce several useful operators in $\mathcal{B}^\eta_{\hat{0},\hat{0}}$. Recall that $\mathcal{B}^\eta_{\hat{0},\hat{0}}$ was defined to be $\mathcal{B}^\eta \cap \mathrm{Hom}(\eta; \hat{0},\hat{0})$ in the previous subsection.
\\
First of all, note that
$\hat{j} \in \hat{\tau}_p^{-1}(\hat{0})$ if and only if $j_k \in \{0,-1\} \; (1\leq k\leq g)$. 
%
\begin{defi} Define $B_k,\, D_{l} \in \mathcal{B}^\eta_{\hat{0},\hat{0}} \;\; (1\leq k \leq g,\; 1\leq l \leq g-1)$ as follows;
\begin{align*}
B_k &:= -p(\hC^0_0 )^{-1} \, \mathrm{Proj}(\mathit{odd}; \hat{0}) \circ \hat{\Psi}_{2k-1,\mathrm{unil}} \circ \mathrm{Proj}(\mathit{odd}; \hat{0}) \circ \hat{\Psi}_{2k,\mathrm{unil}} \circ \mathrm{Proj}(\mathit{odd}; \hat{0}) ,
\\
D_{l} &:= p(\hC^0_0 )^{-1}\, \mathrm{Proj}(\mathit{odd}; \hat{0}) \circ \hat{\Psi}_{2l+1,\mathrm{unil}} \circ \mathrm{Proj}(\mathit{odd}; \hat{0}) \circ \hat{\Psi}_{2l,\mathrm{unil}}\circ \mathrm{Proj}(\mathit{odd}; \hat{0}) .
\end{align*}
\end{defi}

\begin{prop}\label{operator B} Suppose that $\hat{j} \in \hat{\tau}_{p}^{-1}(\hat{0})$. We have the following formulae;
\begin{align*}
\text{When } k< g, \; B_k(\be_{\hat{j}} \bu_{2l-1} ) &= \delta_{k,l}
\begin{cases} 
\btw_{2k-1}^{\hat{j}} & \text{ if } j_k=-1, \\
\btw_{2k-1}^{\hat{j}-\re_k+\re_{k+1} }  & \text{ if } (j_k, j_{k+1}) = (0,-1),\\
0 & \text{ if } (j_k, j_{k+1}) = (0,0),
\end{cases}
\\
\text{When } k= g, \; B_g( \be_{\hat{j}} \bu_{2l-1} ) &= \delta_{g,l}
\begin{cases} 
\btw_{2g-1}^{\hat{j}} & \text{ if } j_g =-1, \\
\btw_{2g-1}^{\hat{j}-\re_k } & \text{ if } j_g= 0,
\end{cases}
\\
\text{For } k < g, \; D_{k}(\be_{\hat{j}} \bu_{2l-1} ) &= \delta_{k,l}
\begin{cases} 
\btw_{2k+1}^{\hat{j}} & \text{ if } j_{k+1}=-1, 
\\
\btw_{2k+1}^{\hat{j}+\re_k-\re_{k+1} } & \text{ if } (j_k, j_{k+1}) = (-1,0),\\
0 & \text{ if } \text{ if } (j_k, j_{k+1}) = (0,0).
\end{cases}
\end{align*}
In particular, both $B_k$ and $D_k$ preserve $L^\eta_{g\, \mathrm{odd}}(\hat{0})$.
\end{prop}

\begin{defi} Define $B^{\dagger}_k,\, D^{\dagger}_{l} \in \mathcal{B}^\eta_{\hat{0},\hat{0}} \;\; (1\leq k \leq g,\; 2\leq l \leq g)$ as follows;
\begin{align*}
B^{\dagger}_k &:=
-p \mathrm{Proj}(\mathit{odd}; \hat{0}) \circ \hat{\Psi}_{2k,\mathrm{unil}} \circ \mathrm{Proj}(\mathit{odd}; \hat{0}) \circ \hat{\Psi}_{2k-1,\mathrm{unil}}\circ \mathrm{Proj}(\mathit{odd}; \hat{0})
\\
D^{\dagger}_{l} &:=
-p \mathrm{Proj}(\mathit{odd}; \hat{0}) \circ \hat{\Psi}_{{2l-2},\mathrm{unil}} \circ \mathrm{Proj}(\mathit{odd}; \hat{0}) \circ \hat{\Psi}_{2l-1,\mathrm{unil}} \circ \mathrm{Proj}(\mathit{odd}; \hat{0})
\end{align*}
\end{defi}
%
\begin{prop}\label{operator B dag} For $\hat{j} \in \hat{\tau}^{-1}(\hat{0})$, we have the following formulae modulo $\lgo^{\eta}(\hat{0})$;
\footnotesize
\begin{align*}
\text{For } 1\leq k<g, \; B^{\dagger}_k (\btv_{2l}^{\hat{j}}) & \underset{\lgo^{\eta}(\hat{0})}{\equiv} \delta_{k,l} 
\begin{cases}  
\btv_{2k}^{\hat{j}} - \btv_{2k+2}^{\hat{j}} 
& \text{ if } (j_k,j_{k+1}) = (0,0),
\\
\btv_{2k}^{\hat{j} } - \btv_{2k+2}^{\hat{j} } 
+ \eta \btv_{2k}^{\hat{j} -\re_k + \re_{k+1} } - \btv_{2k+2}^{\hat{j} -\re_k + \re_{k+1} } 
& \text{ if } (j_k,j_{k+1}) = (0,-1) ,
\\ 0 & \text{ if } j_k=-1.
\end{cases}
\\
 B^{\dagger}_g (\btv_{2l}^{\hat{j}}) & \underset{\lgo^{\eta}(\hat{0})}{\equiv} \delta_{g,l} 
\begin{cases}  
\btv_{2g}^{\hat{j}} + \eta \btv_{2g}^{\hat{j} -\re_k } & \text{ if }  j_g =0,
\\ 
0 & \text{ if } j_g=-1.
\end{cases}
\\
\text{For } 2\leq k \leq g, \; D^{\dagger}_{k}(\btv_{2l}^{\hat{j}}) & \underset{\lgo^{\eta}(\hat{0})}{\equiv} \delta_{k,l}
\begin{cases}  
\btv_{2k}^{\hat{j}} - \btv_{2k-2}^{\hat{j}}   
& \text{ if } (j_{k-1}, j_{k}) = (0,0),
\\
\btv_{2k}^{\hat{j} }     
 - \eta \btv_{2k-2}^{\hat{j} }  + \btv_{2k}^{\hat{j} + \re_{k-1} - \re_{k} }  - \btv_{2k-2}^{\hat{j} + \re_{k-1} - \re_{k} }
& \text{ if } (j_{k-1}, j_{k}) = (-1,0) ,
\\ 0 & \text{ if } j_{k} =-1 .
\end{cases}
\end{align*}
\normalsize
Further, both $B^{\dagger}_k$ and $D^{\dagger}_{k}$ vanish on $\lgo^\eta(\hat{0})$.
\end{prop}
\begin{proof} Proposition \ref{operator B} and \ref{operator B dag} above follow immediately from Proposition \ref{important unil odd} and \ref{important unil even} respectively.
\end{proof}
\begin{cor} It holds that
\begin{align*}
& B_{*}(L^\eta_{g \,\mathrm{odd}}(\hat{0})),\; D_{*}(L^\eta_{g \,\mathrm{odd}}(\hat{0})) \subset L^\eta_{g \,\mathrm{odd}}(\hat{0}) ,
\\
& B_{*}^{\dagger}(L^\eta_{g \,\mathrm{odd}}(\hat{0})) = 0 = D_{*}^{\dagger}(L^\eta_{g \,\mathrm{odd}}(\hat{0})) .
\end{align*}
Further, $B_{*}$ and $D_{*}$ are all idempotents. The endomorphisms induced by $B_{*}^{\dagger}$ and $D_{*}^{\dagger}$ on $L^\eta_{g}(\hat{0}) / L^\eta_{g \,\mathrm{odd}}(\hat{0})$ are idempotents
\end{cor}
\begin{rem} In the case where $g=1$, we have neither $D_*$ nor $D_{*}^{\dagger}$.
\end{rem}
%


\begin{defi} We introduce some important constants for $s \in \Z / (p)$;
\begin{align*} 
\hlambda_s &:= \hC^0_s - \frac{ 1+ \eta  }{ (1+\eta^{\frac{1}{2}})^2 } \hC^{-\frac{1}{8}}_s, \quad 
\hgamma_s := \hB^0_s - \frac{ 1+ \eta  }{ (1+\eta^{\frac{1}{2}})^2 }\hB^{-\frac{1}{8}}_s.
\end{align*}
Notice that $\hlambda_s = \hlambda_{1-s},\; \hgamma_s = \hgamma_{-s}$ by the symmetry of $\hC^0_s$ and $\hB^0_s$.
\end{defi}
\begin{lem} It holds that
\begin{align*}
\hlambda_0 &= \hlambda_1 = 0,
\quad 
\hlambda_{-1} = \hlambda_{2}= \frac{(1+\eta)}{(1-\eta)} \left\{ (\eta^{-\frac{1}{2}} -\eta^{\frac{1}{2}} )^2 - (\eta^{-\frac{1}{4}} -\eta^{\frac{1}{4}} )^2 \right\} ,
\\
\hgamma_{\!\pm 1} &= (\eta^{-\frac{1}{2}} + \eta^{\frac{1}{2}} )^2 / (\eta^{-\frac{1}{4}} + \eta^{\frac{1}{4}} )^2 .
\\
\end{align*}
\end{lem}
\begin{proof}
By the very definition of the constants $\hC^a_s$, we see that 
\begin{align*}
\hC^0_0 \slash \hC^{-\frac{1}{8}}_0  &= \hC^0_1 \slash \hC^{-\frac{1}{8}}_1 
\\
&= \frac{\eta^{-\frac{1}{2}} + \eta^{\frac{1}{2}} }{ \eta^{-\frac{1}{2}} - \eta^{\frac{1}{2}} } \cdot 
\frac{\eta^{-\frac{1}{4}} - \eta^{\frac{1}{4}} }{ \eta^{-\frac{1}{4}} + \eta^{\frac{1}{4}} }
= \frac{ (\eta^{-\frac{1}{2}} + \eta^{\frac{1}{2}}) }{ (\eta^{-\frac{1}{4}}+\eta^{\frac{1}{4}} )^2 }.
\end{align*}
\t
Thus, the 1st assertion follows. Next,
\begin{align*}
\hlambda_{-1} = \hlambda_{2} &= \frac{\eta^{-\frac{3}{2}} + \eta^{\frac{3}{2}} }{ \eta^{-\frac{1}{2}} - \eta^{\frac{1}{2}} } 
 \;-\;   \frac{ \eta^{-\frac{1}{2}} + \eta^{\frac{1}{2}} }{ (\eta^{-\frac{1}{4}} + \eta^{\frac{1}{4}})^2 } \cdot 
\frac{\eta^{-\frac{3}{4}} + \eta^{\frac{3}{4}} }{ \eta^{-\frac{1}{4}} - \eta^{\frac{1}{4}} }
\\
&=\frac{\eta^{-\frac{1}{2}} + \eta^{\frac{1}{2}} }{ \eta^{-\frac{1}{2}}  - \eta^{\frac{1}{2}} } \cdot
\frac{\eta^{-\frac{3}{2}} + \eta^{\frac{3}{2}} }{ \eta^{-\frac{1}{2}} + \eta^{\frac{1}{2}} }
 \;-\;   \frac{ \eta^{-\frac{1}{2}} + \eta^{\frac{1}{2}} }{ \eta^{-\frac{1}{2}} - \eta^{\frac{1}{2}} } 
\cdot \frac{\eta^{-\frac{3}{4}} + \eta^{\frac{3}{4}} }{ \eta^{-\frac{1}{4}} + \eta^{\frac{1}{4}} }
\\
&= \frac{\eta^{-\frac{1}{2}} + \eta^{\frac{1}{2}} }{ \eta^{-\frac{1}{2}}  - \eta^{\frac{1}{2}} } \{ ( \eta^{-1} - 1 + \eta ) - ( \eta^{-\frac{1}{2}} - 1 + \eta^{\frac{1}{2}} ) \}
\\
&= \frac{\eta^{-\frac{1}{2}} + \eta^{\frac{1}{2}} }{ \eta^{-\frac{1}{2}}  - \eta^{\frac{1}{2}} } \{ ( \eta^{-1} - 2 + \eta ) - ( \eta^{-\frac{1}{2}} - 2 + \eta^{\frac{1}{2}} ) \}
\\
&=  \frac{\eta^{-\frac{1}{2}} + \eta^{\frac{1}{2}} }{ \eta^{-\frac{1}{2}}  - \eta^{\frac{1}{2}} } \{ ( \eta^{-\frac{1}{2}} - \eta^{\frac{1}{2}} )^2 - ( \eta^{-\frac{1}{4}} - \eta^{\frac{1}{4}} )^2 \} .
\end{align*}
\end{proof}

\vspace{-0.5cm}

\begin{defi}[The Constant $\rmh$] Define the constant $\rmh := (\eta^{-\frac{1}{2}} + \eta^{\frac{1}{2}})^{-1}$.
\end{defi}
\begin{rem}\label{rem rmh} Recal that $\eta$ is a primitive $p$th root of unity. Then we see that
\[
\rmh = \pm 1 \;\Leftrightarrow\; \rmh=1 \;\Leftrightarrow\; \eta^3 =1 \;\Leftrightarrow\; p=3
\]
since $p$ is odd and since
\[
(\eta^{-\frac{1}{2}} + \eta^{\frac{1}{2}} -1) (\eta^{-\frac{1}{2}} + \eta^{\frac{1}{2}} +1) = \eta + \eta^{-1} +1.
\] 
\end{rem}
\begin{cor}\label{lambda gamma equation} It holds that
\[
\hlambda_{2} = \hlambda_{-1}  = (\eta^{-\frac{1}{2}} - \eta^{\frac{1}{2}} ) \{ 1+ \rmh \} \,\hgamma_{\!\pm1} .
\]
\end{cor}
\begin{proof}

\vspace{-0.5cm}

\begin{align*}
\hlambda_{2} (\eta^{-\frac{1}{2}} - \eta^{\frac{1}{2}} )^{-1} \slash \hgamma_{\!\pm1} 
&= ( \eta^{-\frac{1}{2}}  + \eta^{\frac{1}{2}} )^{-1} \cdot \left( \frac{ \eta^{-\frac{1}{4}}  +  \eta^{\frac{1}{4}}  }{ \eta^{-\frac{1}{2}}  -  \eta^{\frac{1}{2}} } \right)^2 \cdot 
\left\{ (\eta^{-\frac{1}{2}}  - \eta^{\frac{1}{2}} )^2 - (\eta^{-\frac{1}{4}} -\eta^{\frac{1}{4}} )^2 \right\}
\\
&=  ( \eta^{-\frac{1}{2}}  + \eta^{\frac{1}{2}} )^{-1} \left\{ (\eta^{-\frac{1}{4}}  + \eta^{\frac{1}{4}} )^2 - 1 \right\}
\\
&= ( \eta^{-\frac{1}{2}}  + \eta^{\frac{1}{2}} )^{-1} \left\{ \eta^{-\frac{1}{2}}  + \eta^{\frac{1}{2}}  + 1 \right\}
\\
&= 1 + \rmh .
\end{align*}
\end{proof}
%
%
%
\begin{defi} Denote by $L_{g}^\eta(\hat{0})^\perp$ the complementary distinguished subspace to $L_{g}^\eta(\hat{0})$ in $L_g$, that is, 
\[
L_{g}^\eta(\hat{0})^\perp := \bigoplus_{\hat{a} \in \mathrm{I}_p^g \setminus \{ \hat{0} \} } L_{g}^\eta(\hat{a}).
\]
\end{defi}
We introduce the other useful operators as follows.
\begin{defi} Define the operator $A_k \in \mathcal{B}^\eta_{\hat{0},\hat{0}}$ for $1\leq k\leq g$ as follows;
\[
A_k :=  \frac{p}{\hgamma_1}
\cdot \, \mathrm{Proj}(\mathit{odd}; \hat{0}) \circ  \left\{ \bbphi\left(\hat{\Psi}_{2k,\mathrm{ss}} ; 0 \right) - \frac{(1+\eta)}{( 1+\eta^{\frac{1}{2}} )^2} \bbphi\left(\hat{\Psi}_{2k,\mathrm{ss}} ; -\frac{1}{8} \right) - \hgamma_{\!0} \right\} \circ \mathrm{Proj}(\mathit{odd}; \hat{0}) .
\]
\end{defi}
\begin{prop} Suppose that $\hat{j} \in \hat{\tau}_{p}^{-1} (\hat{0})$. If $1\leq k\leq g-1$, then we have
\begin{align*}
A_k(\bte_{\hat{j}} \bu_{2l-1}) 
 \underset{ \mathrm{mod}\; L_{g}^\eta(\hat{0})^\perp }{\equiv} 
& \eta^{-\frac{1}{2}} \, \bte_{\hat{j} + \re_k -\re_{k+1}} \bu_{2l-1} +  \eta^{\frac{1}{2}} \, \bte_{\hat{j} - \re_k + \re_{k+1}} \bu_{2l-1}
\\
 + & \;\delta_{k,l} \cdot (1+\rmh ) \eta^{-\frac{1}{2}} \, \btw^{\hat{j} + \re_k -\re_{k+1}}_{2k-1}
\\
 + & \;\delta_{k,l} \cdot (1+\rmh ) \eta^{\frac{1}{2}} \, \btw^{\hat{j} - \re_k +\re_{k+1}}_{2k+1} .
\end{align*}
If $k=g$, then we have
\begin{align*}
A_g(\bte_{\hat{j}} \bu_{2l-1}) 
 \underset{ \mathrm{mod}\; L_{g}^\eta(\hat{0})^\perp }{\equiv} 
& \eta^{-\frac{1}{2}} \, \bte_{\hat{j} + \re_g} \bu_{2l-1} +  \eta^{\frac{1}{2}} \, \bte_{\hat{j} - \re_g} \bu_{2l-1}
\\
 + & \;\delta_{g,l} \cdot (1+\rmh ) \eta^{-\frac{1}{2}} \, \btw^{\hat{j} + \re_g}_{2g-1}
\\
 + & \;\delta_{g,l} \cdot (1+\rmh ) (\eta^{-\frac{1}{2}}-\eta^{\frac{1}{2}}) \eta \, \btv_{2g}^{\hat{j} - \re_g } .
\end{align*} 
In particular, $A_k$ preserves $L^\eta_{g\mathit{odd}}(\hat{0}) \subset L_{g}^\eta(\hat{0})$ if $1\leq k\leq g-1$. 
\end{prop}
\begin{proof}
We will make use of Prop \ref{important odd}. By linearity, if $1\leq k\leq g-1$, then we have
\begin{align*}
 \hgamma_1 \cdot A_k (\be_{\hat{j}} \bu_{2l-1})
\underset{ \mathrm{mod}\; L_{g}^\eta(\hat{0})^\perp }{\equiv}  
& \; \eta^{-\frac{1}{2}} \hgamma_{\!1} \be_{\hat{j} + \re_k -\re_{k+1}} \bu_{2l-1} + \eta^{\frac{1}{2}} \hgamma_{\!-1} \be_{\hat{j} - \re_k + \re_{k+1}} \bu_{2l-1}
\\
+ & \;\delta_{k,l} \cdot \hlambda_{2} ( \eta -1)^{-1} \,\btw^{\hat{j} + \re_k -\re_{k+1}}_{2k-1}
\\
+ & \;\delta_{k,l} \cdot \hlambda_{-1}( \eta^{-1} -1 )^{-1} \,\btw^{\hat{j} - \re_k +\re_{k+1}}_{2k+1} .
\end{align*}
If $k=g$, then we have
\begin{align*}
 \hgamma_1 \cdot A_g (\be_{\hat{j}} \bu_{2l-1})
\underset{ \mathrm{mod}\; L_{g}^\eta(\hat{0})^\perp }{\equiv} 
& \; \eta^{-\frac{1}{2}} \hgamma_{\!1} \be_{\hat{j} + \re_g} \bu_{2l-1} + \eta^{\frac{1}{2}} \hgamma_{\!-1} \be_{\hat{j} - \re_g } \bu_{2l-1}
\\
+ & \;\delta_{g,l} \cdot \hlambda_{2} ( \eta -1)^{-1} \,\btw^{\hat{j} + \re_g }_{2g-1}
\\
+ & \;\delta_{g,l} \cdot \hlambda_{-1} \eta \, \btv^{\hat{j} - \re_g}_{2g} .
\end{align*}
In either case, Corollary \ref{lambda gamma equation} shows the result. \\
The reason why the coefficients of all the $\btv$-terms vanish in the case where $1\leq k\leq g-1$ is as follows; Corollary \ref{important odd cor} implies that
\[
\hgamma_1 \cdot A_k(\be_{\hat{j}} \bu_{2l-1} )
\equiv  \delta_{l,k}\, \sum_{ s \in \Z/(p)} \hlambda_s \,\bv(\hat{j} + \re_{k+1}, \,k,\,s) \quad \mathrm{mod}\;\; L^{\eta}_{g\,\mathrm{odd}} + L^{\eta}_{g}(\hat{0})^{\perp}.
\]
Each term in the sum vanishes if $s=0,1$ since $\hlambda_0 = 0 =\hlambda_1$.  
On the other hand, if $s\neq 0,1$, a moment thought shows that $\bv(\hat{j} + \re_{k+1}, \,k,\,s) \subset L^{\eta}_{g}(\hat{0})^{\perp}$ since $(j_k,\ j_{k+1}) + (s, 1-s) \not\in \{0,-1\}^2$.
\end{proof}


As a corollary, we obtain the following formulae;
\begin{prop}\label{ak odd} Suppose that $\hat{j} \in \hat{\tau}_p^{-1}(\hat{0})$.
If $1\leq k \leq g-1$, then it holds that, for $1\leq l\leq g$,
\begin{align*}
A_k(\btw^{\hat{j}}_{2l-1})  & \equiv 
\begin{cases}
\eta^{-\frac{1}{2}}\, \btw^{\hat{j} + \re_k -\re_{k+1}}_{2l-1} +  \eta^{\frac{1}{2}}\, \btw^{\hat{j} - \re_k + \re_{k+1}}_{2l-1} & \text{ if } l \neq k,k+1,
\\
 \eta^{\frac{1}{2}}\, \btw^{\hat{j} - \re_k +\re_{k+1}}_{2k-1}
-\rmh \eta^{-\frac{1}{2}}\, \btw^{\hat{j} + \re_k -\re_{k+1}}_{2k-1} 
+(1+\rmh ) \eta^{\frac{1}{2}}\, \btw^{\hat{j} - \re_k +\re_{k+1}}_{2k+1}
&  \text{ if } l= k,
\\
 \eta^{-\frac{1}{2}}\, \btw^{\hat{j} + \re_k - \re_{k+1}}_{2k+1}
-\rmh \eta^{\frac{1}{2}}\, \btw^{\hat{j} - \re_k + \re_{k+1}}_{2k+1} 
+(1+\rmh ) \eta^{-\frac{1}{2}}\, \btw^{\hat{j} + \re_k - \re_{k+1}}_{2k-1}
&  \text{ if } l= k+1
\end{cases}
\\
&  \mathrm{mod} \; L_{g}^\eta(\hat{0})^\perp .
\end{align*}
\t
If $k=g$, then it holds that, for $1\leq l\leq g$,
\[
A_g(\btw^{\hat{j}}_{2l-1})  \underset{ \mathrm{mod}\; L_{g}^\eta(\hat{0})^\perp }{\equiv} 
\begin{cases}
\eta^{-\frac{1}{2}}\, \btw^{\hat{j} + \re_g}_{2l-1} +  \eta^{\frac{1}{2}}\, \btw^{\hat{j} - \re_g }_{2l-1} & \text{ if } l \neq g ,
\\
\eta^{\frac{1}{2}}\, \btw^{\hat{j} - \re_g}_{2g-1}
-\rmh \eta^{-\frac{1}{2}}\, \btw^{\hat{j} + \re_g}_{2g-1} + (1+\rmh ) ( \eta^{-\frac{1}{2}}-\eta^{\frac{1}{2}} ) \eta \, \btv_{2g}^{\hat{j} - \re_g}
&  \text{ if } l= g .
\end{cases}
\]
\end{prop}
\begin{cor} If $1\leq k \leq g-1$, then it holds that
\[
A_k(\lgo^\eta(\hat{0})) \subset \lgo^\eta(\hat{0}) .
\]
\end{cor}
%
%
Unfortunately, $A_g$ does not preserve $L_{g\, \mathrm{odd}}^\eta(\hat{0})$. To compensate this disadvantage, we will introduce the other useful operators, which are the key in the subsequent subsection.
\begin{defi}
Set
\[
T_g := (- \eta^{-\frac{1}{2}} \rmh )^{-1}A_g \circ B_g, \quad T_g^{\dagger} := (-\eta^{\frac{1}{2}} \rmh)^{-1} A_g \circ B_g^{\dagger}.
\]
\end{defi}
\begin{prop} For $\hat{j} \in \hat{\tau}_p^{-1}(\hat{0})$ and for $1\leq l\leq g$, it holds that 
\begin{align*}
T_g(\btw^{\hat{j}}_{2l-1})  & = \delta_{g,l}
\begin{cases}
\btw^{\hat{j}}_{2g-1} 
&  \text{ if }  j_g=0,
\\
\btw^{\hat{j} + \re_g}_{2g-1} 
& \text{ if }  j_g=-1.
\end{cases}
\end{align*}
\end{prop}
\begin{proof} This follows from Proposition \ref{operator B} and \ref{ak odd}.
\end{proof}
\begin{prop} For $\hat{j} \in \hat{\tau}_p^{-1}(\hat{0})$ and for $1\leq l\leq g$, it holds that
\begin{align*}
T_g^{\dagger}(\btw^{\hat{j}}_{2l-1}) =  0 ,
\quad \;
T_{g}^{\dagger}(\btv^{\hat{j}}_{2l})  & \, \underset{\mathrm{mod} \, \lgo^{\eta} }{\equiv} \; \delta_{g,l}
\begin{cases}
 \btv^{\hat{j}}_{2g} - \rmh^{-1} \btv^{\hat{j} - \re_g}_{2g}  &  \text{ if } j_g=0,
\\
0  &  \text{ if } j_g=-1 .
\end{cases}
\end{align*}
\end{prop}
\begin{proof} This follows from Proposition \ref{operator B dag} and \ref{ak even} below.
\end{proof}
\begin{cor} It is readily seen that 
\begin{align*}
& T_g(\lgo^\eta(\hat{0})) \subset \lgo^\eta(\hat{0}), \quad T_g^{\dagger}(L_{g\,\mathrm{odd}}^\eta(\hat{0})) = 0.
\end{align*}
Further, both $T_g |_{\lgo^\eta(\hat{0})}$ and the induced action of $T_g^{\dagger}$ on $L_g^\eta(\hat{0}) / \lgo^\eta(\hat{0})$ are idempotents.
\end{cor}

\begin{prop}\label{ak even} Suppose that $\hat{j} \in \hat{\tau}_p^{-1}(\hat{0})$. If $1\leq k \leq g-1$, then it holds that, for $1\leq l\leq g$,
\begin{align*}
A_k(\btv_{2l}^{\hat{j}})  \equiv & 
\begin{cases}
\eta^{-\frac{1}{2}}\, \btv_{2l}^{\hat{j} + \re_k -\re_{k+1}}  
+ \eta^{\frac{1}{2}}\, \btv_{2l}^{\hat{j} - \re_k + \re_{k+1}} 
& \text{ if } l\neq k,k+1,
\\
\eta^{\frac{1}{2}}\, \btv_{2k}^{\hat{j} - \re_k + \re_{k+1}}  
- \rmh \eta^{-\frac{1}{2}}\, \btv_{2k}^{\hat{j} + \re_k -\re_{k+1}}  
+ (1+\rmh ) \eta^{-\frac{1}{2}}\, \btv_{2k+2}^{\hat{j} + \re_k -\re_{k+1}}
& \text{ if } l= k,
\\
\eta^{-\frac{1}{2}}\, \btv_{2k+2}^{\hat{j} + \re_k -\re_{k+1}}  
- \rmh \eta^{\frac{1}{2}}\, \btv_{2k+2}^{\hat{j} - \re_k + \re_{k+1}} 
+ (1+\rmh ) \eta^{\frac{1}{2}}\, \btv_{2k}^{\hat{j} - \re_k + \re_{k+1}}
& \text{ if } l= k+1
\end{cases}
\\
& \mathrm{mod} \;\; L^\eta_{g\mathrm{odd}} + L_{g}^\eta(\hat{0})^\perp.
\end{align*}
If $k=g$, then it holds that, for $1\leq l\leq g$,
\begin{align*}
A_g(\btv_{2l}^{\hat{j}})  \underset{ \mathrm{mod} \; L^\eta_{g\mathrm{odd}} + L_{g}^\eta(\hat{0})^{\perp}  }{\equiv}  
\begin{cases}
\eta^{-\frac{1}{2}}\, \btv_{2l}^{\hat{j} + \re_g}  +  \eta^{\frac{1}{2}}\, \bhv_{2l}^{\hat{j} - \re_g} & \text{ if } l\neq k ,
\\
\eta^{\frac{1}{2}}\, \bhv_{2g}^{\hat{j} - \re_g}  -  \rmh \eta^{-\frac{1}{2}}\, \bhv_{2g}^{\hat{j} + \re_g}
& \text{ if } l = k .
\end{cases}
\end{align*}
\end{prop}
\begin{proof} The result follows from Proposition \ref{important even} in almost the same fashion as in the proof of Proposition \ref{ak odd}.  
\end{proof}
Since $A_k\;(1\leq k\leq g-1)$ preserves the odd subspace $L^\eta_{g\mathrm{odd}}(\hat{0}) \subset L_{g}^\eta(\hat{0})$, we can consider  the induced action of $A_k \;(1\leq k\leq g-1)$ on the "even quotient"  $L_{g}^\eta(\hat{0}) / L^\eta_{g\mathrm{odd}}(\hat{0})$. A short calculation shows that
\begin{cor} For $1\leq k\leq g-1$, the minimal polynomials of the actions of $A_k$ on $L^\eta_{g\mathrm{odd}}(\hat{0})$ and on $L_{g}^\eta(\hat{0}) / L^\eta_{g\mathrm{odd}}(\hat{0})$ are equal to
\[
 t(t^2 - \rmh^2 )( t^2 -1) .
\]
\end{cor}


\subsection{Proof of the Main Theorem}\label{subsection main thm} 

We will devote the whole of the present subsection to prove the following theorem. Recall that $\mathcal{B}^\eta_{\hat{0},\hat{0}} \subset \mathrm{End}(\eta\, ; \hat{0})$ is a subalgebra with $1_{\mathcal{B}^\eta_{\hat{0},\hat{0}}} =1_{\mathrm{End}(\eta\,; \hat{0})}$.
\begin{thm}[Main Theorem]\label{main thm} If $p$ is a prime greater than 3 and $g \geq 2$, then the subalgebra $\mathcal{B}^\eta_{\hat{0},\hat{0}}$ coincides with $\mathrm{End}(\eta\,; \hat{0})$.
\end{thm}

First of all, we agree with the following convention;
\begin{conv*} If we refer to an algebra, say, $A$, we did not assume that $A$ is unital unless otherwise specified. If we have a subset $X$ of an algebra $A$, the subalgebra $B$ that $X$ generates means the minimal subalgebra that contains $X$. Therefore, whether $B$ is unital or not depends on the situation.
\end{conv*}
\begin{nota*} For simplicity, we adopt the following notation;  
\begin{align*}
\cR &:= \mathcal{B}^\eta_{\hat{0},\hat{0}}, 
\quad L_g := L_{g}^\eta(\hat{0}), 
\quad \lgo := L_{g,\,\mathrm{odd}}^\eta(\hat{0}),
\end{align*}
\end{nota*}
Consider the filtration $F: \{0\} \subset \lgo \subset L_{g}$. 
Set $\cR^{F}$ to be the subalgebra of $\cR$ consisting of the elements that preserve $F$. The actions of $\cR^{F}$ on $\lgo$ and on $L_g/\lgo$ induce the following two maps respectively;
\begin{align*}
\pi &: \cR^{F} \to \mathrm{End}_{\C}(\lgo), 
\quad
\pi^{\prime} : \cR^{F} \to \mathrm{End}_{\C}( L_{g} \slash L_{g,\mathrm{odd}} ) .
\end{align*}
\begin{defi} Define the two algebras $\cro$ and $\cre$ respectively as follows;
\begin{align*}
\cro &:= \pi(\cR^{F}) \subset \mathrm{End}_{\C}(\lgo),
\quad
\cre := \pi^{\prime}(\cR^{F}) \subset \mathrm{End}_{\C}( L_{g} \slash L_{g,\mathrm{odd}} ).
\end{align*} 
\end{defi}
%
\t
Under the same assumption as Theorem \ref{main thm}, we propose two assertions;
\begin{thm}\label{main assertion1} The subalgebra $\cro$ coincides with $\mathrm{End}_{\C}(\lgo)$.
\end{thm}
\begin{thm}\label{main assertion2} The subalgebra $\cre$ coincides with $\mathrm{End}_{\C}(L_{g}/\lgo)$. 
\end{thm}
The proof of these two are almost the same. We will prove Theorem \ref{main assertion1} first. As for Theorem \ref{main assertion2}, we will state the outline with some minor changes needed there and omit the detail to avoid unnecessary repetition. Theorem \ref{main thm} follows as a corollary of Theorem \ref{main assertion1} and Theorem \ref{main assertion2}.
\begin{defi} For the distinguished basis 
\[
 \left\{ \btw_{2k-1}^{\hat{j}}  \mid \; 1\leq k \leq g,\;  \hat{j} \in \hat{\tau}^{-1}(\hat{0}) \right\}
\]  
of $\lgo$, the corresponding distinguished endomorphisms $\mathrm{H}^{\hat{i},\hat{j}}_{l,m} \in \mathrm{End}_{\C}( L_{g})$
are determined by
\[
\mathrm{H}^{\hat{i},\hat{j}}_{2l-1,2m-1}(\btw_{2n-1}^{\hat{k}}) := \delta_{m,n}\, \delta^{\hat{j},\hat{k}} \, \btw_{2l-1}^{\hat{i}} \quad ( \hat{i},\, \hat{j},\,\hat{k}  \in \hat{\tau}^{-1}(\hat{0}),\; 1\leq l,\,m,\,n \leq g ) .
\]
In particular, denote the corresponding idempotents as follows; 
\[
\mathrm{E}^{\hat{i}}_{l} \equiv \mathrm{H}^{\hat{i},\hat{i}}_{2l-1,2l-1} \quad ( \hat{i} \in \hat{\tau}^{-1}(\hat{0}),\; 1\leq l \leq g ) .
\]
\end{defi}
\begin{nota*} For two multi-indexes $\hat{j}_1$ of length $k$ and $\hat{j}_2$ of length $l$, the juxtaposition $\hat{j}_1 \hat{j}_2$ is regarded as a multi-index of length $k+l$.
\end{nota*}
We will add some notations. 
\begin{defi} For $h\;(1\leq h\leq g)$, set
\begin{align*}
W_{\leq h} &:= \langle \btw_{2l-1}^{\hat{j}} \mid 1\leq l \leq h,\; \hat{j} \in \{0,-1\}^g \rangle_{\C},
\\
W_h &:= \langle \btw_{2h-1}^{\hat{j}} \mid \hat{j} \in \{0,-1\}^g \rangle_{\C} .
\\
W_{\geq h} &:= \langle \btw_{2l-1}^{\hat{j}} \mid h\leq l \leq g,\; \hat{j} \in \{0,-1\}^g \rangle_{\C}.
\end{align*}
The subspaces $W_{< h}$ and $W_{> h}$ are defined in the same fashion. For example, we see that $L_{g\,\mathrm{odd}} = W_{<h} \oplus W_{\geq h}$. 
The distinguished projection to $W_{\geq h}$ and $W_{h}$ are written  respectively as  
\begin{align*}
P_{\geq h} &:= \sum_{l=h}^{g} \sum_{ \hat{\alpha} \in \{0,-1 \}^{g} } \mathrm{H}^{\hat{\alpha},\, \hat{\alpha} }_{2l-1,2l-1} ,
\\
P_{h} &:= \sum_{ \hat{\alpha} \in \{0,-1 \}^{g} } \mathrm{H}^{\hat{\alpha},\, \hat{\alpha} }_{2h-1,2h-1}.
\end{align*}
\end{defi}

%
We will use the Burnside theorem below everywhere without notice.
\begin{thm}[Burnside Theorem]
Let $A$ be a $\C$-algebra and $V$ a $A$-module. Denote by $\mathrm{res}_{V} : A \to \mathrm{End}_{\C}(V)$ the map induced by the $A$-action on $V$. Assume that $\mathrm{res}_{V} \neq 0$. Then $V$ is an irreducible $A$-module if and only if $\mathrm{res}_{V}$ is surjective.
\end{thm} 
We provide some useful lemmas below, which are corollaries of the Burnside theorem.
\begin{lem}\label{lem1 useful}
Let $V_i\;(i=1,2,3)$ be finite dimensional $\C$-vector spaces. Consider the direct sum decomposition
\[
\mathrm{End}_{\C}(V_1\oplus V_2 \oplus V_3) = \bigoplus_{1\leq i, j \leq 3} \mathrm{Hom}_{\C}(V_{i},V_{j}).
\]
Suppose that
\[
f \in \mathrm{Hom}_{\C}(V_{1},V_{2}) \oplus \mathrm{End}_{\C}(V_{3},V_{3}) \;\text{ and }\; g \in \mathrm{Hom}_{\C}(V_{2},V_{1}) \oplus \mathrm{End}_{\C}(V_{3},V_{3})
\]
satisfy the condition that $f|_{V_{1}} : V_{1} \to V_{2}$ is surjective and $g|_{V_2} : V_{2} \to V_{1}$ is injective. Then the subalgebra $\mathcal{S}$ generated by the subset $\mathrm{End}_{\C}(V_1) \cup \{f,\, g\} \subset \mathrm{End}_{\C}(V_1\oplus V_2 \oplus V_3)$ contains $\mathrm{End}_{\C}(V_1\oplus V_2)$.
\end{lem}
\begin{proof} The statement is obviously true if either $V_1$ or $V_2$ is equal to $\{0\}$. Thus we may assume that neither of them be $\{0\}$. The two conditions on $f$ and $g$ together imply that
\[
\mathrm{End}_{\C}(V_2) = f \circ \mathrm{End}_{\C}(V_1) \circ g .
\]
Therefore, $\mathcal{S}$ contains the subalgebra $\mathcal{T}:=\mathrm{End}_{\C}(V_1) \oplus \mathrm{End}_{\C}(V_2)$. It is enough to show that $V_1 \oplus V_2$ is irreducible as a $\mathcal{S}$-module. Let $M \subset V_1\oplus V_2$ be a minimal non zero $\mathcal{S}$-submodule. Since $V_1$ and $V_2$ are irreducible $\mathcal{T}$-modules not isomorphic to each other,  $M$ must contain at least one of them. But none of them can be a $\mathcal{S}$-submodule because $f(V_1)\setminus V_1 \neq \emptyset \neq g(V_2)\setminus V_2$. Thus $M$ must contain both of them, which implies that $M=V\oplus W$. Thus we are done.    
\end{proof}
\begin{lem}\label{lem2 useful}
Let $U, V_1,\dots,V_k, W$ be finite dimensional $\C$-vector spaces. Suppose that
we have 
\[
f_i \in \mathrm{Hom}_{\C}(U,V_i) \oplus \mathrm{End}_{\C}(W), \;\, g_i \in \mathrm{Hom}_{\C}(V_i, U) \oplus \mathrm{End}_{\C}(W) \quad (1\leq i \leq l)
\]
satisfying the condition that all $f_i|_{U} : U \to V_i$ are surjective and all $g_i|_{V_i} : V_i \to U$ are injective. Then the subalgebra $\mathcal{S}$ generated by the subset
\[
\mathrm{End}_{\C}(U) \cup \{ f_i,\,g_i \mid 1\leq i \leq k \} \subset \mathrm{End}_{\C}(U \oplus \bigoplus_{1\leq i \leq k } V_i \oplus W)  
\]
contains
\[
\mathrm{End}_{\C}(U \oplus \bigoplus_{1\leq i \leq k } V_i ) .
\]
\end{lem}
\begin{proof} Applying Lemma \ref{lem1 useful} succesively, one can show that 
\[
\mathrm{End}_{\C}(U \oplus \bigoplus_{1\leq i \leq l} V_i ) \subset \mathcal{S}
\]
by an upward inductive argument with respect to $l =1,2,\dots, k$.
\end{proof}
%
%
%
%
%
%
%
%
%
\t
{\bf Proof of Theorem \ref{main assertion1}.)} The proof divides into 2 steps.
\\
\t
{\bf 1st Step.)} We will show the following assertion:
\begin{prop}\label{prop 1st step} $\cro$ contains the $4$-dimensional subalgebra spanned by the distinguished morphisms below; 
\[
\sum_{\hbeta \in \{ 0,-1\}^{g-1}} \mathrm{H}_{2g-1,\,2g-1}^{\hbeta(i),\, \hbeta(j)} \quad ( i,j \in \{0,-1\}).
\]
In particular, $\cro$ includes the distinguished projection $P_g: \lgo = W_{<g} \oplus W_g \twoheadrightarrow W_g$.
\end{prop}
\begin{rem} Proposition \ref{prop 1st step} implies that, for each $\hbeta \in \{ 0,-1\}^{g-1}$,  
\[
W_g(\hbeta) := \langle \btw_{2g-1}^{\hbeta(i)} \mid i \in \{ 0,{-1}\} \,\rangle_{\C}
\]
is an irreducible and faithful $P_g \circ \cro \circ P_g$-module.
\end{rem}
\begin{defi} For any multi-index $\halpha \in \{0,-1\}^{g-2}$, define the following distinguished subspaces of $\lgo$; 
\begin{align*}
W_{g}^{\dagger}(\halpha) &:= \left\langle \btw_{2g-3}^{\halpha\,(-1,0)},\; \btw_{2g-3}^{\halpha\,(0,-1)},\; \btw_{2g-1}^{\halpha\,(-1,0)},\; \btw_{2g-1}^{\halpha\,(0,-1) },\; \btw_{2g-1}^{\halpha\,(-1,-1)},\; \btw_{2g-1}^{\halpha\,(0,0)} \right\rangle_{\C},
\\
W_{g}^{\dagger} &:= \bigoplus_{\halpha \in \{0,-1\}^{g-2}} W_{g}^{\dagger}(\halpha) \quad \subset \; W_{g-1} \oplus W_{g} .
\end{align*}
\end{defi}
\begin{defi} Denote by $\mathcal{W}$ the subalgebra of $\cro$ generated by $\{
\pi(A_{g-1}),\, \pi(B_g),\; \pi(T_g) \}$.
\end{defi}
Notice that each $W_{g}^{\dagger}(\halpha)$ is a $\mathcal{W}$-submodule and that the distinguished complement $W_{g}^{\dagger\,\perp}$ to $W_{g}^{\dagger} \subset \lgo$ is preserved under the $\mathcal{W}$-action. 
\begin{defi} Let
\[
\mathrm{res}_{\halpha}^{\dagger} : \mathcal{W} \to \mathrm{End}_{\C}(W_{g}^{\dagger}(\halpha))
\]
be the map induced by the $\mathcal{W}$-action on $W_{g}^{\dagger}(\halpha)$. 
\end{defi}
Notice that there is the canonical (obvious) identification $\tau_{\hbeta,\,\halpha} : W_{g}^{\dagger}(\halpha) \overset{\equiv}{\rightarrow} W_{g}^{\dagger}(\hbeta)$ for any $\halpha, \hbeta \in \{0,-1\}^{g-2}$, which turns out to be a $\mathcal{W}$-module isomorphism.
%
\begin{defi} For each $\halpha \in \{0,-1\}^{g-2}$, define the distinguished subspace $W_{g}^{\diamond}(\halpha) \subset W_{g}^{\dagger}(\halpha)$ to be that spanned by the distinguished subbasis
\[
\{ \btw_{2g-3}^{\halpha\,(-1,0)},\; \btw_{2g-3}^{\halpha\,(0,-1)},\; \btw_{2g-1}^{\halpha\,(-1,0)},\; \btw_{2g-1}^{\halpha\,(0,-1) } \} .
\]
Notice that we have the canonical inclusion $\mathrm{End}_{\C}(W_{g}^{\diamond}(\halpha)) \subset \mathrm{E}_{\C}(W_{g}^{\dagger}(\halpha))$ due to the existence of the distinguished compliment $W_{g}^{\diamond}(\halpha)^{\perp} \subset W_{g}^{\dagger}(\halpha)$ to $W_{g}^{\diamond}(\halpha)$.
\end{defi}
%
%
%
\begin{lem}\label{lem1 1st step} It holds that
$\mathrm{End}_{\C}(W_{g}^{\diamond}(\halpha))  \subset \mathrm{Image}(\mathrm{res}_{\halpha}^{\dagger})$.
\end{lem}
\begin{proof} Fix an $\halpha \in \{0,-1\}^{g-2}$. Set 
\[
A:= \mathrm{res}_{\halpha}^{\dagger}(\pi(A_{g-1})),\; B:= \mathrm{res}_{\halpha}^{\dagger}(\pi(B_g)),\; T:= \mathrm{res}_{\halpha}^{\dagger}(\pi(T_g)) \; \in \; \mathrm{End}_{\C}(W_{g}^{\dagger}(\halpha)).
\]
Notice that $W_{g}^{\diamond}(\halpha)$ is $A$-invariant and that 
$A|_{W_{g}^{\diamond}(\halpha)^{\perp} } =0$. 
\\
The representation matrix $M_{A}$ of $A|_{W_{g}^{\diamond}(\halpha)}$ is
\[
M_{A}=
\left( \begin{array}{cccc} 
0 & -\rmh & 0 & 0 \\
1 & 0 & 1+\rmh & 0 \\
0 & 1+\rmh & 0 & 1 \\
0 & 0 & -\rmh & 0 
\end{array}  \right)
\left( \begin{array}{cccc} 
\eta^{\frac{1}{2}} & 0 & 0 & 0 \\
0 & \eta^{-\frac{1}{2}} &0 & 0 \\
0 & 0 & \eta^{\frac{1}{2}} & 0 \\
0 & 0 &0 & \eta^{-\frac{1}{2}}
\end{array}  \right) .
\]
The characteristic polynomial of $M_{A}$ is $(t^2-1)(t^2-\rmh^2)$, which has no multiple root due to the assumption that $p \neq 3$. Thus $M_{A}$ is diagonalized by the invertible matrix
\[
P= 
\left( \begin{array}{cccc} 
\eta^{-\frac{1}{2}} & 0 & 0 & 0 \\
0 & \eta^{\frac{1}{2}} &0 & 0 \\
0 & 0 & \eta^{-\frac{1}{2}} & 0 \\
0 & 0 &0 & \eta^{\frac{1}{2}}
\end{array}  \right)
\left( \begin{array}{rrrr} 
\rmh & \rmh & 1 & 1 \\
-1 & 1 & -1 & 1 \\
-1 & -1 & -1 & -1 \\
\rmh & -\rmh & 1 & -1 
\end{array}  \right)
\]
as follows;
\[
P^{-1} M_{A} P =\mathrm{diag}(1,-1,\rmh,-\rmh) .
\]
A short calculation shows that
\[
P^{-1}= \frac{1}{2(1-\rmh)}
\left( \begin{array}{rrrr} 
-1 & -1 & -1 & -1 \\
-1 & 1 & -1 & 1 \\
1 & \rmh & \rmh & 1 \\
1 & \rmh & -\rmh & -1  
\end{array}  \right)
\left( \begin{array}{cccc} 
\eta^{\frac{1}{2}} & 0 & 0 & 0 \\
0 & \eta^{-\frac{1}{2}} &0 & 0 \\
0 & 0 & \eta^{\frac{1}{2}} & 0 \\
0 & 0 &0 & \eta^{-\frac{1}{2}}
\end{array}  \right)_.
\]
Since the representation matrix of $A$ is semisimple, there exists some polynomial $f_{\lambda}(t)$ of complex coefficients for each $\lambda \in \{ 0, \pm 1,\, \pm\rmh\} $ such that the projection $\mathrm{P}(\lambda) : W_{g}^{\dagger}(\halpha) \twoheadrightarrow \mathrm{E}_{A}(\lambda)$ is equal to $f_{\lambda}(A)$, where $\mathrm{E}_{A}(\lambda)$ denotes the eigenspace of $A$ for the eigenvalue $\lambda$. In particular, the distinguished projection $\mathrm{P}^{\diamond}(\halpha) : W_{g}^{\dagger}(\halpha) \twoheadrightarrow W_{g}^{\diamond}(\halpha)$ coincides with the sum $\sum_{\lambda = \pm 1,\, \pm\rmh} f_{\lambda}(A)$, which shows that   $\mathrm{P}^{\diamond}(\halpha) \in \mathrm{Image}(\mathrm{res}_{\halpha}^{\dagger})$. 
\\
Set $T^{\prime} := \mathrm{P}^{\diamond}(\halpha) \circ T \circ \mathrm{P}^{\diamond}(\halpha)$. It is readily seen that the representation matrix of $T^{\prime}$ is equal to $\mathrm{diag}(0,0,0,1,0,0)$. It follows that  
$T^{\prime}$ is an idempotent of $\mathrm{rank}\,1$ such that\\
\textbullet\, $T^{\prime}\circ \mathrm{P}(\lambda)$ does not vanish for $\lambda = \pm 1,\, \pm\rmh$.\\
\textbullet\, $\mathrm{P}(\lambda) \circ T^{\prime}$ does not vanish for $\lambda = \pm 1,\, \pm\rmh$. \\
In fact, all the entries of the square matrix $C$ below are non-zero;
\[
C := P^{-1} \mathrm{diag}(0,0,0,1) P = (\text{4th column of } P^{-1}) (\text{4th row of } P) 
\]
Denote by $c_{\lambda,\mu}$ the $(\lambda,\mu)$-entry of the matrix $C$ and set 
\[
\mathrm{E}_{\lambda,\mu} := c_{\lambda,\mu}^{-1} \,\mathrm{P}(\lambda) \circ T^{\prime} 
\circ \mathrm{P}(\mu) = c_{\lambda,\mu}^{-1} \,\mathrm{P}(\lambda) \circ T 
\circ \mathrm{P}(\mu) \quad (\lambda,\, \mu = \pm 1,\, \pm\rmh) .
\]
It follows that  
\[
\mathrm{E}_{\kappa, \lambda} \circ \mathrm{E}_{\mu,\nu} = \delta_{\mu,\nu} \mathrm{E}_{\kappa,\nu}
\]
, which implies that $\{ \mathrm{E}_{\lambda,\mu} \mid \lambda, \mu \in \{ \pm 1,\, \pm\rmh\} \}$ spans the $16$-dimensional subspace $\mathrm{End}_{\C}(W_{g}^{\diamond}(\halpha) ) \subset \mathrm{End}_{\C}(W_{g}^{\dagger}(\halpha))$. Notice that each $\mathrm{E}_{\lambda,\mu}$ belongs to $\mathrm{Image}(\mathrm{res}_{\halpha}^{\dagger})$ by its construction. Thus we are done.
\end{proof}
Define the distinguished subspace 
\[
W_{g}^{\diamond} := \bigoplus_{\halpha\in \{0,-1\}^{g-2}}\, W_{g}^{\diamond}(\halpha) \quad \subset \;  W_{g-1} \oplus W_{g}.
\]
%
\begin{cor}\label{cor1 1st step} 
There exists a $16$-dimensional subalgebra $\mathcal{M}^{\diamond} \subset \mathcal{W}$ such that
\\
\textbullet\, $W_{g}^{\diamond\,\perp}$ vanishes under the $\mathcal{M}^{\diamond}$-action, where $W_{g}^{\diamond\,\perp}$ is the distinguished complement in $\lgo$ of $W_{g}^{\diamond}$.
\\
\textbullet\, It holds that $\mathrm{res}_{\halpha}^{\dagger}(\mathcal{M}^{\diamond}) = \mathrm{End}_{\C}(W_{g}^{\diamond}(\halpha))$ for any $\halpha\in \{0,-1\}^{g-2}$.
\end{cor}
\begin{rem} The corollary above implies that $W_{g}^{\diamond}(\halpha)$ is an irreducible and faithful $\mathcal{M}^{\diamond}$-module.
\end{rem}
\begin{proof} 
Notice that the whole argument in the proof of Lemma \ref{lem1 1st step} is independent of the multi-index $\halpha$. We have obtained some subalgebra $\mathcal{W}^{\prime} \subset \mathcal{W}$ there that satisfies the following three properties;
\\ 
\textbullet\, Both $W_{g}^{\dagger}$ and $W_{g}^{\dagger\,\perp}$ are preserved under the $\mathcal{W}^{\prime}$-action, where the latter is the distinguished complement in $\lgo$ of $W_{g}^{\dagger}$.
\\
\textbullet\, $\mathrm{res}_{\halpha}^{\dagger}(\mathcal{W}^{\prime}) =\mathrm{End}_{\C}(W_g^{\diamond}(\halpha) ) \subset \mathrm{End}_{\C}(W_g^{\dagger})$.
\\
\textbullet\, $\mathcal{W}^{\prime}$ is contained in the ideal $\mathcal{M}:=\mathcal{W} \circ \pi(T_{g}) \circ \mathcal{W} \subset \mathcal{W}$.
\\
The 3rd condition together with the 1st one implies that $W_g^{\dagger\,\perp}$ vanishes under the $\mathcal{M}$-action. In fact, $W_g^{\dagger\,\perp}$ is preserved under the $\mathcal{W}$-action and killed by $\pi(T_{g})$. 
Thus it follows that
\\
\textbullet\, $W_{g}^{\diamond\,\perp}$ vanishes under the $\mathcal{W}^{\prime}$-action, where $W_{g}^{\diamond\,\perp}$ is the distinguished complement of $W_{g}^{\diamond}$ in $\lgo$.
\\
Set $\mathcal{M}^{\diamond} := \mathcal{W}^{\prime}$, which is what we seek for. Thus we are done .
\end{proof} 
\begin{defi} For each $\halpha \in \{0,-1\}^{g-2}$, define the distinguished subspace $W_g^{\triangledown}(\halpha) \subset W_g^{\dagger}(\halpha)$ to be the one spanned by the distinguished subbasis
\[
\left\{ \btw_{2g-1}^{\halpha\,(-1,0) },\; \btw_{2g-1}^{\halpha\,(0,-1) },\; \btw_{2g-1}^{\halpha\,(-1,-1)},\; \btw_{2g-1}^{\halpha\,(0,0)} \right\} .
\]
\end{defi}
\begin{lem}\label{lem2 1st step} It holds that
$\mathrm{End}_{\C}(W_{g}^{\triangledown}(\halpha)) \subset \mathrm{Image}(\mathrm{res}_{\halpha}^{\dagger})$.
\end{lem}
\begin{proof} Fix an $\halpha \in \{0,-1\}^{g-2}$. Corollary \ref{cor1 1st step} implies that the set of the distinguished endomorphisms
\[
\left\{ \mathrm{E}_{2g-1}^{\halpha(-1,0)},\; \mathrm{E}_{2g-1}^{\halpha(0,-1)},\; \mathrm{H}_{2g-1,\,2g-1}^{\halpha(-1,0),\, \halpha(0,-1)},\;\mathrm{H}_{2g-1,\,2g-1}^{\halpha(0,-1),\,\halpha(-1,0)} 
\right\} \subset \mathrm{End}_{\C}( W_{g}( \halpha )^{\dagger} )
\]
is contained in $\mathrm{res}_{\halpha}^{\dagger}(M^{\diamond})$. Consider the distinguished direct sum decomposition
\[
W_{g}^{\triangledown}(\halpha) = W_{g}^{-}(\halpha) \oplus W_{g}^{+}(\halpha)
\] 
, where we set
\begin{align*}
W_{g}^{-}(\halpha) &:= \langle \btw_{2g-1}^{\halpha\,(-1,0) },\; \btw_{2g-1}^{\halpha\,(-1,-1)} \rangle_{\C}\,,
\quad
W_{g}^{+}(\halpha) := \langle \btw_{2g-1}^{\halpha\,(0,-1) },\; \btw_{2g-1}^{\halpha\,(0,0)} \rangle_{\C} \,.
\end{align*}
\begin{slem}\label{sublem1 1st step} $\mathrm{End}_{\C}(W_{g}^{\pm}(\halpha)) \subset \mathrm{Image}(\mathrm{res}_{\halpha}^{\dagger})$
\end{slem}
\begin{proof} Set $X:=\mathrm{E}_{2g-1}^{\halpha(-1,0)},\; X^{\prime} := \mathrm{E}_{2g-1}^{\halpha(0,-1)} \in \mathrm{End}_{\C}(W_{g}(\halpha)^{\dagger})$. Further, set $B:= \mathrm{res}_{\halpha}^{\dagger}(\pi(B_g)),\; T:= \mathrm{res}_{\halpha}^{\dagger}(\pi(T_g))$
as in the proof of Lemma \ref{lem1 1st step}.
\\
\textbullet\, Set $Y:= B- X^{\prime}B,\;  Z:= XT$. Then it is readily seen that $X,\,Y,\,Z$ belong to the subalgebra $\mathrm{End}_{\C}(W_{g}^{-}(\halpha)) \subset \mathrm{End}_{\C}(W_{g}(\halpha)^{\dagger})$ and that their representation matrices are
\[
\left(
\begin{array}{cc}
1 & 0 \\ 0 & 0
\end{array}
\right)
,\;
\left(
\begin{array}{cc}
0 & 0 \\ 1 & 1
\end{array}
\right)
,\;
\left(
\begin{array}{cc}
1 & 1 \\ 0 & 0
\end{array}
\right)
\]
, which shows that $X,\,Y,\,Z,\, YX$ spans the $4$-dimensional space
 $\mathrm{End}_{\C}(W_{g}^{-}(\halpha))$.
\\
\textbullet\, Set $Y^{\prime}:= X^{\prime}B,\; Z^{\prime} := T - XT$. Then it is readily seen that $X^{\prime},\,Y^{\prime},\,Z^{\prime}$ belong to the subalgebra $\mathrm{End}_{\C}(W_{g}^{+}(\halpha)) \subset \mathrm{End}_{\C}(W_{g}(\halpha)^{\dagger})$ and that their representation matrices are
\[
\left(
\begin{array}{cc}
1 & 0 \\ 0 & 0
\end{array}
\right)
,\;
\left(
\begin{array}{cc}
1 & 1 \\ 0 & 0
\end{array}
\right)
,\;
\left(
\begin{array}{cc}
0 & 0 \\ 1 & 1
\end{array}
\right)
\]
, which shows that $X^{\prime},\,Y^{\prime},\,Z^{\prime},\, Z^{\prime}X^{\prime}$ spans the $4$-dimensional space $\mathrm{End}_{\C}(W_{g}^{+}(\halpha))$. 
\end{proof}
\begin{defi} Set 
\[
W_g^{\mp} := \bigoplus_{\halpha \in \{0,-1\}^{g-2}} W_g^{\mp}(\halpha).
\]
Notice that $ W_g^{-} \oplus W_g^{+} = W_g$. 
\end{defi}
\begin{cor}\label{cor2 1st step} $\mathcal{W}$ includes the distinguished projection $P^{\mp} : \lgo \twoheadrightarrow W_g^{\mp}$.
\end{cor}
\begin{proof} Denote by $\mathcal{T}$ the ideal of $\mathcal{W}$ generated by $\{\pi(B_g),\, \pi(T_g) \}$. The proof of SubLemma \ref{sublem1 1st step} shows that 
\[  
\mathrm{End}_{\C}(W_{g}^{-}(\halpha)) \oplus \mathrm{End}_{\C}(W_{g}^{+}(\halpha)) \subset  \mathrm{res}_{\halpha}^{\dagger}( \mathcal{T} ) \subset \mathrm{End}_{\C}(W_{g}^{\dagger}(\halpha)).
\]
It follows that there exists the unique $P^{\mp} \in \mathcal{T}$ such that 
$\mathrm{res}_{\halpha}^{\dagger}(P^{\mp})$ is the distinguished projection $W_{g}^{\dagger}(\halpha) \twoheadrightarrow W_{g}^{\mp}(\halpha)$. This in turn implies that 
$P^{\mp}|_{W_{g}^{\dagger}}$ is the distinguished projection $W_{g}^{\dagger} \twoheadrightarrow W_{g}^{\mp}$ since the argument so far is independent of the choice of the multi-index $\halpha$. On the other hand, $W_{g}^{\dagger\, \perp}$ vanishes under the $\mathcal{T}$-action since it is preserved under the  $\mathcal{W}$-action and since it vanishes under both $\pi(B_g)$- and $\pi(T_g)$-action. It follows that $P^{\mp}(W_{g}^{\dagger\,\perp}) =\{0\}$. Thus we are done.
\end{proof} 
Now we will return to the proof of Lemma \ref{lem2 1st step}. Thanks to  Corollary \ref{cor2 1st step}, $\mathcal{W}$ includes the distinguished projection \[
P_g = P^{-} + P^{+} : \lgo \twoheadrightarrow W_{g}^{-} \oplus W_{g}^{+}=W_{g}. 
\]
Set $\mathcal{U} : = P_g \circ \mathcal{W} \circ P_g$. To complete the proof, it is sufficient to show that $\mathrm{res}_{\halpha}^{\dagger}(\mathcal{U}) = \mathrm{End}_{\C}(W_g^{\triangledown}(\halpha))$ for some ( therefore all ) $\halpha \in \{0,-1\}^{g-2}$. This is equivalent to show that $W_g^{\triangledown}(\halpha)$ is an irreducible $\mathcal{U}$-module. Assuming to the contrary that we had a nonzero minimal $\mathcal{U}$-submodule $M \neq W_g^{\triangledown}(\halpha)$. SubLemma \ref{sublem1 1st step} states that $\mathrm{res}_{\halpha}^{\dagger}(\mathcal{U})$ contains the subalgebra $\mathcal{V}_{\halpha} :=\mathrm{End}_{\C}(W_g^{-}(\halpha)) \oplus \mathrm{End}_{\C}(W_g^{+}(\halpha))$ and that, in the direct sum decomposition
\[
W_g^{\triangledown}(\halpha) = W_g^{-}(\halpha) \oplus W_g^{+}(\halpha)
\] 
, the two components on the R.H.S. are irreducible $\mathcal{V}_{\halpha}$-modules, which are not isomorphic to each other. Therefore, regarded as a $\mathcal{V}_{\halpha}$-module, $M$ must be either $W_g^{-}(\halpha)$ or $W_g^{+}(\halpha)$. But $\mathrm{res}_{\halpha}^{\dagger}(\mathcal{U})$ includes the distinguished endomorphisms $\mathrm{H}_{2g-1,\,2g-1}^{(-1,0),\, (0,-1)},\; \mathrm{H}_{2g-1,\,2g-1}^{(0,-1),\, (-1,0)} \in \mathrm{End}_{\C}(W_g^{\dagger}(\halpha))$, which implies that, if $M$ contains one of the two, say, $W_g^{-}(\halpha)$, then it must contain the other, say, $W_g^{+}(\halpha)$, and vise verse. This leads to a contradiction. Thus we are done.
\end{proof}  
Now we will complete the proof of Proposition \ref{prop 1st step}. Notice that the argument so far is independent of the multi-index $\halpha$. Taking this into account, the subalgebra $\mathcal{U} \subset \cro$ that appeared in the last part of the proof of Lemma \ref{lem2 1st step} is nothing but what we seek for. Thus we are done.

\rightline{$\Box$}

\newpage
\t
{\bf 2nd Step.)} 
\begin{defi} For $1\leq h \leq g$, 
let $\mathcal{R}_{\mathrm{odd},\,h} \subset \cro$ be the subalgebra generated by
\[
\{ \pi(A_l) \ \mid h\leq l \leq g-1 \} \cup \{\pi(B_g),\, \pi(T_g) \}.
\]
\end{defi}
\begin{defi} For any multi-index $\hat{\alpha} \in \{0,-1\}^{h-1}$, define the distinguished subspace
\[
W_{\geq h}(\hat{\alpha}) :=
\left\langle \btw_{2l-1}^{\hat{\alpha}\hat{\gamma}} \mid h\leq l \leq g,\, \hat{\gamma} \in \{0,-1\}^{g-h+1} \right\rangle_{\C} \subset \lgo
\]
, which turns out to be a $\mathcal{R}_{\mathrm{odd},h}$-submodule. Note that, if $h=g$, then we denote it by $W_{g}(\hat{\alpha})$. 
\end{defi}
Notice that $W_{<h},\, W_{\geq h}$ and each $W_{\geq h}(\hat{\alpha})$ are $\mathcal{R}_{\mathrm{odd},h}$-submodules and that we have the direct sum decomposition
\[
W_{\geq h} = \bigoplus_{\hat{\alpha} \in \{0,-1\}^{h-1} } W_{\geq h}(\hat{\alpha}) 
\]
\begin{defi} Denote by $P_{\geq h}$ the distinguished projection $P_{\geq h} : \lgo =W_{\geq h} \oplus W_{<h} \twoheadrightarrow W_{\geq h}$.
\end{defi}
\begin{rem} Obviously we have the decreasing filtration by subalgebra as follows;
\[
\cro = \mathcal{R}_{\mathrm{odd},1}\supset \mathcal{R}_{\mathrm{odd},2} \supset \dots \supset \mathcal{R}_{\mathrm{odd},g} \supset \{0\} .
\]
For any $\hat{\alpha}=(a_1,\dots,a_{g-1}) \in \{0,-1\}^{g-1}$, we have the decreasing sequence
\[
\lgo = W_{\geq 1}(\emptyset) \supset W_{\geq 2}(a_1) \supset W_{\geq 3}(a_1,a_2) \dots \supset 
W_{g}(a_1,\dots,a_{g-1}) \supset \{0\}.
\]
, which is compatible with the filtration above in the sense that each $W_{\geq h}(a_1,\dots,a_{h-1})$ is a $\mathcal{R}_{\mathrm{odd}, h}$-submodule for $1\leq h \leq g$. 
\end{rem}
\begin{rem}\label{categorical isom} Notice that the isomorphism class of the $\mathcal{R}_{\mathrm{odd},\, h}$-module $W_h(\hat{\alpha})$ is independent of the choice of $\hat{\alpha}$. In fact, for any two multi-indices $\hat{\alpha},\hat{\beta}$, we have the canonical $\C$-linear isomorphism
\[
\tau^{\geq h}_{\hbeta \halpha}: W_{\geq h}(\hat{\alpha}) \overset{\cong}{\longrightarrow} W_{\geq h}(\hat{\beta}) \;:\; \btw_{2l-1}^{\hat{\alpha}\hat{\gamma}} \mapsto \btw_{2l-1}^{\hat{\beta}\hat{\gamma}} \quad (h\leq l \leq g,\; \hat{\gamma} \in \{0,-1\}^{g-h+1})
\]
, which turns out to be a $\mathcal{R}_{\mathrm{odd}, h}$-module homomorphism. 
\end{rem}
\begin{defi} Set 
\[
\mathrm{res}^{\geq h}_{\halpha} : \mathcal{R}_{\mathrm{odd},h} \to \mathrm{End}_{\C}(W_{\geq h}(\hat{\alpha}))
\]
to be the map induced by the $\mathcal{R}_{\mathrm{odd},h}$-action on $W_{\geq h}(\hat{\alpha})$.
\end{defi}

%
We will propose the following assertion A(h) for $1\leq h \leq g$, which will be proven by the downward inductive argument w.r.t. $h$.
\begin{assertion*}{\bf A(h).} For any $\hat{\alpha} \in \{0,-1\}^{h-1}$, $W_{\geq h}(\hat{\alpha})$ is an irreducible $\mathcal{R}_{\mathrm{odd}, h}$-submodule. Further, $\mathcal{R}_{\mathrm{odd},h}$ contains the distinguished projection $P_{\geq h} : \lgo \twoheadrightarrow W_{\geq h}$.
\end{assertion*}
\begin{rem} The assertion A(h) implies that $P_{\geq h}$ is a central idempotent of $\mathcal{R}_{\mathrm{odd}, h}$ and that $W_{\geq h} = P_{\geq h}(\lgo)$ is an "isotypic" component of the $\mathcal{R}_{\mathrm{odd}, h}$-module $\lgo$. In fact, we see that $\mathcal{R}_{\mathrm{odd}, h}$ decomposes into the direct sum of ideals as 
\[
\mathcal{R}_{\mathrm{odd}, h} = (1-P_{\geq h}) \!\circ\! \mathcal{R}_{\mathrm{odd}, h} \!\circ\! (1-P_{\geq h}) \;\oplus\; P_{\geq h} \!\circ\! \mathcal{R}_{\mathrm{odd}, h} \!\circ\! P_{\geq h} 
\]
since $\lgo = W_{<h} \oplus W_{\geq h}$ is faithful with the direct summands being $\mathcal{R}_{\mathrm{odd}, h}$-submodules.   
\end{rem}
\t
{\bf Proof of A(g)} \, This has already been proven by Proposition \ref{prop 1st step}.

\rightline{$\Box$}
\t
{\bf Proof of A(h).} \, Suppose $h<g$. Assume that A(h+1), \dots, A(g) be true. We will prove A(h). Two remarks here are in order;
\begin{rem}\label{A(h) rem1} It follows from the induction assumption that the distinguished projection $P_{h+1} \!:\! \lgo \twoheadrightarrow W_{h+1}$ is included in $\mathcal{R}_{\mathrm{odd}, h+1}$. In fact, if $h<g-1$, then $P_{h+1}=P_{\geq h+1} - P_{\geq h+2}$, where the two terms of the R.H.S. are included in $\mathcal{R}_{\mathrm{odd}, h+1}$ by A(h+1) and by A(h+2),  respectively. If $h=g-1$, then A(g) implies that $P_{g} \subset \mathcal{R}_{\mathrm{odd}, g}$.
\end{rem}
\begin{rem}\label{A(h) rem2} Notice that $P_{\geq h+1} \circ \mathcal{R}_{\mathrm{odd}, h+1} \circ P_{\geq h+1}$ is a subalgebra of $\mathcal{R}_{\mathrm{odd}, h+1}$. The 1st statement of A(h+1) implies that, for any $\hbeta \in \{0,-1\}^{h}$, the subspace $W_{\geq h+1}(\hbeta)$ is an irreducible and faithful $P_{\geq h+1} \circ \mathcal{R}_{\mathrm{odd}, h+1} \circ P_{\geq h+1}$-module. Thus the Burnside theorem together with Remark \ref{categorical isom} implies that, for any distinguished subspace $U \subset W_{\geq h+1}(\hbeta_{0})$, the distinguished projection of $\lgo$ to the distinguished subspace
\[
\bigoplus_{\hbeta \in \{0,-1\}^{h} } \tau^{\geq h}_{\hbeta ,\hbeta_{0} }(U) \subset \lgo
\]
is included in $P_{\geq h+1} \circ \mathcal{R}_{\mathrm{odd}, h+1} \circ P_{\geq h+1} \subset \mathcal{R}_{\mathrm{odd}, h}$.   
\end{rem}
%
We will define a useful subalgebra $\mathcal{A}_h \subset \mathcal{R}_{\mathrm{odd}, h}$.
\begin{defi} Let $P_{h+1}^{(-1)} : \lgo \twoheadrightarrow U_{h+1} \subset \lgo$ be the distinguished projection where we set
\[
U_{h+1}:= \langle \btw_{2h+1}^{ \hbeta(-1)\hat{\gamma} } \mid \hbeta \in \{0,-1\}^{h} ,\; \hat{\gamma} \in \{0,-1\}^{g-h-1} \rangle_{\C}.
\]
\end{defi}
\begin{lem}\label{lem1 2nd step} It holds that $P_{h+1}^{(-1)} \in P_{\geq h+1} \circ \mathcal{R}_{\mathrm{odd}, h+1} \circ P_{\geq h+1}$.
\end{lem}
\begin{proof} For any $\hbeta \in \{0,-1\}^{h}$, set the distinguished subspace
\[
U_{h+1}(\hbeta) := \langle \btw_{2h-1}^{\hbeta(-1)\hat{\gamma} } \mid  \gamma \in \{0,-1\}^{g-h-1} \rangle_{\C} \; \subset \; W_{\geq h+1}(\hbeta).
\]
Notice that
\[
U_{h+1} = \bigoplus_{\hbeta  \in \{0,-1\}^{h}} U_{h+1}(\hbeta) = \bigoplus_{\hbeta  \in \{0,-1\}^{h}} \tau^{\geq h}_{\beta,\beta_0}(U_{h+1}(\hbeta_0)) .
\]
Applying the argument in Remark \ref{A(h) rem2} to $U_{h+1}(\hbeta) \subset  W_{\geq h+1}(\hbeta)$, the result follows.
\end{proof}
\begin{defi} Let $\mathcal{A}_h \subset \mathcal{R}_{\mathrm{odd}, h}$ be the subalgebra generated by $\{\pi(A_h),\, P_{h+1}^{(-1)}\}$.
\end{defi}
When regarded as a $\mathcal{A}_h$-module, $\lgo$ decomposes into the direct sum 
\[
W_{<h} \oplus W_{>h+1} \oplus \bigoplus_{\hat{\alpha} \in \{0,-1\}^{h-1}} \bigoplus_{\hat{\gamma} \in \{0,-1\}^{g-h-1}} \left\{ W_{h+1}^{\diamond}(\hat{\alpha}, \hat{\gamma}) \oplus W_{h+1}^{\star}(\hat{\alpha}, \hat{\gamma}) \oplus W_{h+1}^{\ast}(\hat{\alpha}, \hat{\gamma}) \right\}  
\]
, where we set
\begin{align*} 
W_{h+1}^{\diamond}(\hat{\alpha}, \hat{\gamma}) &:=
\left\langle \btw_{2h-1}^{\hat{\alpha}\,(-1,0)\,\hat{\gamma} },\; \btw_{2h-1}^{\hat{\alpha}\,(0,-1)\,\hat{\gamma} },\; \btw_{2h+1}^{\hat{\alpha}\,(-1,0)\,\hat{\gamma} },\; \btw_{2h+1}^{\hat{\alpha}\,(0,-1)\,\hat{\gamma} } \right\rangle_{\C},
\\
W_{h+1}^{\star}(\hat{\alpha}, \hat{\gamma}) &:=
\left\langle \btw_{2h-1}^{\hat{\alpha}\,(-1,-1)\,\hat{\gamma} },\; \btw_{2h-1}^{\hat{\alpha} \,(0,0)\,\hat{\gamma} } \right\rangle_{\C},
\\
W_{h+1}^{\ast}(\hat{\alpha}, \hat{\gamma}) &:=
\left\langle \btw_{2h+1}^{\hat{\alpha}\,(-1,-1)\,\gamma},\; \btw_{2h+1}^{\hat{\alpha}\,(0,0)\,\hat{\gamma} } \right\rangle_{\C}.
\\
\end{align*}

\vspace{-1cm}

\t
Further, set
\begin{align*}
W_{h+1}^{\diamond}(\hat{\alpha}) & :=\bigoplus_{ \hat{\gamma} \in \{ 0,-1\}^{g-h-1} } W_{h+1}^{\diamond}(\hat{\alpha}, \hat{\gamma})
\\
W_{h+1}^{\star}(\hat{\alpha}) & :=\bigoplus_{ \hat{\gamma} \in \{ 0,-1\}^{g-h-1} } W_{h+1}^{\ast}( \hat{\alpha}, \hat{\gamma} ) .
\end{align*}
Let $\mathrm{res}_{\hat{\alpha}, \hat{\gamma}}^{\diamond}: \mathcal{A}_h \to \mathrm{End}_{\C}( W_{h+1}^{\diamond}(\hat{\alpha}, \hat{\gamma}) )$ be the map induced by the action of $\mathcal{A}_h$ on $W_{h+1}^{\diamond}(\hat{\alpha}, \hat{\gamma})$. Notice that the representation matrices of $\mathrm{res}_{\hat{\alpha}, \hat{\gamma} }^{\diamond}$ w.\ r.\ t.\ the distinguished basis is independent of the the choice of the pair$(\hat{\alpha}, \hat{\gamma})$. 
\begin{lem}\label{lem2 2nd step} $\mathrm{res}_{\hat{\alpha}, \hat{\gamma}}^{\diamond}$ is surjective.
\end{lem}
\begin{proof} We can prove the result in exactly the same manner as Lemma \ref{lem1 1st step}.  In fact, we should replace $A$ and $T$ in the proof of that lemma by $\mathrm{res}_{\hat{\alpha},\hat{\gamma}}^{\diamond}(\pi(A_h))$ and $\mathrm{res}_{\hat{\alpha},\hat{\gamma}}^{\diamond}(P_{h+1}^{(-1)})$ here, respectively.
\end{proof}
\begin{defi} Let $\mathcal{A}^{\prime}_h \subset \mathcal{A}_h$ be the ideal generated by $\{ P_{h+1}^{({-1})} \circ \pi(A_h) \}$.
\end{defi} 
\begin{lem}\label{lem3 2nd step} $W_{<h} \oplus W_{>h+1}$ and each $W_{h+1}^{\star}(\hat{\alpha},\hat{\gamma}) \oplus W_{h+1}^{\ast}(\hat{\alpha},\hat{\gamma})$ vanish under the $\mathcal{A}^{\prime}_h$-action. 
\end{lem}
\begin{proof}  Since $W_{<h} \oplus W_{>h+1}$ is a $\mathcal{A}_h$-submodule and since it vanishes under the $P_{h+1}^{(-1)}$-action, it vanishes under the $\mathcal{A}^{\prime}_h$-action. Similarly, since each $W_{h+1}^{\star}(\hat{\alpha},\hat{\gamma}) \oplus W_{h+1}^{\ast}(\hat{\alpha},\hat{\gamma})$ is a $\mathcal{A}_h$-submodule and since it vanishes under the $\pi(A_h)$-action, it does so under the $\mathcal{A}^{\prime}_h$-action. 
\end{proof}
%
%
\begin{lem}\label{lem diamond isom}
$\mathrm{res}_{\hat{\alpha},\hat{\gamma}}^{\diamond}|_{\mathcal{A}^{\prime}_h}$ is bijective.
\end{lem}
\begin{proof} Due to the result of Lemma \ref{lem3 2nd step} and the fact that $\lgo$ is a faithful $\mathcal{R}_{\mathrm{odd}, h}$-module, the injectivity of $\mathrm{res}_{\hat{\alpha},\hat{\gamma}}^{\diamond}|_{\mathcal{A}^{\prime}_h}$ follows. Thus it is enough to show the surjectivity.
Since $\mathrm{res}_{\hat{\alpha},\hat{\gamma}}^{\diamond}(P_{h+1}^{(-1)} \circ \pi(A_h))\neq 0$ and since $\mathrm{res}_{\hat{\alpha},\hat{\gamma}}^{\diamond}|_{\mathcal{A}_h}$ is surjective, we see that $\mathrm{res}_{\hat{\alpha},\hat{\gamma}}^{\diamond}(\mathcal{A}^{\prime}_h) \subset \mathrm{End}_{\C}( W_{h+1}^{\diamond}(\hat{\alpha},\hat{\gamma}) )$ is a non zero ideal. But the former must coincide with the latter since the latter is a simple algebra. Thus we are done. 
\end{proof}
%
%
%

Recall that, for each $\hbeta \in \{ 0,-1\}^{h}$, $W_{\geq h+1}(\hbeta)$ is an irreducible and faithful $P_{\geq h+1} \circ \mathcal{R}_{\mathrm{odd},h+1} \circ P_{\geq h+1}$-module. Further, any two of them are canonically isomorphic to each other as modules of this subalgebra. Thus $P_{\geq h+1} \circ \mathcal{R}_{\mathrm{odd},h+1} \circ P_{\geq h+1}$ includes  
the four operators $\phi_{h+1}^{i,j} \; ( i,j \in \{ 0,-1\})$ determined by
\begin{align*}
\phi_{h+1}^{i,j}(\btw_{2l-1}^{\hat{\mu}(m)\hhgamma}) &:= \delta_{l,h+1}\, \delta^{j,m}\, \btw_{2h+1}^{ \hat{\mu}(i)\hhgamma}
\end{align*}
, where $1\leq l \leq g,\; m \in \{ 0,-1\},\;  \hat{\mu} \in \{ 0,-1\}^{h},\;  \hat{\gamma} \in \{ 0,-1\}^{g-h-1}$. We remark that the set $\{ \phi_{h+1}^{i,j} \mid  i,j \in \{ 0,-1\} \}$ composes an $4$-dimensional subalgebra.
\begin{defi}
Let $\mathcal{B}_{h}^{\prime} \subset \mathcal{R}_{\mathrm{odd} ,h}$ be the subalgebra generated by 
\[
P_{\geq h+1} \circ \mathcal{A}_{h}^{\prime} \circ P_{\geq h+1} \;\cup\; \{ \phi_{h+1}^{i,j} \mid  i, j \in \{ 0,-1\} \}.
\]
\end{defi} 
%
%
\begin{defi} For any $\halpha \in \{ 0,-1\}^{h-1}$ and $\hhgamma \in \{ 0,-1\}^{g-h-1}$, set
\begin{align*} 
W_{h+1}^{\triangledown}(\halpha, \hhgamma) &:= \left\langle
\btw_{2h+1}^{\halpha\,(-1,0)\,\hhgamma},\; \btw_{2h+1}^{\halpha\,(0,-1)\,\hhgamma},\; \btw_{2h+1}^{\halpha\,(-1,-1)\,\hhgamma},\; \btw_{2h+1}^{\halpha\,(0,0)\,\hhgamma} \right\rangle_{\C}
\end{align*}
, which turns out to be a $\mathcal{B}_{h}^{\prime}$-submodule. Set
\begin{align*}
W_{h+1}^{\triangledown}(\halpha) & :=\bigoplus_{\hhgamma \in \{ 0,-1\}^{g-h-1} } W_{h+1}^{\triangledown}(\halpha, \hhgamma).
\end{align*}
Further, let
\[
\mathrm{res}_{\halpha, \hhgamma}^{\triangledown} : \mathcal{B}_{h}^{\prime} \to \mathrm{End}_{\C}( W_{h+1}^{\triangledown}(\halpha, \hhgamma))
\]
be the map induced by the $\mathcal{B}_{h}^{\prime}$-action on $W_{h+1}^{\triangledown}(\halpha, \hhgamma)$.
\end{defi}
%
\begin{lem}\label{lem irr triangle} $W_{h+1}^{\triangledown}(\halpha, \hhgamma)$ is an irreducible $\mathcal{B}_{h}^{\prime}$-module.
\end{lem}
\begin{proof} Set $\mathcal{A}_{h}^{\prime\prime} :=P_{\geq h+1} \circ \mathcal{A}_{h}^{\prime} \circ P_{\geq h+1}$. Lemma \ref{lem diamond isom} implies that it is the subalgebra of $\mathcal{B}_{h}^{\prime}$ spanned by the $4$ operators $\psi_{h+1}^{i,j} \; ( i,j \in \{ 0,-1\})$ determined as follows; 
\begin{align*}
\psi_{h+1}^{i,j}(\btw_{2l-1}^{\halpha(m,n)\hhgamma}) &:=
\begin{cases}
\delta_{l, h+1}\, \delta^{j,m}\, \btw_{2h+1}^{ \halpha(i,-i-1)\hhgamma} & \text { if } m+n=-1\\
 0 & \text { if } m+n \neq -1\\
\end{cases}
\end{align*}
, where $1\leq l \leq g,\; (m,n) \in \{ 0,-1\}^2, \; \halpha \in \{ 0,-1\}^{h-1},\;  \hhgamma \in \{ 0,-1\}^{g-h-1}$. It follows that $W_{h+1}^{\triangledown}(\alpha, \gamma)$ decomposes into the direct sum of $\mathcal{A}_{h}^{\prime\prime}$-submodules as
\[
W_{h+1}^{\sharp}(\halpha, \hhgamma) \oplus W_{h+1}^{\ast}(\halpha, \hhgamma)
\]
, where the 1st summand and the 2nd one are spanned respectively by
\[
\{ \btw_{2h+1}^{\halpha\,(-1,0)\,\hhgamma},\; \btw_{2h+1}^{\halpha\,(0,-1)\,\hhgamma} \}, \quad 
\{ \btw_{2h+1}^{\halpha\,(0,0)\,\hhgamma},\; \btw_{2h+1}^{\halpha\,(-1,-1)\,\hhgamma} \}.
\]
Notice that the 1st summand is an irreducible $\mathcal{A}_{h}^{\prime\prime}$-module and that the 2nd one vanishes under the $\mathcal{A}_{h}^{\prime\prime}$-action. Let $1_{\mathcal{A}_{h}^{\prime\prime} }$ be the (multiplicative) identity of the algebra $\mathcal{A}_{h}^{\prime\prime}$. Set
\begin{align*}
f &:=1_{\mathcal{A}_{h}^{\prime\prime} } \circ (\phi^{-1,0}_{h+1} + \phi^{0,-1}_{h+1})
\\
g &:= (\phi^{-1,0}_{h+1} + \phi^{0,-1}_{h+1}) \circ 1_{\mathcal{A}_{h}^{\prime\prime} }
\end{align*}
Then the actions of $f$ and $g$ on $W_{h+1}^{\triangledown}(\halpha, \hhgamma)$ are described as follows;\\
\textbullet\, $f(W_{h+1}^{\sharp}(\halpha, \hhgamma)) = \{0\} = g(W_{h+1}^{\ast}(\halpha, \hhgamma) )$.
\\
\textbullet\, $f|_{W_{h+1}^{\ast}(\halpha, \hhgamma)} : W_{h+1}^{\ast}(\halpha, \hhgamma) \to W_{h+1}^{\sharp}(\halpha, \hhgamma)$ is bijective.
\\
\textbullet\, $g|_{W_{h+1}^{\sharp}(\halpha, \hhgamma)} : W_{h+1}^{\sharp}(\halpha, \hhgamma) \to W_{h+1}^{\ast}(\halpha, \hhgamma)$ is bijective.
\\
Applying Lemma \ref{lem2 useful}, we see that the subalgebra $\mathcal{B}_{h}^{\prime\prime} \subset \mathcal{B}_{h}^{\prime}$ generated by 
\[
\{ \psi_{h+1}^{i,j} \mid i,j \in \{ 0,-1\} \} \cup \{ f,g \}
\]
generates the endomorphism algebra $\mathrm{End}_{\C}(W_{h+1}^{\triangledown}(\halpha, \hhgamma))$ via the map $\mathrm{res}_{\halpha, \hhgamma}^{\triangledown}$. Thus we are done.
\end{proof}
\begin{defi} Denote by $\mathcal{B}_h$ the subalgebra $P_{\geq h+1} \circ \mathcal{R}_{\mathrm{odd},h} \circ P_{\geq h+1} \subset \mathcal{R}_{\mathrm{odd} \,h}$. Notice that $\mathcal{B}_h \supset \mathcal{B}^{\prime}_h$.
\end{defi}
\begin{defi} For any $\hat{\alpha} \in \{ 0,-1\}^{h-1}$, define the subspace
\[
Y_{h+1}(\halpha) :=W_{\geq h+1}(\hat{\alpha}(0)) \oplus W_{\geq h+1}(\hat{\alpha}(-1)) 
\]
, which turns out to be a $\mathcal{B}_h$-submodule.
\end{defi}
\begin{prop}
$Y_{h+1}(\halpha)$ is an irreducible $\mathcal{B}_h$-module. 
\end{prop}
\begin{proof} Notice that the $2$ direct summands $W_{\geq h+1}(\hat{\alpha}(0))$ and $W_{\geq h+1}(\hat{\alpha}(-1))$ are irreducible $P_{\geq h+1} \circ \mathcal{R}_{\mathrm{odd},h+1} \circ P_{\geq h+1}$-modules by A(h+1). Assume to the contrary that $Y_{h+1}(\halpha)$ were not irreducible as a $\mathcal{B}_h$-module. Let $M\subset Y_{h+1}(\halpha)$ be a minimal nonzero submodule. Then $M \neq Y_{h+1}(\halpha)$. Since $M$ is a $P_{\geq h+1} \circ \mathcal{R}_{\mathrm{odd},h+1} \circ P_{\geq h+1}$-submodule, $M$ must coincide with one of the direct summands above. On the other hand, $\mathcal{B}_h$ includes the operators $\psi_{h+1}^{i,j} \; ( i,j \in \{ 0,-1\})$ in the proof of Lemma \ref{lem irr triangle}. It is readily seen that, for any $a \in \{0,-1\}$, 
\[
\{0\} \neq \psi_{h+1}^{-a-1,a}\left( W_{\geq h+1}(\halpha(a)) \right) \subset W_{\geq h+1}( \halpha(-a-1))
\]
, which implies that neither summand can be $\mathcal{B}_h$-submodule. This leads to a contradiction. Thus we are done.
\end{proof}
\begin{defi} Let $\mathcal{C}_h \subset \mathcal{R}_{\mathrm{odd} ,h}$ be the subalgebra generated by $\mathcal{A}^{\prime}_h \cup \mathcal{B}_h$. 
\end{defi}
\begin{defi} For any $\halpha \in \{ 0,-1\}^{h-1}$, set
\[
W_{\geq h}^{\prime}(\halpha) := W_{h+1}^{\diamond}(\halpha) + Y_{n+1}(\halpha)
\]
, which turns out to be a faithful $\mathcal{C}_h$-submodule.
\end{defi} 
\begin{rem} Notice that
\[
W_{\geq h}(\halpha) = W_{h+1}^{\star}(\halpha) \oplus W_{\geq h}^{\prime}( \halpha).
\]
\end{rem}
\begin{prop} $W_{\geq h}^{\prime}( \hat{\alpha})$ is an irreducible $\mathcal{C}_{h}$-submodule. Further, $\mathcal{C}_{h}$ includes the distinguished projection $P^{\prime}_{\geq h}$ of $\lgo$ to the subspace
\[
 W_{\geq h}^{\prime} := \bigoplus_{\halpha \in \{0,-1\}^{h-1} } W_{\geq h}^{\prime}( \hat{\alpha}) .
\]
\end{prop}
\begin{proof} Set
\[
M^{\flat}_h(\halpha) := \left\langle \btw_{2h-1}^{\halpha\,(-1,0)\,\hhgamma},\; \btw_{2h-1}^{\halpha\,(0,-1)\,\hhgamma} \;\middle|\; \hhgamma \in \{0,-1 \}^{g-h-1}  \right\rangle_{\C}.
\]
Then we have the direct sum decomposition
\[
W_{\geq h}^{\prime}(\halpha) =  M^{\flat}_h(\halpha) \oplus Y_{ h+1}(\halpha) .
\]
Lemma \ref{lem diamond isom} implies that $\mathcal{A}^{\prime}_h$ contains the operators $\chi_{h+1}^{R},\; \chi_{h+1}^{L}$ determined by
\begin{align*}
\chi_{h+1}^{R}(\btw_{2l-1}^{\halpha(m,n)\hhgamma }) &= \delta_{l,h}\,  \delta_{m+n,-1}\,  \btw_{2h+1}^{\halpha(m,n)\hhgamma} ,
\\
\chi_{h+1}^{L}(\btw_{2l-1}^{\halpha(m,n)\hhgamma}) &= \delta_{l,h+1}\, \delta_{m+n,-1}\, \btw_{2h-1}^{\halpha(m,n)\hhgamma}
\end{align*}
, where $1\leq l \leq g,\; \halpha \in \{ 0,-1\}^{h-1},\; (m,n) \in \{ 0,-1\}^{2},\; \hhgamma \in \{ 0,-1\}^{g-h-1}$.
Then it is readily seen that
\begin{align*}
& \mathrm{Ker}(\chi_{h+1}^{R}) \supset W_{<h} \oplus W_{\geq h+1},\quad \chi_{h+1}^{R}(W_{h}) \subset W_{h+1},
\\
& \mathrm{Ker}(\chi_{h+1}^{L}) \supset W_{\leq h} \oplus W_{> h+1},\quad \chi_{h+1}^{R}(W_{h+1}) \subset W_{h}.
\end{align*}
Thus we are in the following situation;\\
\textbullet\, $Y_{h+1}(\halpha)$ is an irreducible and faithful $\mathcal{B}_h$-module.
\\ 
\textbullet\, $M^{\flat}(\halpha)$ vanishes under the $\mathcal{B}_h$-action.\\
\textbullet\, $\chi_{h+1}^{R}|_{M^{\flat}(\halpha)} : M_h^{\flat}(\halpha) \to Y_{h+1}(\halpha)$ is injective.
\\
The 1st property implies that there exists a unique element $\mathrm{b}_h \in \mathcal{B}_h$ such that the induced $\mathrm{b}_h$-action on $Y_{h+1}(\halpha)$ coincides with the distinguished projection $Y_{h+1}(\halpha) \twoheadrightarrow \chi_{h+1}^{R}(M^{\flat}(\halpha) )$. With this understood, we see that 
\\
\textbullet\, $\mathrm{b}_h \!\circ\! \chi_{h+1}^{R}|_{M^{\flat}(\halpha)} : M_h^{\flat}(\halpha) \to Y_{h+1}(\halpha)$ is injective.\\
\textbullet\, $\chi_{h+1}^{L} \!\circ \mathrm{b}_h\,(\,Y_{h+1}(\halpha)\,) =M_h^{\flat}(\halpha)$.
\\
\textbullet\, The distinguished subspace $W_{\geq h+1}^{\prime\,\perp} \subset \lgo$ complementary to $W_{\geq h+1}^{\prime}$ vanishes under the action of $\mathcal{B}_h \cup \{ \mathrm{b}_h \!\circ\! \chi_{h+1}^{R},\; \chi_{h+1}^{L} \!\circ \mathrm{b}_h \}$.
\\
Thus Lemma \ref{lem2 useful} together with all the six properties above except the 3rd implies that the subalgebra of $\mathcal{C}_h$ generated by $\mathcal{B}_h \cup \{ \mathrm{b}_h \!\circ\! \chi_{h+1}^{R},\; \chi_{h+1}^{L} \!\circ \mathrm{b}_h \}   \subset \mathcal{C}_h$ contains some subalgebra $\mathcal{C}_h^{\prime}$ such that\\ 
\textbullet\, $W_{\geq h}^{\prime}(\halpha)$ is an irreducible and faithful $\mathcal{C}_h^{\prime}$-module. \\
\textbullet\, $W_{\geq h+1}^{\prime\,\perp}$ vanishes under the $\mathcal{C}_h^{\prime}$-action.
\\
Notice that $\mathcal{C}_h=\mathcal{C}_h^{\prime}$ since $W_{\geq h}^{\prime}(\halpha)$ is a faithful $\mathcal{C}_h$ module and because of the next to the last property above. Thus the identity element of $\mathcal{C}_h$ is nothing but the distinguished projection $P^{\prime}_{\geq h}: \lgo \twoheadrightarrow W_{\geq h}^{\prime}$, which is what we seek for. Thus we are done.
\end{proof}
%
%
Now we will finish the proof of the assertion A(h). Recall that, for any $\halpha \in \{0,-1\}^{h}$,
\[
W_{\geq h}(\halpha) = W_{h+1}^{\star}(\halpha) \oplus W_{\geq h}^{\prime}( \halpha).
\]
With this in mind, define the elements $S_{h}^{(i)} \; (h+2 \leq i \leq g)$ of $\mathcal{R}_{\mathrm{odd} ,h}$ by
\begin{align*}
S_{h}^{(i)} &:= (1-P^{\prime}_{\geq h}) \circ \pi(A_{i}) \circ (1-P^{\prime}_{\geq h}).
\end{align*}
 (Of course, this definition is in vein when $h=g-1$. But still we put it here for the uniform treatment.) Notice that $W^{\prime}_{\geq h}$ vanishes and $W_{<h}$ and each $W_{h+1}^{\star}(\halpha)$ are preserved under the $S_{h}^{(i)}$-action.
\\
We will introduce some grading on $W_{h+1}^{\star}(\halpha)$. For this purpose, consider the direct sum decomposition 
\[
W_{h+1}^{\star}(\halpha) = W_{h+1,-1}^{\star}(\halpha) \oplus W_{h+1,\,0}^{\star}(\halpha)
\]
, where
\begin{align*}
W_{h+1,a}^{\star}(\halpha) &: =\left\langle \btw_{2h-1}^{\halpha(a,a)\hhgamma} \,\middle|\, \hhgamma \in \{0,-1 \}^{g-h-1} \right\rangle_{\C} \quad (a \in \{0,-1\}) .
\end{align*}
The grading of these subspaces are determined by setting their degree $f$-part as
\begin{align*}
W_{h+1,a}^{\star}(\halpha; f)
&: =\left\langle \btw_{2h-1}^{\halpha(a,a)\hhgamma} \in W_{h+1,a}^{\star}(\halpha)  \,\middle|\, \gamma_{j} = a \; (1\leq j \leq f),\; \gamma_{f+1} =-a-1 \right\rangle_{\C} \quad (a \in \{0,-1\}) 
\end{align*}
, where we understand the multi-index $\hhgamma = (\gamma_1, \gamma_2,\dots,\gamma_{g-h-1}$). The degree $f$ takes values in $\{ 0,1,\dots, g-h-1\}$. For example, if $f=0$ and $a={-1}$, we see that
\[
W_{h+1,-1}^{\star}(\halpha; 0) = \left\langle \btw_{2h-1}^{\halpha(-1,-1,0)\hhgamma^{\prime}} \,\middle|\, \hhgamma^{\prime} \in \{0,-1\}^{g-h-2} \right\rangle_{\C}.
\]
If $f=g-h-1$ and $a={-1}$, we see that
\[
W_{h+1,-1}^{\star}(\halpha; g-h-1) = \left\langle \btw_{2h-1}^{\halpha(-1,-1,-1,\dots,-1)} \right\rangle_{\C}.
\]  
The induced grading on $W_{h+1}^{\star}(\halpha)$ and $W_{h+1}^{\star}$ are determined as follows;
\begin{align*}
W_{h+1}^{\star}(\halpha; f) &:= W_{h+1,-1}^{\star}(\halpha;f) \oplus W_{h+1,0}^{\star}(\halpha;f),
\quad
W_{h+1}^{\star}(f) := \bigoplus_{\halpha \in \{0,-1\}^{h-1} } W_{h+1}^{\star}(\halpha; f).
\end{align*}
\begin{defi} Define some useful elements in $\mathcal{R}_{\mathrm{odd} ,h}$ as below; 
\begin{align*}
T_{h;f} &:= 
\begin{cases}
S_{h}^{(h+2)} \circ S_{h}^{(h+3)} \circ \dots \circ S_{h}^{(h+f+1)} & \text{ if } f\geq 1,
\\
(1-P^{\prime}_{\geq h})  & \text{ if } f =0.
\end{cases}
\\
T_{h;f}^{*} &:= 
\begin{cases}
S_{h}^{(h+f+1)} \circ S_{h}^{(h+f)} \circ \dots \circ S_{h}^{(h+2)}  & \text{ if } f\geq 1,
\\
(1-P^{\prime}_{\geq h}) & \text{ if } f =0,
\end{cases}
\end{align*}
where $0 \leq f \leq g-h-1$.
\end{defi}
Notice that these operators preserve $W_{<h}$ and kill $W_{\geq h}^{\prime}$. 
\begin{lem}\label{lem6 2nd step} For $1 \leq f \leq g-h-1$ and for $0 \leq  f^{\prime} \leq g-h-1$, it holds that
\\  
\textbullet\, $T_{h;f} |_{ W_{h+1}^{\star}(f^{\prime}) } = 0$  if  $f \neq f^{\prime}$.
\\
\textbullet\, $T_{h;f} ( W_{h+1}^{\star}(\halpha;f) )  \subset  W_{h+1}^{\star}(\halpha;0)$
\\
\textbullet\, $T_{h;f} |_{ W_{h+1}^{\star}(f) }$ is injective.
\\
\textbullet\, $T_{h;f}^{*} |_{W_{h+1}^{\star}(f^{\prime}) } = 0$ if $f^{\prime} \neq 0$.
\\
\textbullet\, $T_{h;f}^{*} ( W_{h+1}^{\star}(\halpha;0) ) =  W_{h+1}^{\star}(\halpha;f)$.
\end{lem}
\begin{defi} For $0 \leq f \leq g-h-1$, define the elements $Q_{h;f},\,Q_{h:f}^{*} \in \mathcal{R}_{\mathrm{odd} ,h}$ as follows;
\begin{align*}
Q_{h;f} &:= (1-P_{\geq h+1} ) \circ \pi(A_{h+1}) \circ (1-P_{\geq h+1} ) \circ T_{h;f}, 
\\
Q_{h;f}^{*} &:= T_{h;f}^{*} \circ (1-P_{\geq h+1} ) \circ \pi(A_{h+1}) \circ (1-P_{\geq h+1} ) .
\end{align*}
\end{defi}
\begin{lem}\label{lem7 2nd step} For $0 \leq f,\, f^{\prime} \leq g-h-1$ and for any $\halpha \in \{ 0,-1\}^{h-1}$, it holds that \\
\textbullet\, $Q_{h;f}$ preserves $W_{<h}$ and kills $W_{\geq h+1}^{\prime}$.
\\  
\textbullet\, $Q_{h;f} |_{ W_{h+1}^{\star}(f^{\prime}) } = 0$ if $f \neq f^{\prime}$.
\\
\textbullet\, $Q_{h;f} ( W_{h+1}^{\star}(\halpha; f ))  \subset  W_{\geq h+1}^{\prime}(\halpha)$ such that  $Q_{h;f} |_{ W_{h+1}^{\star}(\halpha; f) }$ is injective.
\\
\textbullet\, $Q_{h;f}^{*}$ preserves $W_{<h}$ and kills $W_{h+1}^{\star}$.
\\
\textbullet\, $Q_{h;f}^{*} ( W_{\geq h+1}^{\prime}(\halpha)) =W_{h+1}^{\star}(\halpha;f)$.
\end{lem}
Lemma \ref{lem6 2nd step} and Lemma \ref{lem7 2nd step} are readily deduced from a moment inspection on those operators. Now we will complete the proof of the assertion A(h). Applying Lemma \ref{lem2 useful} to the result of Lemma \ref{lem7 2nd step}, the subalgebra $\mathcal{D}_h \subset \mathcal{R}_{\mathrm{odd}, h}$ generated by
\[
\mathcal{C}_h \,\;\cup\;\, \{\, Q_{h;f},\, Q_{h;f}^{*} \mid 0 \leq f \leq g-h-1 \,\}
\]
contains some subalgebra $\mathcal{D}_h^{\prime}$ such that \\
\textbullet\; each $W_{\geq h}(\halpha)$ is an irreducible and faithful $\mathcal{D}_h^{\prime}$-submodule,
\\
\textbullet\; $W_{<k}$ vanishes under the $\mathcal{D}_h^{\prime}$-action.
\\
The 1st conclusion above implies that each $W_{\geq h}(\halpha)$ is an irreducible $\mathcal{R}_{\mathrm{odd}, h}$-module. Since the argument so far is independent of the choice of the multi-index $\halpha$ and since
\[
 W_{\geq h} = \bigoplus_{\halpha \in \{ 0,-1\}^{h-1} } W_{\geq h}(\halpha)
\]
, the 2nd conclusion above implies that the identity element of $\mathcal{D}_h^{\prime}$ is nothing but the distinguished projection $P_{\geq h} : \lgo \twoheadrightarrow W_{\geq h}$. Thus we are done. 

\rightline{$\Box$}

\t
{\bf Proof of Theorem \ref{main assertion2}.} Define the distinguished basis of $L_{g}/\lgo$ as follows;
\[
\{ \pi^{\prime} (\btv^{\hat{j}}_{2l}) \mid 1\leq l \leq g,\; \hat{j} \in \{0,-1 \}^g \}
\]
Consider the following correspondence between the basis vectors;
\[
L_{g}/\lgo \ni \; \pi^{\prime}(\btv^{\hat{j}}_{2l}) \;\longleftrightarrow\; \btw^{\hat{j}^{*}}_{2l-1} \; \in \lgo
\]
, where we set
$\hat{j}^{*} := ({-j_1}\!-\!1,\, {-j_2}\!-\!1,\,\dots,\, {-j_g}\!-\!1)\;\;  \text {for }\;\; \hat{j} =(j_1, j_2, \dots ,\, j_g)$.
\\
Accordingly, define the distinguished subspaces of $L_{g}/\lgo$ such as $V_{\leq h},\, V_{\geq h},\, V_{\geq h}(\halpha)$ in exactly the same fashion as $W_{\leq h},\, W_{\geq h},\, W_{\geq h}(\halpha^{*}) \subset \lgo$. We propose the Assertion A'(h) for $1\leq h \leq g\,$ analogously, whose statement is the same as A(h) except that we replace $\cro$ by $\cre$, \;$\btw^{\hat{j}}_{2l-1}$ by $\btv^{\hat{j}^{*}}_{2l}$, \;$W_{\geq h}(\halpha)$ by $V_{\geq h}(\halpha^{*})$ and so on. In the proof, we should replace $\pi(A_k) \; (1\leq k \leq g-1)$ by $\pi^{\prime}(A_k)$,\; $\pi(B_g)$ by $\pi^{\prime}(B_g^{\dagger})$ and $\pi(T_g)$ by $\pi^{\prime}(T_g^{\dagger})$. Then the same inductive argument as in the proof of Theorem \ref{main assertion1} works word for word to prove A'(h). 

\rightline{$\Box$}

\t
{\bf Proof of Theorem \ref{main thm}.} It follows from Theorem \ref{main assertion1} and Theorem \ref{main assertion2} that the filtration
\[
F \,:\; \{0\} \subset \lgo \subset L_g
\]
is a composition series of $L_g$ as a $\mathcal{R}^{F}$-module. 
Set
\begin{align*}
\Theta &:= \mathrm{Proj}(\mathit{odd}:\hat{0}) \circ \hat{\Psi}_{2g-1,\mathrm{unil}} \circ \mathrm{Proj}(\mathit{odd}:\hat{0}),
\\
\Theta^{\dagger} &:= \mathrm{Proj}(\mathit{odd}:\hat{0}) \circ \hat{\Psi}_{2g,\mathrm{unil}} \circ \mathrm{Proj}(\mathit{odd}:\hat{0}).
\end{align*}
(1.)\, It follows from Proposition \ref{important unil odd} that 
\[
\mathrm{Image}(\Theta) \subset \lgo \subset \mathrm{Ker}(\Theta) 
\]
and that the induced map $\bar{\Theta} : L_g/\lgo \to \lgo$ is non zero.
\\
(2.)\, It follows from Proposition \ref{important unil even} that the composition
\[
\lgo \overset{\Theta^{\dagger}}{\longrightarrow} L_g \twoheadrightarrow L_g/\lgo
\]
is nonzero. 
\\
It is sufficient to show that $L_g$ is an irreducible $\mathcal{R}$-module. Suppose to the contrary that there were a proper non-zero $\mathcal{R}$-submodule $M$, which we may assume to be minimal. If we regard $M$ as a $\mathcal{R}^{F}$-submodule, then one of the following two cases occurs without fail;
\\
Case 1.)\; $M$ coincides with $\lgo$. But if it were the case, $M$ can not be $\mathcal{R}$-submodule because of (1).
\\
Case 2.)\; $M$ projects isomorphically to $L_g/\lgo$. But if it were the case, $M$ can not be $\mathcal{R}$-submodule because of (2).
\\
Therefore neither case can occur, which leads to a contradiction. Thus we are done.

\rightline{$\Box$}

\end{document}